\documentclass[reqno,11pt]{amsart}
\usepackage[latin1]{inputenc}
\usepackage[T1]{fontenc}
\usepackage[ngerman,english]{babel}
\usepackage{latexsym}
\usepackage{amsthm,amssymb,amsmath,amscd}
\usepackage{enumerate,enumitem}
\usepackage{ifthen}
\usepackage[super]{nth}
\usepackage{mathtools}
\usepackage{color}
\usepackage{apptools}
\usepackage{graphicx}
\usepackage{caption}
\usepackage{subcaption}

\newcommand{\Aa}{\mathcal{A}}

\newcommand{\Dd}{\mathcal{D}}

\newcommand{\Bb}{\mathcal{B}}

\newcommand{\Hh}{\mathcal{H}}

\newcommand{\Ss}{\mathcal{S}}

\newcommand{\Ll}{\mathcal{L}}

\newcommand{\Ii}{\mathcal{I}}

\newcommand{\R}{\mathbb{R}}
\newcommand{\N}{\mathbb{N}}

\newcommand{\Z}{\mathbb{Z}}
\newcommand{\C}{\mathbb{C}}

\newcommand{\formt}{\mathbf{t}}

\newcommand{\forma}{\mathbf{a}}

\newcommand{\formb}{\mathbf{b}}

\newcommand{\Rd}{\R^d}

\newcommand{\norm}[1]{\left\lVert#1\right\rVert}
\newcommand{\abs}[1]{\left\lvert#1\right\rvert}

\newcommand{\set}[2]{\left\{#1:#2\right\}}

\newcommand{\scalarprod}[2]{\left(#1,#2\right)}

\newcommand{\ls}{\operatorname{span}}
\newcommand{\sgn}{\operatorname{sgn}}

\newcommand{\dom}{\operatorname{dom}}

\newcommand{\diag}{\operatorname{diag}}
\newcommand{\dist}{\operatorname{dist}}

\newcommand{\Loneloc}{L^1_{\operatorname{loc}}}

\newcommand{\app}{\sigma_{\operatorname{ap}}}

\newcommand{\re}{\operatorname{Re}}
\newcommand{\im}{\operatorname{Im}}
\newcommand{\e}{\operatorname{e}}
\renewcommand{\d}{{\rm d}\hspace{0.1mm}}

\renewcommand{\i}{{\rm{i}}}

\newcommand{\defeq}{\vcentcolon=}
\newcommand{\eqdef}{=\vcentcolon}

\newtheorem{thm}{Theorem}[section]
\newtheorem{cor}[thm]{Corollary}
\newtheorem{lem}[thm]{Lemma}
\newtheorem{prop}[thm]{Proposition}

\theoremstyle{definition}
\newtheorem{defi}[thm]{Definition}
\newtheorem{exple}[thm]{Example}

\theoremstyle{remark}
\newtheorem{rem}[thm]{Remark}

\numberwithin{equation}{section}

\AtAppendix{
	
	\setcounter{thm}{0}
	
}

\setlist{leftmargin=9mm}

\mathtoolsset{showonlyrefs}

\begin{document}
\title{Pseudo Numerical Ranges and Spectral Enclosures}
\author{Borbala Gerhat}
\address{Math.\ Institut, Universit\"at Bern, Sidlerstr.\ 5, 3012 Bern,  Switzerland}
\email{borbala.gerhat@math.unibe.ch}
\author{Christiane Tretter}
\address{Math.\ Institut, Universit\"at Bern, Sidlerstr.\ 5, 3012 Bern,  Switzerland}
\email{tretter@math.unibe.ch}

\begin{abstract}
We introduce the new concepts of pseu\-do numerical range for operator functions and families of sesquilinear forms as well as the pseu\-do block numerical range for $n\!\times\! n$ operator matrix functions. While these notions are new even in the bounded case, we cover operator polynomials with unbounded coefficients, unbounded holomorphic form families of type (a) and associated operator families of type (B). Our main results include spectral inclusion properties of pseudo numerical ranges and pseudo block numerical ranges. For diagonally dominant and off-diagonally dominant operator matrices they allow us to prove
spectral enclosures in terms of the pseudo numerical ranges of Schur complements that no longer require dominance order  $0$ and not even 
$\!<\!1$.\! As an application, we establish a new type of spectral bounds for linearly damped wave equations with possibly unbounded and/or \vspace{-8mm} singular~damping.
\end{abstract}

\dedicatory{To the memory of Professor B.\ Malcolm Brown, a dear colleague and \vspace{-2mm} friend}

%\keywords{Numerical range, spectrum, resolvent estimate, operator function, operator polynomial, damped wave equation}
%\classifications{}

%\date{\today}
\maketitle

\section{Introduction}

Spectral problems depending non-linearly on the eigenvalue parameter arise frequently in applications, see e.g.\! the comprehensive collection in \cite{Betcke-Higham-Mehrmann-Schroeder-Tisseur-2013} or the monograph \cite{Moeller-Pivovarchik-2015}. The dependence ranges from quadratic in problems originating in second order Cauchy problems such as damped wave equations, see e.g.\ 
\cite{Jacob-Trunk-2009}, \cite{Jacob-Tretter-Trunk-Vogt-2018}, to rational as in electromagnetic problems with frequency dependent materials such as photonic crystals, see e.g.\ \cite{MR3543766}, \cite{MR4209764}. In addition, if energy dissipation is present due to damping or lossy materials, then the values of the corresponding operator functions need not be selfadjoint.

While for operator functions $T(\lambda)$, $\lambda\!\in\!\Omega \!\subseteq\! \C$, with unbounded operator values in a Hilbert space $\Hh$ the notion of numerical range $W(T)$ \vspace{-0.5mm} exists,
\begin{equation}
\label{intro-c0}
\hspace*{-4mm}
\begin{aligned}
	W(T)\!:=&\!\set{\lambda\!\in\!\Omega}{0\!\in\! W(T(\lambda))} \\
	    \!=&\set{\lambda\!\in\!\Omega}{\exists\, f\!\in\!\dom T(\lambda), f\!\ne\!0, \ (T(\lambda)f,f)\!=\!0},
\end{aligned}
\vspace{-1mm}	
\end{equation}
a spectral inclusion result $\sigma_{\rm ap}(T) \!\subseteq\! \overline{W(T)} \cap \Omega$ for the approximate point spectrum is lacking. Even in the case of bounded values $T(\lambda)$, spectral inclusion only holds under a certain condition that is not easy to verify. Moreover, spectral inclusion results are even lacking for the most important case of quadratic operator polynomials with unbounded coefficients, one of the most relevant cases for applications.

In the present paper we fill these gaps. To this end, we introduce the novel concept of \emph{pseudo numerical range} of operator functions $T(\lambda)$, $\lambda\!\in\!\Omega \!\subseteq\! \C$, with unbounded \vspace{-1mm} values,
\begin{equation}
\label{eq:intro.def.pseudo.nr}
	W_\Psi(T)\defeq\bigcap\nolimits_{\varepsilon>0}W_\varepsilon(T), \quad W_\varepsilon(T)\defeq\bigcup\nolimits_{B\in L(\Hh) \atop \norm{B}<\varepsilon}W(T+B), \quad \varepsilon>0,
\vspace{-2.5mm}
\end{equation}
and analogously for families of unbounded quadratic forms $\formt(\lambda)$, $\lambda\!\in\!\Omega \!\subseteq\! \C$. The sets $W_\varepsilon(T)$, $\varepsilon>0$, can be shown to have the equivalent form
\begin{equation}
\label{eq:intro.W.psi.epsilon} 
  W_\varepsilon(T) =\set{\lambda\in\Omega}{\exists ~f\in\dom T(\lambda), ~\norm{f}=1, ~\abs{(T(\lambda)f,f)}<\varepsilon};
\vspace{-1mm}	
\end{equation}
hence they coincide with the so-called $\varepsilon$-pseudo numerical range first considered in \cite{Engstroem-Torshage-2017}.
As a consequence, the pseudo numerical range $W_\Psi(T)$ can equivalently be described \vspace{-0.5mm} 
as
\begin{equation}
	\label{eq:intro.pseudo.num.id} 
			W_\Psi(T)  \!=\!\big\{\lambda\!\in\!\Omega:0\!\in\!\overline{W(T(\lambda))}\big\} \eqdef W_{\Psi,0}(T).
\vspace{-1mm}
\end{equation}
One could be tempted to think that the condition $0\!\in\!\overline{W(T(\lambda))}$ in $W_{\Psi,0}(T)$ is equivalent to 
$\lambda\!\notin\! \overline{W(T)}$, but this is neither true for operator functions with bounded values, 
as already noted in \cite{Wagenhofer-PhD-2007}, nor for non-monic linear operator pencils for which the set $W_{\Psi,0}(T)$ 
was used recently in \cite{Boegli-Marletta-2019}. 

One of the crucial properties of the pseudo numerical range is that, \emph{without any assumptions on the operator \vspace{-1mm} family},
\begin{equation}
	\app(T)\subseteq W_\Psi(T),
\vspace{-1mm}	
\end{equation}
see Theorem \ref{thm:spec.incl.pseudo.num.ran},  and that the norm of the resolvent of $T$ can be \vspace{-1mm} estimated~by
\begin{equation}
	\norm{T(\lambda)^{-1}}\le \varepsilon^{-1},	\quad \lambda\in\rho(T)\setminus W_\varepsilon(T) \subseteq \rho(T)\setminus W_\Psi(T).
\vspace{-1mm}	
\end{equation}
Not only from the analytical point of view, but also from a computational perspective, the pseudo numerical range seems to be more convenient since it is much easier to determine whether a number is small rather than zero.

Like the numerical range of an operator function, but in contrast to the numerical range 
or essential numerical range of an operator \cite{Kato-1995}, \cite{MR4083777}, \cite{HT-2022}, the pseudo numerical range need~not be convex. An exception is the trivial case of a monic linear operator pencil $T(\lambda)\!=\!A\!-\!\lambda I$, $\lambda \!\in\! \C$, where the pseudo numerical range is simply the closure of the numerical range, $W_\Psi(T)\!=\!\overline{W(T)}\!=\!\overline{W(A)}$. 
In general, we only have the obvious enclosure $W(T) \subseteq W_\Psi(T)$. Neither the interiors nor the closures in $\Omega$ of $W_\Psi(T)$ and $W(T)$ need to coincide  and 
there is also no inclusion either way between $W_\Psi(T)$ or its closure $\overline{W_\Psi(T)}\cap \Omega$ in $\Omega$ and the closure $\overline{W(T)} \cap \Omega$ of $W(T)$ in~$\Omega$; we give various counter-examples to illustrate these~effects.

In our first main result we use the pseudo numerical range of holomorphic form families $\formt (\lambda)$, $\lambda \in \Omega$, of type (a) to prove the spectral inclusion  for the associated holomorphic operator functions $T(\lambda)$, $\lambda \in \Omega$, of type (B) of m-sectorial operators $T(\lambda)$. More precisely, we show that if there exist $k\in\N_0$, $\mu\in\Omega$ and a core $\Dd$ of $\formt(\mu)$ \vspace{-1mm} with 
	\begin{equation}
    \label{eq:intro:ass.pseudo.dense.hol.fam}
		0 \notin \overline{W\big(\formt^{(k)}(\mu)\big|_\Dd\big)},
	\vspace{-1mm}	
	\end{equation}
	then $\sigma(T) \subseteq W_\Psi(\formt)=\overline{W(\formt)} \cap \Omega$ and, if
	in addition, the operator family $T$ has constant domain, \vspace{-1.5mm} then
	\begin{equation}
	\label{intro-c1}
		\sigma(T) \!\subseteq%\subset
		\, W_\Psi(T)=\overline{W(T)}\cap \Omega,
	\vspace{-2mm}	
	\end{equation}
see Theorem \ref{thm:pseudo.dense.hol.fam}. Note that, due to \eqref{eq:intro.pseudo.num.id}, condition \eqref{eq:intro:ass.pseudo.dense.hol.fam} for $k\!=\!0$, \vspace{-0.5mm} i.e.\ 
$0 \notin \overline{W\big(\formt(\mu)\big|_\Dd\big)}$ for some $\mu\in \C$, is equivalent to $W_\Psi(T)\ne \Omega$.

For operator polynomials $T(\lambda) = \sum_{k=0}^n \lambda^k A_k$ with domain $\dom T(\lambda)=\bigcap_{k=0}^n \dom A_k$, $\lambda \in \C$, 
we prove that, if $0\notin\overline{W(A_n)}$, \vspace{-0.4mm} then
	\begin{equation}
	\label{eq:introo.poly.pseudo.num.op}
		\app(T) \subseteq W_\Psi(T)\subseteq\overline{W(T)}\cap \Omega,
	\vspace{-2mm}	
	\end{equation}
see Proposition \ref{prop:poly.pseudo.num.op}. The inclusion \eqref{intro-c1} follows if, in addition, $\sigma(T(\lambda)) \!\subseteq\! \overline{W(T(\lambda))}$, $\lambda\!\in\!\C$, which is a weaker condition than m-sectoriality of all %\vspace{-4mm}
$T(\lambda)$.

The second new concept we introduce in this paper is the \emph{pseudo block numerical range} of operator functions $\Ll(\lambda)$, $\lambda\in\Omega$, that possess an operator matrix representation with respect to a 
decomposition $\Hh=\Hh_1\oplus\cdots \oplus\Hh_n$, $n\in\N$, of the given Hilbert space $\Hh$. This means \vspace{-1mm} that 
\begin{equation}
\label{eq:intro.op.matrix.fam}
	\Ll(\lambda)=\big( L_{ij} (\lambda) \big)_{i,j=1}^n, \quad  \dom\Ll(\lambda)=\bigoplus\nolimits_{\!j=1}^{\!n} \ \bigcap\nolimits_{i=1}^n \dom L_{ij}(\lambda),
\vspace{-1mm}	
\end{equation}
with operator functions $L_{ij}(\lambda)$, $\lambda\!\in\!\Omega %\C
$, of densely defined and closable linear operators from $\Hh_j$ to $\Hh_i$, $i$, $j=1,\dots, n$.

Extending earlier concepts we first define the \emph{block numerical range}\vspace{-1mm}  of~$\Ll$~as
\begin{equation}
 	W^{n}(\Ll) \defeq  \bigcup\nolimits_{(f_i)\in\dom\Ll(\lambda) \atop \|f_i\|\!=\!1} 
 	\sigma_p \big( \Ll(\lambda)_{(f_i)}  \big), \quad \Ll(\lambda)_{(f_i)} \!\defeq\! \left( \Ll_{ij}(\lambda) f_j, f_i \right) \!\in\! \C^{n\times n}\!;
	\vspace{-1mm}
\end{equation}
for bounded values $\Ll(\lambda)$ see \cite{MR3302436} and \cite{Tretter-2010} for $n=2$, for unbounded operator matrices $\Ll(\lambda) \!=\! {\mathcal A} - \lambda I_\Hh$ see \cite{Rasulov-Tretter-2018}. Then we introduce the \emph{pseu\-do block numerical range} of $\Ll$ as
\begin{equation}
 	W^n_\Psi(\Ll)\defeq\bigcap\nolimits_{\varepsilon>0}W_\varepsilon^n(\Ll), \qquad 
 	W_\varepsilon^n(\Ll)\defeq \bigcup\nolimits_{\Bb\in L(\Hh)\atop \norm{\Bb}<\varepsilon} W^n(\Ll+\Bb), \quad \varepsilon>0.
	\vspace{-2mm}			
\end{equation}

For $n\!=\!1$ both block numerical range and pseudo block numerical range coincide with the numerical range and pseudo numerical range of $\Ll$, respectively. For $n\!>\!1$, the trivial inclusion $W^n(\Ll) \subseteq W^n_\Psi(\Ll)$ and the characterisation %characterization 
\eqref{intro-c0}, \vspace{-1mm} i.e. 
\begin{equation}
	\label{eq:intro.qnr.equiv}
		W^n(\Ll) =\big\{ \lambda \in \Omega: 0\in W^n(\Ll(\lambda)) \big\}, \quad n\in \N,
\vspace{-1mm}		
\end{equation} 
and a resolvent norm \vspace{-1mm} estimate
\begin{equation}
	\norm{\Ll(\lambda)^{-1}}\!\le\! \varepsilon^{-1}, 
	\quad 
	\lambda\!\in\!\rho(\Ll)\setminus W_\varepsilon^{n}(\Ll) \subseteq\! \rho(\Ll)\setminus W_\Psi^n(\Ll), \quad n\!\in\!\N,
\end{equation}
see Theorem \ref{thm:spec.incl.pseudo.qnr} for both, 
continue to hold, but otherwise not much carries over from the case $n\!=\!1$. The first difference is that, for the simplest case $\Ll(\lambda)=\Aa-\lambda I_\Hh$, $\lambda\in \C$, we may have $W_\Psi^n(\Ll)\neq\overline{W^n(\Ll)}$ for $n\!>1\!$, see Example \ref{ex:jordan}.

More importantly, for $n\!>\!1$ the relation \eqref{eq:intro.pseudo.num.id} need not hold for the pseudo block numerical range; here we only have the \vspace{-1mm} inclusion
\begin{equation}
  W_\Psi^{n}(\Ll) \supseteq \set{\lambda\!\in\!\Omega}{0\in\overline{W^{n}(\Ll(\lambda))}} \eqdef W_{\Psi,0}^{n}(\Ll), \quad n\in \N,
\vspace{-1mm}
\end{equation}
see Proposition \ref{prop:nested.def.pseudo.qnr}. Therein we also assess two other candidates $W_{\Psi,i}^{n}(\Ll) \!=\! \bigcap_{\varepsilon>0} W_{\varepsilon,i}^{n}(\Ll)$, $i\!=\!1,2$, for the pseudo block numerical range for which $W_{\varepsilon,1}^{n}(\Ll)$ is defined by the scalar condition $\det  \Ll(\lambda)_{(f_i)} \!<\! \varepsilon$ and $W_{\varepsilon,2}^{n}(\Ll)$ by restricting to \emph{diagonal} perturbations $\Bb\in L(\Hh)$ with $\norm{\Bb}<\varepsilon$. In fact, we show that
\begin{equation}
	\label{eq:introo.pbnri}
		W^n(\Ll) \subseteq W^{n}_{\Psi,1}(\Ll)\subseteq  W_{\Psi,0}^{n}(\Ll)\subseteq W^{n}_{\Psi,2}(\Ll)\subseteq W^{n}_\Psi(\Ll),
	\end{equation}
and that, like the pseudo numerical range, the pseudo block numerical range $W^{n}_\Psi(\Ll)$ has the spectral inclusion property, i.e. 
\begin{equation}
	\app(T) \subseteq W^{n}_\Psi(\Ll) \subseteq W_\Psi(T), \quad n\in \N,
\end{equation}
but, in general, none of the subsets of $W^{n}_\Psi(\Ll)$ in \eqref{eq:introo.pbnri} is large enough to contain $\app(T)$, see Example \ref{ex:jordan}.

Our second main result concerns the most important case $n\!=\!2$, the so-called \emph{quadratic numerical range} and \emph{pseudo quadratic numerical range}. Here we prove a novel type of spectral inclusion for diagonally dominant and off-diagonally dominant $\Ll(\lambda) \!=\! (L_{ij}(\lambda))_{i,j=1}^2$ in terms of the pseudo numerical ranges of the Schur complements $S_1$, $S_2$ and, further, the pseudo quadratic numerical range of $\Ll$,   
\begin{equation}
	\label{eq:intro.mat.fam.schur.app.incl}
		\app(\Ll)\setminus(\sigma(L_{11})\cup\sigma(L_{22})) \subseteq W_\Psi(S_1)\cup W_\Psi(S_2)	\subseteq W^2_\Psi (\Ll),
	\vspace{-1mm}	
	\end{equation}
see Theorem \ref{thm:mat.spec.incl.schur.app}, where $S_1(\lambda)\!=\!L_{11}(\lambda) \!-\! L_{12}(\lambda) L_{22}(\lambda)^{-1} L_{21}(\lambda)$, $\lambda \!\in\! \rho(L_{22})$, and similarly for $S_2$ with the indices $1$ and $2$ reversed. For symmetric and anti-symmetric corners, i.e.\ $L_{21}(\lambda) \subseteq%\subset 
\pm L_{12}(\lambda)^*$, $\lambda\!\in\!\Omega$, we even show that
\begin{equation}
 \app(\Ll)%)
 \!\subseteq\! W_\Psi(S_1)\cup W_\Psi(L_{22}),
\end{equation}
if $L_{11}(\lambda)$ is accretive, $\mp L_{22}(\lambda)$ is m-sectorial and $\dom L_{22}(\lambda) \!\subseteq%\subset
\! \dom L_{12}(\lambda)$, see Theo\-rem \ref{thm:spec.incl.def.indef}/Corollary \ref{cor:spec.incl.def.indef}, and similarly for the Schur complement~$S_2$.

As an interesting consequence, we are able to establish spectral separation and inclusion theorems for unbounded $2\!\times\!2$ operator matrices $\Aa=(A_{ij})_{i,j=1}^2$ with 'separated' diagonal entries; here 'separated' means that the numerical ranges of $A_{11}$ and $A_{22}$ lie in half-planes and/or sectors in the right and left half-plane $\C_+$  and $\C_-$, respectively, separated by a vertical strip $S\!\defeq\!\{z\!\in\!\C:\delta \!<%\le
\! \re z \!<%\le
\! \alpha\}$ with $\delta\!<\!0\!<\!\alpha$ around~$\i\R$. More precisely, \emph{without} any bounds on the order of diagonal dominance or off-diagonal dominance we show that, if $\varphi$, $\psi \!\in\! [0,\frac \pi 2]$ are the semi-angles of $A_{11}$ and $A_{22}$ and $\tau\!:=\! \max\{\varphi,\psi\}$, then
\begin{equation}
 	\app(\Aa) \subseteq ( - \!\Sigma_\tau \cup \Sigma_\tau ) \setminus S \eqdef \Sigma, \quad
 	\Sigma_\tau \defeq \{z\!\in\!\C: |\arg z| \le \tau \},
\end{equation}
and $\sigma(\Aa) \subseteq %\subset 
\Sigma$ if $\rho(\Aa) \cap (\C \setminus \Sigma) \!\ne\! \emptyset$,  see Theorem \ref{thm:spec.incl.BB*}. This result is a great step ahead compared to the earlier result \cite[Thm.\ 5.2]{Tretter-2009} where the dominance order had to be restricted to $0$. 

Moreover, even to ensure the condition $\rho(\Aa) \cap (\C \setminus \Sigma) \!\ne\! \emptyset$ for the enclosure of the entire spectrum $\sigma(\Aa)$ in Theorem \ref{thm:spec.incl.BB*}, we do not have to restrict~the dominance order as usual for perturbation arguments. Our new weak conditions involve only products of the columnwise relative bounds $\delta_1$ in the~first and $\delta_2$ in the second column, see Proposition \ref{thm:full.spec.incl.BB*}; in particular, either $\delta_1\!=\!0$ or  $\delta_2\!=\!0$ guarantees $\rho(\Aa) \cap (\C \setminus \Sigma) \!\ne\! \emptyset$ in Theorem \ref{thm:spec.incl.BB*} %\vspace{-4mm} 
and~hence~$\app(\Aa) \!\subseteq\! \Sigma$.

As an application of our results, we consider abstract quadratic operator polynomials $T(\lambda)$, $\lambda\!\in\!\C$,  induced by forms $\formt(\lambda)\!=\!\formt_0\!+\!2\lambda\forma\!+\!\lambda^2$ with $\dom\formt(\lambda)=\dom\formt_0$, $\lambda\in\C$, as they arise e.g.\ from linearly damped wave equations
\begin{equation}
\label{intro-c2}
	u_{tt}(x,t)+2a(x)u_t(x,t)=\left(\Delta_x-q(x)\right)u(x,t), \quad x\in\Rd, \quad t>0,
\end{equation}
where the non-negative potential $q$ and damping $a$ may be singular and/or unbounded, cf.\ \cite{Freitas-Siegl-Tretter-2018, Jacob-Tretter-Trunk-Vogt-2018,Jacob-Trunk-2007,Jacob-Trunk-2009} where also accretive damping was considered, and for which it is well-known that the spectrum is symmetric with respect to $\R$ and con\-fined to the closed left half-plane. 

Here we use a finely tuned assumption on the 'unboundedness' of $\forma$ with respect to $\formt_0$, namely \emph{$p$-subordinacy} for $p\!\in\![0,1)$, comp.\ \cite[\S\,5.1]{Markus-1988} or \cite[Sect.~3]{Tretter-Wyss-2014} for the operator case. More precisely, if 
$\formt_0\!\ge\!\kappa_0\!\ge\!0$, $\forma\!\ge\!\alpha_0\!\ge\!0$ with $\dom\formt_0\!\subseteq\!\dom\forma$ and there exist~$p\!\in\![0,1)$ and $C_{p}\!>\!0$ \vspace{-1mm}with
	\begin{equation}
	\label{eq:intro.pencil.subordinate}
	\forma[f]\le C_{p}\big(\formt_0[f]\big)^p \big(\norm{f}^2\big)^{1-p}, \quad f\in\dom\formt_0,
	\end{equation}
we use the enclosure $\sigma(T) \!\subseteq\! W_\Psi(T) \!=\! W_\Psi(\formt) \!=\! \overline{W(\formt)}$ to prove that the non-real spectrum of $T$ satisfies the \vspace{-1mm} bounds
\begin{align*}
    \sigma(T)\setminus \R \! \subseteq  \!
			\Big\{ z\!\in\!\C: \, |z| \ge \sqrt{\kappa_0}, \, & \, \re z\le-\alpha_0,  %\max\{\alpha_0, \sqrt{\kappa_0}\}
			\\[-3mm] &\, \abs{\im z}^2\!\!\ge\! \max\!\big\{0,C_{p}^{-\frac{1}{p}}\!\abs{\re z}^\frac{1}{p}\!\!-\!\abs{\re z}^2\big\}\Big\}
			%\hspace{-10mm}
			\\[-7mm]
\end{align*}%
%\bgcomm{cases were incorrect; describing all cases correctly would be very long here}
and the real spectrum $ \sigma(T)\cap \R \subset [-\infty,0]$ is either empty or it is confined to one bounded interval, to one unbounded interval or to the disjoint union of a bounded and an unbounded interval%one bounded interval if  \linebreak $p \!<\!\frac 12$, empty, $\{0\}$ or one unbounded interval if $p \!<\!\frac 12$ and an unbounded interval with $\{0\}$ possibly joined if $p \!>\!\frac 12$
, see Theorem \ref{thm:pencil.spec.incl} and Figure \ref{fig:dwe.spec.incl}. Moreover, we describe both 
the thresholds for the transitions between these cases and the enclosures for $ \sigma(T)\cap \R$ precisely in terms of $p$, $C_p$, $\kappa$ and $\kappa_0$. As a concrete example, we consider the damped wave equation \eqref{intro-c2} \vspace{-2mm}with
\begin{equation}
	\label{eq:intro.dwe.pot.damp.inequ}
	a(x)\!\le\!\sum_{j=1}^n\abs{x\!-\!x_j}^{-t}\!+\! u(x) \!+\! v(x), \ \ v(x) \!\le\! c_1 q(x)^r\!+c_2 \,\ \mbox{ for almost all $x\!\in\! \R^d$},
\vspace{-2mm}
\end{equation}
where $n\!\in\!\N_{0}$, $x_j\!\in\!\R^d$ for $j\!=\!1,\dotsc,n$, $u \!\in\! L^s (\Rd)$ with $s\!>\!\frac d2$, $v\!\in\!\Loneloc(\Rd)$, $t\!\in\![0,2)$,  $c_1$, $c_2\!\ge\!0$ and $r\!\in\![0,1)$. For the special case  $q(x)\!=\!\abs{x}^2$, $a(x)\!=\!\abs{x}^{k}$,  $x\!\in\!\Rd$, with $k \!\in\![0,2)$, the new spectral enclosure in 
\vspace{-1mm}Theorem~\ref{thm:pencil.spec.incl}~yields 
\begin{equation}
	\label{eq:intro.dwe.comp.incl}
		\sigma(T) \setminus\R\subseteq\Big\{z\!\in\!\C:\re z\!\le\!0, \, \abs{z}\!\ge\! \sqrt{d}, \, |\im z| \!\ge\! 
		\sqrt{\max\{0,\abs{\re z}^{\!\frac{2}{k} %{s}
		}\!\!-\!\abs{\re z}^2\}}\Big\}
		\vspace{-2mm}
	\end{equation}
	\vspace{-1mm}and, with $t_0=\max\big\{ \big( k(2-k) \big)^{-\frac 1{k-1}},d\big\}$,
		\begin{align*}
		\sigma(T) \cap \R 
		\begin{cases} 
			= \emptyset & \mbox { if } k\!\in\![0,1), \\
			\subseteq(-\infty,-\sqrt{d}] & \mbox { if }k=1, \\[-1mm]
			\subseteq \!\Big(\!\!-\!\infty,-\sqrt{t_0}^{k}\!+\!\sqrt{t_0^{k}\!-\!t_0 } \,\Big]  & \mbox { if } k\!\in\! (1,2).
		\end{cases}
%\\[-10mm]
\end{align*}

The paper is organised as follows. In Section \ref{sec:pseudo.num} we introduce the pseudo numerical range of operator functions and form functions and study the relation  of $W_\Psi(T)$ and $\overline{W(T)}\cap \Omega$. In Section \ref{subsec:pseudo.nr.spec.encl} we establish spectral inclusion results in terms of the pseudo numerical range. In Section \ref{sec:op.mat.fam}  we define the block numerical range $W^n(\Ll)$ and pseudo block %block pseudo 
numerical range $W^{n}_\Psi(\Ll)$ of unbounded $n\!\times\! n$ operator matrix functions $\Ll$, investigate the differences to the special case $n\!=\!1$ of the pseudo numerical range $W_\Psi^1(\Ll)\!=\!W_\Psi(\Ll)$ and prove corresponding spectral inclusion theorems. In Section \ref{sec:schur.app.encl} we establish new
enclosures of the approximate point spectrum of $2\!\times\! 2$ operator matrix functions by means of the pseudo numerical ranges of their Schur complements. In Section \ref{sec:BB*} we apply them to prove %spectral 
%\bgcomm{removed repetition of "spectral"}
spectral bounds for diagonally dominant and off-diagonally dominant operator matrices with symmetric or anti-symmetric corners without restriction on the dominance order. 
Finally, in Section \ref{sec:dwe}, we apply our results to linearly damped wave equations with possibly unbounded and/or singular damping and potential.

Throughout this paper, $\Hh$ and $\Hh_i$, $i\!=\!1,\dots,n$, denote Hilbert spaces,~$L(\Hh)$ denotes the space of bounded linear operators on $\Hh$ and $\Omega\!\subseteq\!\C$ is a~domain. 

\section{The pseu\-do numerical range of operator functions and form functions}
\label{sec:pseudo.num}

In this section, we introduce the new notion of pseu\-do numerical range for operator functions $\set{T(\lambda)}{\lambda\in\Omega}$ and form functions $\set{\formt(\lambda)}{\lambda\in\Omega}$, respectively,
briefly denoted by $T$ and $\formt$ if no confusion about $\Omega$ can arise. While the values $T(\lambda)$ and $\formt(\lambda)$ may be bounded/unbounded linear operators and sesquilinear forms in a Hilbert space $\Hh$, the notion of pseudo numerical range is new also in the bounded case.

The \emph{numerical range} of $T$ and $\formt$, respectively, are defined as
\begin{alignat*}{2}
	W(T)\!&=\!\set{\lambda\!\in\!\Omega\!}{\!0\!\in\! W(T(\lambda))} 
	&&=\!\set{\lambda\!\in\!\Omega\!}{\!\exists\, f\!\in\!\dom T(\lambda), f\!\ne\!0, (T(\lambda)f,f)\!=\!0}, 
	\\
	W(\formt)\!&=\!\set{\lambda\!\in\!\Omega\!}{\!0\!\in\! W(\formt(\lambda))}
	\!&&=\!\set{\lambda\!\in\!\Omega\!}{\!\exists\, f\!\in\!\dom \formt(\lambda), \,f\!\ne\!0, \,\formt(\lambda)[f]\!=\!0},
\end{alignat*}
comp.\ \cite[\S\,26]{Markus-1988}. In the simplest case of a monic linear operator polynomial $T(\lambda) = T_0 - \lambda I_\Hh$, $\lambda \in \C$, this notion coincides with the numerical range $W(T_0)$ of the linear operator $T_0$, and analogously for forms;
note that the latter is also denoted by $\Theta(T_0)$, e.g.\ in \cite[Sect.~V.3.2]{Kato-1995}.

The following new concept of pseudo numerical range employs the notion of $\varepsilon$-pseudo 
numerical range $W_\varepsilon(T)$, $\varepsilon>0$, introduced in \cite[Def.\ 4.1]{Engstroem-Torshage-2017}; the equivalent original definition therein, see \eqref{eq:W.psi.epsilon} below, was designed to obtain computable enclosures for spectra of rational operator functions. 

\begin{defi}
\label{def:pseudo-nr}
	We introduce the \emph{pseu\-do numerical range} of an operator function $T$ and a form function $\formt$, respectively, as
	\begin{equation}
		\begin{aligned}
			W_\Psi(T) & \defeq\bigcap_{\varepsilon>0}W_\varepsilon(T),  & \quad W_\Psi(\formt) & \defeq\bigcap_{\varepsilon>0}W_\varepsilon(\formt), 
			\\[-7mm] 
		\end{aligned}
	\end{equation}
	where 
	\begin{equation}
			W_\varepsilon(T) \defeq \bigcup_{B \in L(\Hh), \norm{B}<\varepsilon}W(T+B), \quad W_\varepsilon(\formt) 
			\defeq\bigcup_{\norm{\formb}<\varepsilon}W(\formt+\formb), \quad \varepsilon>0;
	\end{equation}
	here $\norm{\formb}=\sup_{\norm{f}=\norm{g}=1}\abs{\formb[f,g]}$ for a bounded sesquilinear form $\formb$ in $\Hh$.
\end{defi}

Clearly, for monic linear operator polynomials $T(\lambda) = A 
- \lambda I_\Hh$, $\lambda \in \C$, the pseu\-do numerical range is nothing but the closure of the classical numerical range 
$\overline{W(A)}$ of the linear operator $A$, 
and analogously for forms.

The pseudo numerical range of operator or form functions,  is, like their numerical ranges, in general neither convex nor connected, and, even for families of bounded operators or forms, it may be unbounded.

\begin{rem}
	\begin{enumerate}
		\item The following enclosures may be proper, see Example \ref{ex:pseudo.num.spec.incl},
		\begin{equation}
			W(T)\subseteq W_\Psi(T), \qquad W(\formt)\subseteq W_\Psi(\formt).
		\end{equation}
		\item In general, the pseu\-do numerical range need neither be open nor closed in $\Omega$ equipped with the relative topology, 
		see Examples \ref{ex:pseudo.num.spec.incl} (i) and \ref{ex:ODE}, respectively.
		\item Neither the closures nor the interiors with respect to the relative topology on $\Omega$ of the pseudo numerical range and the numerical range need to coincide, see Example \ref{ex:pseudo.num.spec.incl} (i) and (ii).
	\end{enumerate}
\end{rem}

The following alternative characterisation of the pseudo numerical range will be frequently used in the sequel.

\begin{prop}
\label{prop:pseudo.num}
	For every $\varepsilon>0$,
	\begin{align}
	\label{eq:W.psi.epsilon} 
		W_\varepsilon(T) & =\set{\lambda\in\Omega}{\exists ~f\in\dom T(\lambda), ~\norm{f}=1, ~\abs{(T(\lambda)f,f)}<\varepsilon}, \\
		W_\varepsilon(\formt) & =\set{\lambda\in\Omega}{\exists ~f\in\dom \formt(\lambda), ~\norm{f}=1, ~\abs{\formt(\lambda)[f]}<\varepsilon},
	\end{align}	
	and, consequently,
	\begin{align}
	\label{eq:pseudo.num.id} 
		\hspace{-7mm}
		W_\Psi(T) & \!=\!\set{\lambda\!\in\!\Omega}{0\!\in\!\overline{W(T(\lambda))}},  \ \  
		W_\Psi(\formt) \!=\!\set{\lambda\!\in\!\Omega}{0\!\in\!\overline{W(\formt(\lambda))}}.
	\end{align}
\end{prop}

\begin{proof}
	We show the claim for $W_\varepsilon(T)$; then the claim for $W_\Psi(T)$ is obvious by Definition \ref{def:pseudo-nr}. The proof for $W_\varepsilon(\formt)$ and $W_\Psi(\formt)$ is analogous.

	Let $\varepsilon>0$ be arbitrary and $\lambda\in W_\varepsilon(T)$. There exists a bounded operator $B$ in $\Hh$ with $\norm{B}<\varepsilon$ such that $\lambda\in W(T+B)$, i.e.\
	\begin{equation}
		\scalarprod{T(\lambda)f}{f}=-(Bf,f), \quad f\in\dom T(\lambda), \quad \norm{f}=1.
	\end{equation}
	Hence, clearly, $\abs{\scalarprod{T(\lambda)f}{f}}\le\norm{B}<\varepsilon$, thus $\lambda$ is an element of the right hand side of \eqref{eq:W.psi.epsilon}. 
	
	Conversely, let $\lambda\in\Omega$ such that there exists $f\in\dom T(\lambda)$, $\norm{f}=1$, with $\abs{\scalarprod{T(\lambda)f}{f}}<\varepsilon$. Setting $B\defeq-\scalarprod{T(\lambda)f}{f}I$, this gives $\lambda\in W(T+B)$ and $\norm{B}=\abs{\scalarprod{T(\lambda)f}{f}}<\varepsilon$, hence $\lambda\in W_\varepsilon(T)$.
\end{proof}

The following properties of the pseudo numerical range with respect to closures, form representations and Friedrichs extensions are immediate consequences of its alternative description \eqref{eq:pseudo.num.id}.

Here an operator $A$ or a form $\forma$ is called \emph{sectorial} if its numerical range lies in a sector $\{z\in\C: |\arg(z-\gamma)| \le \vartheta\}$ for some $\gamma \in\R$ and $\vartheta \in [0, \frac \pi 2)$, see \cite[Sect.\ V.3.10, VI.1.2]{Kato-1995}; if, in addition, $\rho(A) \cap \{z \in \C :|\arg(z-\gamma)| > \vartheta \} \neq \emptyset$, then $A$ is called m-sectorial.

\begin{cor}		
\label{cor:pseudo.num}
	\begin{enumerate}
		\item If the family $T$ or $\formt$, respectively, consists of closable operators or forms \textnormal
		{(}and $\overline{T}$ or $\overline{\formt}$ denotes the family of closures\textnormal{)}, then
		\begin{equation}
			W_\Psi(T)=W_\Psi(\overline{T}), \qquad W_\Psi(\formt)=W_\Psi(\overline{\formt}).
		\end{equation}
		\item If the family $\formt$ consists of densely defined closed sectorial forms and $T$ denotes the family of associated m-sectorial operators, then
		\begin{equation}
			W_\Psi(\formt)=W_\Psi(T).
		\end{equation}
		\item If the family $T$ consists of densely defined sectorial operators and $T_F$ denotes the family of corresponding Friedrichs extensions then
		\begin{equation}
			W_\Psi(T)=W_\Psi(T_F). 
		\end{equation}
	\end{enumerate}
\end{cor}
	
\begin{proof}	
	(i) The equalities follow from Proposition \ref{prop:pseudo.num} and from the fact that $\overline{W(T(\lambda))}=\overline{W(\overline{T(\lambda)})}$ and  $\overline{W(\formt(\lambda))}=\overline{W(\overline{\formt}(\lambda))}$ for $\lambda\in\Omega$, see \cite[Prob.\ V.3.7, Thm.\ VI.1.18]{Kato-1995}.
	
	(ii) The equality follows from Proposition \ref{prop:pseudo.num} and the identity $\overline{W(\formt(\lambda))}=\overline{W(T(\lambda))}$ for $\lambda\in\Omega$, see \cite[Cor.\ VI.2.3]{Kato-1995}.
	
	(iii) The claim is a consequence of (i) and (ii).
\end{proof}

The alternative characterisation \eqref{eq:pseudo.num.id} might suggest that there is a relation between the pseudo numerical range $W_\Psi(T)$ and the closure $\overline{W(T)}\cap \Omega $ of the numerical range $W(T)$ in $\Omega$. However, in general, there is no inclusion either way between them, see e.g.\ Example \ref{ex:pseudo.num.spec.incl} where $W_\Psi(T)\not\subseteq\overline{W(T)}\cap \Omega $ and Example \ref{ex:ODE} where $\overline{W(T)} \cap \Omega \not\subseteq W_\Psi(T)$. 

In fact, it was already noted in \cite[Prop.\ 2.9]{Wagenhofer-PhD-2007}, for continuous functions of \emph{bounded} operators and for the more general case of block numerical ranges, that, for \vspace{-1mm} $\lambda \in \Omega$, 
\[
  \lambda \in \overline{W(T)} \implies 0 \in \overline{W(T(\lambda))};
\]  
the converse holds only under additional assumptions. More precisely, for families of bounded linear operators however, the following is known.

\begin{thm}{\cite[Prop.\ 2.9, Prop.\ 2.12, Thm.\ 2.14]{Wagenhofer-PhD-2007}}
\label{thm:bdd.pseudo.num.ran}
	\begin{enumerate}
		\item If $\,T$ is a $($norm-$)$continuous family of bounded linear operators,  \vspace{-1mm} then
		\begin{equation}
			\overline{W(T)} \cap \Omega \subseteq W_\Psi(T).
		\end{equation}
		\item If $\,T$ is a holomorphic family of bounded linear operators and there exist $k\in\N_0$ and $\mu\in\Omega$  \vspace{-1mm} with
		\begin{equation}
		\label{eq:ass.markus.bdd}
			0\notin\overline{W(T^{(k)}(\mu))},
					\vspace{-2mm}
		\end{equation}
		then
		\begin{equation}
			\sigma(T) \subseteq	\overline{W(T)} \cap \Omega  =W_\Psi(T).
		\end{equation}
	\end{enumerate}
\end{thm}

The following simple example from \cite[Ex.\ 2.11]{Wagenhofer-PhD-2007}, which is easily adapted to the unbounded case, shows that condition \eqref{eq:ass.markus.bdd} is essential both for the equality $\overline{W(T)} \cap \Omega =W_\Psi(T)$ and for the spectral inclusion $\sigma(T) \subseteq \overline{W(T)} \cap \Omega $.

\begin{exple}
\label{ex:deriv.cond}
	Let $f:\Omega \to \C$ be holomorphic, $f \not\equiv 0$, $A$ a bounded or un\-bounded linear operator in $\Hh$ with $0\in \sigma(A)$, 
	$0 \in \overline{W(A)}\setminus W(A)$ and~consider
	\[
		T(\lambda) := f(\lambda) A, \quad \dom T(\lambda) \defeq \dom A, \quad \lambda \in \Omega.
	\]
	Then \eqref{eq:ass.markus.bdd} is violated because, for any $k\in \N_0$ and $\mu \in \Omega$, we have $T^{(k)}(\mu) = f^{(k)}(\mu) A$ with $\dom T^{(k)}(\lambda) = \dom A$, $\lambda \in \Omega$, and so $0\!\in\! \overline{W(T^{(k)}(\mu))}$
	since $0 \!\in\! \overline{W(A)}$. Further, it is easy to see~that 
	\[
		\sigma(T)=\Omega, \quad W(T)= \overline{W(T)}\cap \Omega =\{ z \in \Omega: f(z)=0\} \ne \Omega, \quad W_\Psi(T)=\Omega.
	\]
	Thus neither $\overline{W(T)} \cap \Omega  =W_\Psi(T)$ nor the spectral inclusion $\sigma(T) \subseteq
	\overline{W(T)}\cap \Omega $ hold, while $\sigma(T) = W_\Psi(T)$.
\end{exple}

In the sequel we generalise Theorem \ref{thm:bdd.pseudo.num.ran} (i) and (ii) to families of unbounded operators and/or forms, including operator polynomials and sectorial families with constant form domain. In the remaining part of this section, we study the relation between $W_\Psi(T)$ and $\overline{W(T)}\cap\Omega$; results containing spectral enclosures may be found in Section \ref{subsec:pseudo.nr.spec.encl}.

\begin{prop}
\label{prop:poly.pseudo.num.op}
	Let $T$ be an operator polynomial in $\Hh$ of degree $n\in\N$  with $($possibly unbounded$)$ coefficients 
	\vspace{-1mm} $A_k:\Hh\supseteq\dom A_k\to\Hh$, i.e.
	\begin{equation}
		T(\lambda)\defeq\sum_{k=0}^n\lambda^k A_k, \quad \dom T(\lambda)\defeq\displaystyle\bigcap_{k=0}^n\dom A_k, \quad \lambda\in\C.
	\vspace{-1mm}	
	\end{equation}
	If \,$0\notin\overline{W(A_n)}$, \vspace{-1mm} then
	\begin{equation}
	\label{eq:poly.pseudo.num.op}
		W_\Psi(T)\subseteq\overline{W(T)}\cap \Omega,
	\vspace{-2mm}	
	\end{equation}
 	and analogously for form polynomials.
\end{prop}

\begin{proof}
	Let $\lambda_0\in W_{\Psi}(T)$. By Proposition \ref{prop:pseudo.num}, there is a sequence $\{f_m\}_m\subseteq\dom T(\lambda_0)$ with $\norm{f_m}=1$, $m\in\N$, and $(T(\lambda_0)f_m,f_m)\to0$ for $m\to\infty$. Since $0\notin W(A_n)$ by assumption, the complex \vspace{-2mm} polynomial
	\begin{equation}
		p_m(\lambda)\defeq(T(\lambda)f_m,f_m)=\sum_{k=0}^n(A_kf_m,f_m)\lambda^k, \quad \lambda\in\C,
	\vspace{-2mm}	
	\end{equation}
	has degree $n$ for each $m\in\N$. Let $\lambda^m_1,\dotsc,\lambda^m_n\in\C$ denote its zeros. Then $\lambda^m_j\in W(T)$, $j=1,\dotsc,n$, and $p_m$ admits the \vspace{-2mm} factorisation
	\begin{equation}
		p_m(\lambda)=(A_nf_m,f_m)\prod_{j=1}^{n}(\lambda-\lambda^m_j), \quad \lambda\in\C, \quad m\in\N.
	\vspace{-2mm}
	\end{equation}
	Since $p_m(\lambda_0)\to0$ for $m\to\infty$ and $0\notin\overline{W(A_n)}$, there exists $j_0\in\{1,\dotsc,n\}$ with $\lambda^m_{j_0}\to\lambda_0$, $m\to\infty$, thus $\lambda_0\in\overline{W(T)}$ and $\lambda_0 \in W_\Psi (T) \subseteq \Omega$.
\end{proof}

Next we generalise Theorem \ref{thm:bdd.pseudo.num.ran} (i) to families of sectorial forms with constant domain 
which satisfy a natural continuity assumption, see \cite[Thm.\ VI.3.6]{Kato-1995}. This assumption is met, in particular, by holomorphic form families of type (a) and associated operator families of type (B).

Recall that a family $\formt$ of densely defined closed sectorial sesquilinear forms in $\Hh$ is called holomorphic of type (a) if its domain is constant and the mapping $\lambda\mapsto\formt(\lambda)[f]$ is holomorphic for every $f\in\Dd_\formt\!\defeq\!\dom\formt(\lambda)$. The associated family $T$ of m-sectorial operators is called holomorphic of type~(B), see \cite[Sect.~VII.4.2]{Kato-1995} and also \cite{MR3850318}. Sufficient conditions on form families to be holomorphic of type (a) can be found in \cite[\S VII.4]{Kato-1995}. 

\begin{thm}
\label{prop:cont.fam.form}
\label{prop:hol.fam.incl.i}
	Let $\formt$ be a family of sectorial sesquilinear forms in $\Hh$ with constant domain $\Dd_\formt\defeq\dom\formt(\lambda)$, $\lambda\in\Omega$. Assume that for each $\lambda_0\in\Omega$, there exist $r$, $C>0$ and $w:B_r(\lambda_0)\to[0,\infty)$, $\lim_{\lambda\to\lambda_0}w(\lambda)=0$, such  \vspace{-1mm} that
	\begin{equation}
	\label{eq:cont.fam.form.ass}
		\abs{\formt(\lambda_0)[f]-\formt(\lambda)[f]}\le w(\lambda)\left( \abs{\re\formt(\lambda_0)[f]}+C\norm{f}^2\right)
	 \vspace{-1mm}
	\end{equation}
	for all $\lambda\in B_r(\lambda_0)$ and $f\in\Dd_\formt$.  \vspace{-1mm} Then
	\begin{equation}
		\overline{W(\formt)}\cap \Omega \subseteq W_\Psi(\formt).
	 \vspace{-1mm}	
	\end{equation}
	In particular, 	if $\formt$ is a holomorphic form family of type \textnormal{(a)} with associated holomorphic operator family $T$ of type \textnormal{(B)} in $\Hh$,  \vspace{-1mm}then
	\begin{equation}
	\label{eq:sect-nr-psnr}
		\overline{W(T)}\cap \Omega \subseteq W_\Psi(T), \qquad \overline{W(\formt)}\cap \Omega \subseteq W_\Psi(\formt).
	\end{equation}
\end{thm}

\begin{proof}
	Let $\lambda_0\in\overline{W(\formt)}$. Then there exist $\{\lambda_n\}_n\subseteq\Omega$ and $\{f_n\}_n\subseteq\Dd_\formt$ with $\norm{f_n}=1$, $\formt(\lambda_n)[f_n]=0$, $n\in\N$, and $\lambda_n\to\lambda_0$, $n\to\infty$. We show that $\formt(\lambda_0)[f_n]\!\to\!0$ for $n\!\to\!\infty$ which, in view of \eqref{eq:pseudo.num.id}, implies $\lambda_0\!\in\! W_\Psi(\formt)$.  \vspace{-1mm} By~\eqref{eq:cont.fam.form.ass},
	\begin{equation}
		\abs{\formt(\lambda_0)[f_n]}=\abs{\formt(\lambda_0)[f_n]-\formt(\lambda_n)[f_n]}\le w(\lambda_n)\left(|\re\formt(\lambda_0)[f_n]|+C\right), \quad n\in\N.
	 \vspace{-1mm}
	\end{equation}
	Since $\abs{\re\formt(\lambda_0)[f_n]}\le\abs{\formt(\lambda_0)[f_n]}$ and $w(\lambda_n)\to0$, $n\to\infty$, we obtain that, for $n\in\N$ sufficiently  \vspace{-1mm} large,
	\begin{equation*}
		\abs{\formt(\lambda_0)[f_n]}\le C\frac{w(\lambda_n)}{1-w(\lambda_n)}\longrightarrow 0, \quad n\to\infty.
	 \vspace{-1mm}
	\end{equation*}

	Now suppose that $\formt$ and $T$ are holomorphic families of type \textnormal{(a)}  and \textnormal{(B)}, respectively. We only need to show the second inclusion, the first one then follows from $W(T)\subseteq W(\formt)$ and Corollary \ref{cor:pseudo.num} (ii). The second inclusion follows from what we already proved since for holomorphic form families of type (a), after a possible shift $\formt\!+\!c$ where $c\!>\!0$ is sufficiently large to ensure $\re\formt(\lambda_0)\!\ge\!1$, \cite[Eqn.\ VII.(4.7)]{Kato-1995} shows that assumption \eqref{eq:cont.fam.form.ass} is satisfied.
\end{proof}

Theorem \ref{thm:bdd.pseudo.num.ran} (i) does not extend to analytic families of sectorial linear operators with non-constant form domains, as the following example inspired by \cite[Ex.~VII.1.4]{Kato-1995} illustrates. 

\begin{exple}
\label{ex:ODE}
	Let $\Hh=L^2(0,1)$. The family $T(\lambda)$, $\lambda\in\C$, given by
	\begin{equation}
		\begin{aligned}
			T(\lambda)f & \defeq-f''-\lambda f, \\
			\dom T(\lambda) & \defeq\set{f\in H^2(0,1)}{f(0)=0, \, \lambda f'(1)=f(1)},
		\end{aligned}
	\end{equation}
	is a holomorphic family of m-sectorial operators, but not holomorphic of type~(B). Below we will show  \vspace{-1mm} that
	%\bgcomm{incorporated $0 \in \overline{W_\Psi(T)}$}
	\begin{align}
	\label{eq:veryverylast}
	   0\in\overline{W(T)} \subseteq \overline{W_\Psi (T)}, \qquad 0\notin W_\Psi(T)%, \quad 0 \in \overline{W_\Psi(T)}
	   ;
	 \vspace{-1mm}
	\end{align}
	note that, since $\Omega \!=\! \C$, this implies that the conclusion of Theorem \ref{thm:bdd.pseudo.num.ran} (i) does not hold and that 
	%, in particular, that $\overline{W(T)}\cap \Omega  \not\subseteq W_\Psi(T)$ and hence \bg{also that} $W_\Psi (T) \neq \overline{W_\Psi(T)}\cap \Omega$, \bg{i.e.\ }%and that 
	$W_\Psi(T)$ is not closed in $\C$. %$\Omega$. 
	%\marginpar{\footnotesize \ct{I do not see why} \bg{``$W_\Psi(T)$ not closed in $\Omega$''} \ct{follows from $W_\Psi (T) \!\neq\! \overline{W_\Psi(T)}\cap \Omega$ alone; we need to %argue that $(0,\epsilon]\!\subset\!$ $\overline{W_\Psi(T)} \cap \Omega$ but $0\!\notin\! W_\Psi(T)$.} \\
	%\bg{\footnotesize $W_\Psi (T) \!\neq\! \overline{W_\Psi(T)}\cap \Omega$ is by definition of the relative topology the statement "$W_\Psi(T)$ not closed in $\Omega$"}}

	Indeed, it is not difficult to check that the forms associated to $T(\lambda)$, \vspace{-1mm} $\lambda\in\C$,  
	\begin{equation}
		\formt (0) [f] = \|f'\|^2, \quad \formt (\lambda) [f] = \|f'\|^2 - \lambda \norm{f}^2 - \frac1\lambda |f(1)|^2, 
		\quad \lambda \in\!\C\!\setminus\!\{0\},
	\end{equation} 
	are densely defined, closed and sectorial, but have $\lambda$-depending domain	$\dom \formt(0) \!=\! H_0^1(0,1)$ and $\dom \formt(\lambda)\!=\! \set{f \in H^1(0,1)}{f(0)=0}$ for $\lambda\in\!\C\!\setminus\!\{0\}$. The holomorphy of the family follows from the holomorphy of the integral kernel, i.e.\ the Green's function, of $(T(\lambda)-\mu)^{-1}$, which, for $\lambda\in\C$ and $\mu\in\rho(T(\lambda))\neq\emptyset$, is given by
	\begin{equation}
		G(x,y;\mu,\lambda)=
		\frac{\sin(\sqrt{\mu\!+\!\lambda}x)(\sin(\sqrt{\mu\!+\!\lambda}(1\!-\!y))\!-\!\lambda\sqrt{\mu\!+\!\lambda}\cos(\sqrt{\mu\!+\!\lambda}(1\!-\!y)))}
		{\sqrt{\mu\!+\!\lambda}(\sin\sqrt{\mu\!+\!\lambda}-\lambda\sqrt{\mu\!+\!\lambda}\cos\sqrt{\mu\!+\!\lambda}
		)}
	\end{equation}
	for $0\le x\le y\le1$ and $G(x,y;\mu,\lambda)=G(y,x;\mu,\lambda)$ for $0\le y\le x\le1$, cf.\ \cite[Ex.\ V.4.14, VII.1.5, VII.1.11]{Kato-1995} where the family $T(\lambda) + \lambda$, $\lambda \in \C$, %\vspace{-4mm}
	was~studied. 
	
  	For fixed $\lambda\in\C$, the spectrum of $T(\lambda)$ is given by the singularities %zeros 
  	%\bgcomm{zeros of denominator}
  	of the integral kernel $G(\cdot,\cdot;\mu,\lambda)$,
	\begin{equation}
		\begin{aligned}
			\sigma(T(\lambda)) \setminus \{-\lambda\}
		& 	\!=\!\sigma_{\operatorname{p}}(T(\lambda)) \setminus \{-\lambda\}
			\!=\!\big\{\mu\!\in\!\C \setminus \{-\lambda\}:\lambda\sqrt{\mu\!+\!\lambda}=\tan\sqrt{\mu\!+\!\lambda}\big\}.
		\end{aligned}
	\vspace{-1mm}	
	\end{equation}
	%\bgcomm{only one eigenvalue is of this form \\[1mm] $\mu_\lambda < 0$ for $\lambda \in (0,1)$}
	For $\lambda\in(0,\infty)$ the  operator $T(\lambda)$ is self-adjoint and unbounded from above, and for $\lambda \!\in\! (0,1)$ it has an eigenvalue %its eigenvalues 
	$\mu_\lambda \in \sigma_{\operatorname{p}}(T(\lambda)) \subseteq W(T(\lambda))$ %are 
	of the form $\mu_\lambda = -\lambda - \kappa_\lambda^2 <0$ where $\kappa_\lambda$ is the unique positive solution of $\tanh \kappa = \lambda \kappa$. %Since $\kappa_\lambda \to \infty$ for $\lambda \searrow\ 0$,	there exists $\varepsilon>0$ \ct{such that	$\mu_\lambda <0$ for $\lambda \in (0,\varepsilon]$. 
	Thus $0 \in W(T(\lambda))$ for $\lambda \in (0,1)%\varepsilon]
	$ due to the convexity of $W(T(\lambda))$, which proves $(0,1) \subseteq W(T) \subseteq W_\Psi (T)$ and thus $0\in\overline{W(T)}$%and $(0,\varepsilon]\subset W_\Psi(T)$
	. On the other hand, $0\notin\overline{W(T(0))}=[\pi^2,\infty)$ and so Proposition~\ref{prop:pseudo.num} %(i)
	%\bgcomm{removed (i) in reference to Prop \ref{prop:pseudo.num}}
	implies~$0\notin W_\Psi(T)$.
\end{exple}

\section{Spectral enclosure via pseudo numerical range}
\label{subsec:pseudo.nr.spec.encl}

In this section we derive spectral enclosures for families of unbounded linear operators $T(\lambda)$, $\lambda \in\Omega$, using the pseu\-do numerical range $W_\Psi(T)$. The latter is tailored to enclose the approximate point spectrum.

The spectrum and resolvent set of an operator family $T(\lambda)$, $\lambda \in \Omega$, respectively, are defined \vspace{-1mm} as 
\begin{equation}
	\sigma(T):=\set{\lambda\in\Omega}{0\in\sigma(T(\lambda))} \subseteq %\subset 
	\Omega, 
	\quad \rho(T):=\Omega \setminus \sigma(T),
\end{equation}
and analogously for the various subsets of the spectrum. In addition to the approximate point \vspace{-1mm} spectrum 
\[
	\app(T) \defeq 
	\set{\lambda\in\Omega}{\exists \, \{f_n\}_{n}\subseteq\dom T(\lambda), \norm{f_n}=1, T(\lambda)f_n\to 0, n\to\infty}, 
\]
we introduce the \emph{$\varepsilon$-approximate point spectrum}, see \cite{MR1217705} for the \vspace{-1mm}operator~case,
\begin{equation}
\label{eq:eps-ap-spec}
	\sigma_{{\rm ap}, \varepsilon}(T) \defeq \set{\lambda\in\Omega}{\exists \, f\in\dom T(\lambda), \,\norm{f}=1, \, \norm{T(\lambda)f}<\varepsilon}.
\end{equation}
The latter is a subset of the $\varepsilon$-pseudo \vspace{-1mm}spectrum 
\[
 	\sigma_\varepsilon(T) := \sigma_{{\rm ap}, \varepsilon}(T) \cup\sigma(T),
\]
which was defined for operator functions with unbounded closed values in \cite[Sect.\ 9.2, (9.9)]{MR2359869}, comp.\ also \cite{MR2158921}.

Clearly, for monic linear polynomials $T(\lambda)= A \!-\! \lambda I_{\Hh}$, $\lambda \!\in\! \C$, these notions coincide with the  
spectrum, resolvent set, approximate point spectrum, $\varepsilon$-approximate point spectrum and $\varepsilon$-pseudo spectrum of the linear operator~$A$. 

\begin{prop}
\label{thm:spec.incl.pseudo.num.ran}
	For any operator family $T(\lambda)$, $\lambda \in \Omega$, and every $\varepsilon > 0$,
	we have \vspace{-1mm} $\sigma_{{\rm ap}, \varepsilon}(T) \subseteq W_\varepsilon(T)$,
	\begin{equation}
	\label{eq:app-psi}
	 		\norm{T(\lambda)^{-1}}\le\frac{1}{\varepsilon}, \quad \lambda\in\rho(T)\setminus W_\varepsilon(T),
	\vspace{-1mm} \end{equation}
	\vspace{-1mm} and hence
	\begin{equation}
	\label{eq:res.est}
	   \app(T)\subseteq W_\Psi(T).
	\end{equation}
	If $\sigma(T(\lambda))\subseteq\overline{W(T(\lambda))}$ for all $\lambda\in\Omega$, then
	\begin{equation}
		\sigma(T)\subseteq W_\Psi(T).
	\end{equation}
\end{prop}

\begin{proof}
	The claims follow easily from \eqref{eq:eps-ap-spec} and Definition \ref{def:pseudo-nr} together with Cauchy-Schwarz' inequality and \eqref{eq:W.psi.epsilon} in Proposition \ref{prop:pseudo.num}.
\end{proof}

The following simple examples illustrate some properties of $W_\Psi(T)$ versus $\overline{W(T)}\cap \Omega $, in particular, in view of spectral enclosures.

\begin{exple}
\label{ex:pseudo.num.spec.incl}
	\begin{enumerate}
		\item Let $A\!>\!0$ be self-adjoint in $\Hh$ with $0\!\in\!\sigma(A)$. Then,~for the non-holomorphic family $T(\lambda)\!=\!A\!+\!\abs{\sin\lambda}$, $\lambda \!\in\! \Omega\!:=\C$, it is easy to see~that
		\begin{equation}
			W_\Psi(T)=\sigma(T)=\set{k\pi}{k\in\Z}\not\subseteq\overline{W(T)}\cap \Omega =\emptyset;
		\end{equation}
		notice that this implies $\overline{W_\Psi (T)}\cap \Omega \neq \overline{W(T)}\cap \Omega$, i.e.\ the closures of $W_\Psi(T)$ and $W(T)$ in $\Omega$ do not coincide.
		\item Let $A$ be bounded in $\Hh$ with $\re W(A)>0$, $0\in\sigma(A)$ and $0\notin W(A)$. Consider the  holomorphic family of bounded operators in $\Hh\oplus\Hh$
		\begin{equation}
			T(\lambda)=\left(\begin{array}{cc}
				\lambda A & 0 \\
				0 & \lambda\operatorname{Log}(\lambda+1) I_{\Hh}
			\end{array}\right), \quad \lambda\in \Omega:= \C\setminus(-\infty,-1];
		\end{equation}
		here $\operatorname{Log}:\C\setminus(-\infty,0]\to\set{z\in\C}{\im z\in(-\pi,\pi]}$ denotes the 
		principal value of the complex logarithm. 
		
		This family does not satisfy condition \eqref{eq:ass.markus.bdd} in Theorem~\ref{thm:bdd.pseudo.num.ran} since $0 \in \overline{W(A)}$ by assumption. It is not difficult to show~that
		\begin{equation}
			W_\Psi(T)=\sigma(T)= \C\setminus(-\infty,-1]	\not\subseteq \overline{W(T)} \cap \Omega \subseteq \overline{B_1(-1)} \setminus [-2,-1]. \hspace{-3mm}
		\end{equation}

		In fact, the claims for $W_\Psi(T)$ are obvious. If $\lambda \!\in\! W(T)$, then $\lambda\!\in\!\C\!\setminus\!(-\infty,-1]$ and there exists $x\!=\!(f,g)^{\rm t} \!\in\! \Hh \oplus \Hh$, $(f,g)^{\rm t}\ne (0,0)^{\rm t}$, with
		\[
			\big( T(\lambda) x,x \big) = \lambda \big( (Af,f) + (\ln |\lambda+1| + {\rm i} \arg (\lambda+1) ) (g,g)\big) = 0
			\hspace{-8mm}
		\]
		or, equivalently, noting that $\lambda \neq 0$ implies $g\neq 0$ as $0 \notin W(A)$,
		\[
		 	\lambda =0  \ \vee \ \Big( |\lambda\!+\!1| \!=\! \exp\Big(\!-\!\frac{\re(Af,f)}{(g,g)}\Big) \wedge\arg (\lambda\!+\!1) \!=\! - \frac{\im (Af,f)}{(g,g)} \Big).		
			\hspace{-12mm}													
		\]
		Hence, since  $\re W(A)>0$,
		\begin{align*}
		  	W(T) \setminus \{0\} \!  \subseteq %\subset
		  	\! \big\{ z\!\in\! \C\setminus(-\infty,-1] \,:\, |z\!+\!1| \in (0,1)\big\} \subseteq %\subset 
		  	B_1(-1) \setminus (-2,-1].
			\hspace{-10mm}
		\end{align*}
		
		Moreover, for arbitrary $h\in\Hh$, $h\ne 0$, %choosing $g=f$, we see that
		\begin{equation}
			\left(T\left(\exp{\Big(\!-\!\frac{(Ah,h)}{(h,h)}\Big)} -1\right)\binom{h}{h},  \binom {h}{h} \right)=0.
		\end{equation}
		This shows that $\set{\exp(-z) - 1}{z \in  W(A)} \subseteq W(T)$ and since $\exp$ is entire and non-constant, $W(A)^\circ\neq\emptyset$ implies that $W(T)^\circ\neq\emptyset$ by the open mapping theorem for holomorphic functions. So in this case $W_\Psi(T)^\circ \ne W(T)^\circ$ and both are non-empty.
		\[
		  	W_\Psi(T)^\circ\!=\!\C\setminus(-\infty,-1], \quad  \emptyset \ne W(T)^\circ \subseteq %\subset 
		  	B_1(-1) \setminus (-2,-1]. 
		\]
	\end{enumerate}
\end{exple}

In the following, we generalise the spectral enclosure for bounded holomorphic families in Theorem \ref{thm:bdd.pseudo.num.ran} (ii) to holomorphic form families $\formt$ of type (a) and associated operator families of type (B), i.e.\ $\formt(\lambda)$ is sectorial with vertex $\gamma(\lambda)\!\in\!\R$, semi-angle $\vartheta(\lambda)\!\in\! [0,\frac \pi 2)$ and $\lambda$-independent domain $\dom \formt(\lambda)\!=\! \Dd_\formt$. Here, for $k \in \N_0$, we denote the $k$-th derivative of $\formt$ by
\begin{equation}
	\formt^{(k)}(\lambda)[f] \defeq (\formt(\cdot)[f])^{(k)}(\lambda), \quad f \in \dom \formt^{(k)}(\lambda) \defeq \Dd_\formt = \dom \formt(\lambda), \quad \lambda \in \Omega;
\vspace{-3mm}	
\end{equation}
note that $\formt^{(k)}(\lambda)$ need not be closable or sectorial if $k>0$.

\begin{thm}
\label{thm:pseudo.dense.hol.fam}
\label{thm:markus.B}
	Let $\formt$ be a holomorphic form family of type \textnormal{(a)} with associated holomorphic operator family $T$ of type \textnormal{(B)} in $\Hh$. If there exist $k\in\N_0$, $\mu\in\Omega$ and a core $\Dd$ of $\formt(\mu)$ \vspace{-2mm} with 
	\begin{equation}
	\label{eq:ass.pseudo.dense.hol.fam}
		0 \notin \overline{W\big(\formt^{(k)}(\mu)\big|_\Dd\big)},
	\vspace{-2mm}	
	\end{equation}
	then
	\begin{equation}
		\sigma(T) \subseteq W_\Psi(\formt)=\overline{W(\formt)} \cap \Omega .
	\end{equation}
	If, in addition, the operator family $T$ has constant domain, then
	\begin{equation}
		\sigma(T) \!\subseteq \, W_\Psi(T)=\overline{W(T)}\cap \Omega .
	\end{equation}
\end{thm}

\begin{rem}
	\label{rem:suff.cond.pseudo.dense}
	\begin{enumerate}
		\item Since $\formt(\lambda)$ is densely defined, closed and sectorial for all $\lambda \!\in\! \Omega$, condition~\eqref{eq:ass.pseudo.dense.hol.fam} for $k=0$ has the two equivalent forms
		%\bgcomm{removed repetition of $0 \notin \overline{W\big(\formt(\mu)\big|_\Dd\big)}$}
		\[
 			%0 \notin \overline{W\big(\formt(\mu)\big|_\Dd\big)}$ \ \iff \
 			0 \notin \overline{W\big(\formt(\mu)\big|_\Dd\big)} \ \iff \
 		  	0 \notin \overline{W(T(\mu))};
		\]
		hence, by Proposition \ref{prop:pseudo.num} a sufficient condition for \eqref{eq:ass.pseudo.dense.hol.fam} is 
		$$W_\Psi(T)\neq\Omega.$$
		\item For operator polynomials $T$, which are holomorphic and have constant domain by definition, see Proposition \ref{prop:poly.pseudo.num.op}, no sectoriality assumption is needed for the enclosure
		\begin{equation}
			\app (T) \subseteq W_\Psi (T) \subseteq \overline{W(T)} \cap \Omega.
		\end{equation}
		By Propositions \ref{prop:poly.pseudo.num.op} and \ref{thm:spec.incl.pseudo.num.ran}, the above holds under the mere assumption that $0 \notin \overline{W(A_n)}$ where $A_n$ is the leading coefficient of $T$; note that then \eqref{eq:ass.pseudo.dense.hol.fam} holds with $k=n$ and arbitrary $\mu \in \C$. This generalises the classical result \cite[Thm.\ 26.7]{Markus-1988} for bounded operator polynomials; see also \cite[Prop.\ 3.3]{Wagenhofer-PhD-2007} for the block numerical range.
		\item In general, neither the assumption on holomorphy nor condition \eqref{eq:ass.pseudo.dense.hol.fam} in Theorem \ref{thm:pseudo.dense.hol.fam} can be omitted, see Examples \ref{ex:deriv.cond} and \ref{ex:pseudo.num.spec.incl}.
	\end{enumerate}
\end{rem}

\begin{proof}[Proof of Theorem {\rm \ref{thm:pseudo.dense.hol.fam}}]
	First we show that if condition \eqref{eq:ass.pseudo.dense.hol.fam} holds for some core $\Dd$ of $\formt(\mu)$, it also holds for $\Dd$ replaced by $\Dd_\formt = \dom \formt (\lambda)$, $\lambda \in \Omega$. 
For $k\!=\!0$, this follows from the properties of a core, see \cite[Thm.\ VI.1.18]{Kato-1995}. For $k>0$, 	
	without loss of generality, 
	we may assume that $\re \formt (\mu) \ge 1$. From the proof of \cite[Eqn.\ VII.(4.7)]{Kato-1995}, it is easy to see that the second inequality therein holds for $\formt^{(k)}$, i.e.~there exists a constant $C_\mu >0$ such that
	\begin{equation}
	\label{eq:Kato-VII.(4.7).deriv}
		\big|\formt^{(k)}(\mu)[f,g]\big| \le C_\mu \abs{\formt(\mu) [f]}^\frac12\abs{\formt(\mu) [g]}^\frac12, \quad f,g \in\Dd_\formt.
	\end{equation}
	To prove the claim stated at the beginning assume, to the contrary, that $0 \in \overline{W(\formt^{(k)}(\mu))}$, i.e.\ that there exists a sequence $\{f_n\}_n \subseteq \Dd_\formt$, $\norm{f_n}=1$, $n\in\N$, such that $\formt^{(k)}(\mu)[f_n] \!\to\! 0$ as $n\!\to\!\infty$. By the core property of $\Dd$ for $\formt[\mu]$ and by \cite[Thm.\ VI.1.12]{Kato-1995}, for fixed $n\in\N$, there exists $\{f_{n,m}\}_m\subseteq\Dd$~with
	\begin{equation}
	\label{eq:core.sequ}
	  \hspace{2mm} 
		f_{n,m} \!\to\! f_n, \quad \formt(\mu)[f_{n,m}\!-\!f_n]\!\to\! 0, \quad \formt(\mu)[f_{n,m}] \!\to\! \formt(\mu)[f_n], \quad m\!\to\! \infty. 
	\end{equation}%
	%\marginpar{\vspace{-5cm}\ct{\footnotesize Is this not rather the form analogue of (2.6), i.e.\ that the conditions in [Kato, Thm.\ VII.4.8. (4.14) are necessary?} \\
	%\bg{\footnotesize [Kato, Eq. VII.(4.14)] not needed here, it is a sufficient condition for type (a) holomorphy}}
	Applying \eqref{eq:Kato-VII.(4.7).deriv}, we can estimate
	\begin{equation}
		\begin{aligned}
			\big|\formt^{\!(k)}(\mu) [f_{n,m}] \!-\! \formt^{\!(k)}(\mu) [f_n] \big|
			& \!\le\! \big|\formt^{\!(k)}(\mu) [f_{n,m}, f_{n,m}\!-\!f_n]\big| \!+\! \big|\formt^{\!(k)}(\mu) [f_n \!-\! f_{n,m}, f_n]\big|\\
			& \!\le\! C_\mu \abs{\formt (\mu) [f_{n,m} \!-\! f_n]}^\frac12\!\big(\abs{\formt(\mu) [f_{n,m}]}^\frac12 \!\!+\! \abs{\formt(\mu) [f_n]}^\frac12\!\big).
		\end{aligned}
	\end{equation}
	%\bgcomm{removed repetition of \eqref{eq:core.sequ}}
	Since $\norm{f_n}\!=\!1$, $n\!\in\!\N$, it follows from \eqref{eq:core.sequ} and the above inequality  % together with \eqref{eq:core.sequ}
	that there exists $m_n\ge n$ such\vspace{-1mm} that
	\begin{equation}
	\label{eq:diag.sequ.core}
		\norm{f_{n,m_n}}\ge\frac 12, \quad \abs{\formt^{(k)}(\mu)[f_{n,m_n}]}<\abs{\formt^{(k)}(\mu)[f_n]}+\frac{1}{n}.
	\vspace{-1mm}
	\end{equation}
	In view of $\formt^{(k)}(\mu) [f_n] \to 0$, $n \to \infty$, this implies the required \vspace{-1mm}claim 
	\begin{equation}
		0 \in\overline{W\big(\formt^{(k)}(\mu) \big|_\Dd\big)}.
	\vspace{-1mm}	
	\end{equation}
	This completes the proof that \eqref{eq:ass.pseudo.dense.hol.fam} holds with $\Dd_\formt$ instead of $\Dd$.
	
	By Corollary \ref{cor:pseudo.num} (ii), we have $W_\Psi(\formt)=W_\Psi(T)\subseteq\Omega$. Thus, due to \eqref{eq:sect-nr-psnr}, for the claimed equalities between pseudo numerical and numerical ranges it is sufficient to show 
	$W_\Psi(\formt)\subseteq\overline{W(\formt)}$ and $W_\Psi(\formt)\subseteq\overline{W(T)}$, respectively.
	
	Let $\lambda_0\in W_\Psi(\formt)=W_\Psi(T)$. Then $0\in\overline{W(T(\lambda_0))}$ by Proposition \ref{prop:pseudo.num} and hence there exists $\{f_n\}_n\!\subseteq\!\dom T(\lambda_0)\!\subseteq\!\Dd_\formt$ with $\norm{f_n}\!=\!1$, $n\in\N$, such that
	\begin{equation}
		\label{eq:form-to-0}
		\scalarprod{T(\lambda_0)f_n}{f_n}=\formt(\lambda_0)[f_n]\to 0, \quad n\to\infty.
	\end{equation}
	Define a sequence of holomorphic \vspace{-1mm} functions
	\begin{equation}
		\label{eq:diag.sequ.core.def}
		\varphi_n(\lambda)\defeq \formt(\lambda) [f_n], 
		\quad \lambda\in\Omega, \quad n\in\N.
	\vspace{-1mm}	
	\end{equation}
	Let $K\subseteq\Omega$ be an arbitrary compact subset and let $c>0$ be such that $\re (\formt+c) (\lambda_0) \ge 1$. By \cite[Eqn.\ VII.(4.7)]{Kato-1995}, there exists $b_K>0$ with
	\begin{equation}
		\label{eq:Kato-VII.(4.7)}
		|(\formt+c)(\lambda)[f]|\le b_K |(\formt+c)(\lambda_0)[f]|, \quad  \lambda \in K, \ f \in \Dd_\formt.
	\end{equation}
	Using this, $\norm{f_n}=1$ and \eqref{eq:form-to-0}, we find that, for all $\lambda \in K$,
	\begin{equation}
		\abs{\varphi_n(\lambda)}\le b_{K} \abs{(\formt+c)(\lambda_0)[f_n]} +c\le b_{K} \sup_{n\in\N} \abs{\formt(\lambda_0)[f_n]} +(b_K+1)c<\infty.
	\vspace{-2mm}
	\end{equation}
	Consequently, $\{\varphi_n\}_n$ is uniformly bounded on compact subsets of $\Omega$. By Mon\-tel's Theorem, see e.g.\ \cite[\S VII.2]{Conway-1978}, there exists a subsequence $\{\varphi_{n_j}\}_j\subseteq\{\varphi_n\}_n$ that converges locally uniformly to a holomorphic function $\varphi$. Now assumption \eqref{eq:ass.pseudo.dense.hol.fam} with $\Dd_\formt$, which we proved to hold in the first \vspace{-1mm} step, implies
	\begin{equation}
		\varphi^{(k)}(\mu)=\frac{\d^k}{\d\!\lambda^k}\lim_{j\to\infty}\varphi_{n_j}(\lambda)\bigg|_{\lambda=\mu}=\lim_{j\to\infty}\varphi_{n_j}^{(k)}(\mu) = \lim_{j\to\infty} \formt^{(k)}(\mu) [f_{n_j}]
		\neq0
	\vspace{-1mm}	
	\end{equation}
	and thus $\varphi\not\equiv0$. By \eqref{eq:form-to-0}, we further conclude that $\varphi(\lambda_0)=0$. Then, by Hurwitz' Theorem, see e.g.\ \cite[\S VII.2]{Conway-1978}, there exists a sequence $\{\lambda_j\}_j\subseteq\Omega$ with $\lambda_j\to\lambda_0$ for $j\to\infty$ \vspace{-1mm} and
	\begin{equation}
		0=\varphi_{n_j}(\lambda_j)=\formt(\lambda_j)[f_{n_j}], \quad j\in\N.
	\vspace{-1mm}	
	\end{equation}
	Hence, $\lambda_j\in W(\formt)$ for all $j\in\N$ and so $\lambda_0\in\overline{W(\formt)}\cap \Omega $, as required.
	
	Now assume that the operator family $T$ has constant domain. Then, in the above construction, we have $f_{n_j} \in\dom T(\lambda_0) = \dom T(\lambda_j)$ for every $j\in\N$. It follows that $\lambda_j\in W(T)$, $j\in\N$, and thus $\lambda_0\in\overline{W(T)}\cap \Omega $.
	
	The enclosures of the spectrum follow from Proposition \ref{thm:spec.incl.pseudo.num.ran} and from the fact that 
	$\sigma(T(\lambda)) \subseteq \overline{W(T(\lambda))}$ since $T(\lambda)$ is m-sectorial for all $\lambda\in\Omega$. 
\end{proof}

As forms are the natural objects regarding numerical ranges, it is not surprising that the inclusion $W_\Psi(T)\subseteq\overline{W(T)}\cap \Omega $ in Theorem \ref{thm:pseudo.dense.hol.fam} might cease to hold for more general analytic operator families where the connection to a family of forms is lost. Nevertheless, using an analogous idea as in the proof of Theorem \ref{thm:pseudo.dense.hol.fam}, one can prove the corresponding inclusion for the approximate spectrum.

Recall that an operator family $T$ in $\Hh$ is called holomorphic of type (A) if it consists of closed operators with constant domain and for each $f\in\Dd_T\defeq\dom T(\lambda)$, the mapping $\lambda\mapsto T(\lambda)f$ is holomorphic on $\Omega$. Here, for $k \in\N_0$, the $k$-th derivative of $T$ is defined as
\begin{equation}
	T^{(k)}(\lambda)f \defeq (T(\cdot) f)^{(k)}(\lambda), \quad f \in\dom T^{(k)} (\lambda) \defeq \Dd_T, \quad \lambda\in\Omega.
\end{equation}

\begin{thm}
\label{thm:markus.A}
	Let $T$ be a holomorphic family of type \textnormal{(A)} in $\Hh$. If there exist $k\in\N_0$, $\mu\in\Omega$  and a core $\Dd$ of $T(\mu)$ 
	\vspace{-1mm}
	with
	\begin{equation}
	\label{eq:ass.markus.A}
		0 \notin \overline{W\big(T^{(k)}(\mu)\big|_\Dd\big)},
	\vspace{-1mm}
	\end{equation}
	\vspace{-1mm}then
	\begin{equation}
		\app(T)\subseteq\overline{W(T)} \cap \Omega .
	\end{equation}
\end{thm}

\begin{proof}
	In the same way as in the proof of Theorem \ref{thm:markus.B}, using the analogue~of \cite[Eqn.\,VII.(2.3)]{Kato-1995} for the $k$-th derivative of $T$ and Cauchy-Schwarz' in\-equa\-lity, one shows that \eqref{eq:ass.markus.A} holds with $\Dd_T \!=\! \dom T(\lambda)$, $\lambda \!\in\!\Omega$, instead of~$\Dd$.
	
	We proceed similarly as in the proof of Theorem \ref{thm:pseudo.dense.hol.fam}. Let $\lambda_0\in\app(T)$. There exists a sequence $\{f_n\}_n\subseteq\Dd_T$ with $\norm{f_n}=1$, $n\in\N$, and $T(\lambda_0)f_n\to0$ as $n\to\infty$.	Define a sequence of holomorphic functions
	\begin{equation}
		\varphi_n(\lambda)\defeq \scalarprod{T(\lambda)f_n}{f_n}, \quad \lambda\in\Omega, \quad n\in\N.
	\end{equation}
	Analogously to the proof of Theorem \ref{thm:pseudo.dense.hol.fam}, one uses Cauchy-Schwarz' inequality, equation \cite[Eqn.\ VII.(2.2)]{Kato-1995}, $\lim_{n\to\infty}T(\lambda_0)f_n=0$ and \eqref{eq:ass.markus.A} with $\Dd_T$ in order to show uniform boundedness of $\{\varphi_n\}_n$ on compacta, extract a locally uniformly converging subsequence with limit $\varphi\not\equiv0$ and infer $\varphi(\lambda_0)=0$. One then obtains $\lambda_0\in\overline{W(T)}\cap \Omega $ in the same way as in Theorem \ref{thm:pseudo.dense.hol.fam}.
\end{proof}

\begin{rem}
	Theorems \ref{thm:markus.B} and \ref{thm:markus.A} generalise the classical result \cite[Thm.~III. 26.6]{Markus-1988} for bounded holomorphic families (which follows from Theorem \ref{thm:bdd.pseudo.num.ran}~(ii)).
\end{rem}

Like for the numerical range of unbounded operators, cf.\ \cite[Sct.\ V.3.2]{Kato-1995}, additional conditions are needed for enclosing not only the approximate point spectrum, but the entire spectrum $\sigma(T)$ in $W_\Psi (T)$. 
%\bgcomm{removed repetition of paragraph}
%Like for the numerical range of unbounded operators, cf.\ \cite[Sct.\ V.3.2]{Kato-1995}, additional conditions are needed for enclosing not only the approximate point spectrum, but the entire spectrum $\sigma(T)$ in $W_\Psi (T)$.

\begin{rem}
\label{rem:spec.incl.pseudo.num.ran}
	Let $T$ be a family of closed operators in $\Hh$ and let $T$ be continuous in the generalised sense. If $\app(T)\subseteq\Theta\subseteq\Omega$ and all connected components of $\Omega\setminus\Theta$ contain a point in the resolvent set of $T$, then $\sigma(T)\subseteq\Theta$. In particular, if all connected components of $\Omega\setminus W_\Psi(T)$ have non-empty intersection with $\rho(T)$, \vspace{-2mm} then
	\begin{equation}
		\sigma(T)\subseteq W_\Psi(T).
	\end{equation}
	This follows from the fact that the index of $T(\lambda)$ is locally constant on the set of regular points, see \cite[Thm.\ IV.5.17]{Kato-1995}.
\end{rem}

\section{Pseu\-do block numerical ranges of operator matrix functions and spectral enclosures}
\label{sec:op.mat.fam}
\label{subsec:qnr}
\label{subsec:spec.pseudo.qnr}

In this section we introduce the pseudo block numerical range of $n\times n$ operator matrix functions for which the entries may have unbounded operator values. While we study its basic properties for $n\ge 2$, we study the most important case $n=2$ in greater detail.

We suppose that with respect to a fixed decomposition $\Hh=\Hh_1\oplus\cdots \oplus\Hh_n$ with $n\in\N$, a family $\Ll=\set{\Ll(\lambda)}{\lambda\in\Omega}$ of densely defined linear operators in $\Hh$ admits a matrix representation
\begin{equation}
\label{eq:op.matrix.fam}
	\Ll(\lambda)=\left( L_{ij} (\lambda) \right)_{i,j=1}^n :\Hh\supseteq\dom\Ll(\lambda)\to\Hh;
\end{equation}
here $L_{ij}$ are families of densely defined and closable linear operators from $\Hh_j$ to $\Hh_i$, $i$, $j=1,\dots, n$, and  
\vspace{-1.5mm} $\dom\Ll(\lambda)=\Dd_1(\lambda)\oplus\cdots \oplus\Dd_n(\lambda)$, 
\begin{equation}
	\Dd_j(\lambda)\defeq \bigcap_{i=1}^n \dom L_{ij}(\lambda), \quad j=1,\dots,n.
\vspace{-1.5mm}	
\end{equation}

The following definition generalises, and unites, several earlier concepts: the block numerical range of $n\times n$ operator matrix families whose entries have bounded linear operator values, see \cite{MR3302436}, the block numerical range of unbounded $n \times n$ operator matrices, see \cite{Rasulov-Tretter-2018}, and in the special case $n\!=\!2$, the quadratic numerical range for bounded analytic operator matrix families and unbounded operator matrices, see \cite{Tretter-2010} and \cite{Langer-Tretter-1998}, \cite{Tretter-2009}, respectively. Further, we introduce the new concept of pseudo block numerical range. 

\begin{defi}
\label{def:quad.num.ran}
	\begin{enumerate}
		\item We define the \emph{block numerical range} of $\Ll$ (with respect to the decomposition $\Hh=\Hh_1\oplus\cdots \oplus\Hh_n$) as
		\begin{equation}
			W^{n}(\Ll)\defeq 
			 \{\lambda\in\Omega: \exists\, f\in \!\dom\Ll(\lambda)\cap {\mathcal S}^n \ 0 \!\in\! \sigma(\Ll(\lambda)_f)\}
			%\bigcup_{f\in\dom\Ll(\lambda)\cap{\mathcal S}^n} \sigma \big( \Ll(\lambda)_{f} \big )
		\end{equation}
		where ${\mathcal S}^n\defeq \{ f\!=\!(f_i)_{i=1}^n \!\in\! \Hh	: \norm{f_i}\!=\!1, i\!=\!1,\dots,n\}$ and,
		for $f\!=\!(f_i)_{i=1}^n\!\in\!\dom\Ll(\lambda)\cap {\mathcal S}^n$ with $\lambda \!\in\! \Omega$,
		\begin{equation}
			\Ll(\lambda)_{f}\defeq\left( \Ll_{ij}(\lambda) f_j, f_i \right)\in\C^{n\times n}.
		\end{equation}
		\item We introduce the \emph{pseu\-do block numerical range} of $\Ll$ as
		\begin{equation}
			W^n_\Psi(\Ll)\defeq\bigcap_{\varepsilon>0}W_\varepsilon^n(\Ll), \qquad 
			W_\varepsilon^n(\Ll)\defeq\hspace{-2mm} \bigcup_{\Bb\in L(\Hh),	\norm{\Bb}<\varepsilon}\hspace{-2mm} W^n(\Ll+\Bb), \quad \varepsilon>0.
		\end{equation}
	\end{enumerate}
\end{defi}

Note that, indeed, if $\Ll(\lambda)\!=\!\Aa\!-\!\lambda I_\Hh$, $\lambda \!\in\! \C$, with an (unbounded) operator matrix $\Aa$ in $\Hh$, then 
$\dom \Ll(\lambda)\!=\!\dom \Aa$ is constant for $\lambda\!\in\!\C$ and 
$W^n(\Ll)$ coincides with the block numerical range $W^n(\Aa)$ first introduced in \cite{Rasulov-Tretter-2018} and, for $n\!=\!2$, in \cite{Tretter-2009}. While the pseudo numerical range also satisfies $W_\Psi(\Ll)\!=\!\overline{W(\Ll)} = \overline{W(\Aa)}$ this is no longer true for the pseudo block numerical range when $n>1$; in fact, Example \ref{ex:jordan} below shows that $W_\Psi^2(\Ll)\neq \overline{W^2(\Ll)} = \overline{W^{2}(\Aa)}$ is possible.

\begin{rem}
\label{rem:4.2}
	It is not difficult to see that, for the block numerical range and the pseudo block numerical range of general operator matrix \vspace{-1mm} families,
	\begin{equation}
	\label{eq:qnr.equiv}
		\lambda\in W^n(\Ll) \iff 0\in W^n(\Ll(\lambda))
	\vspace{-1mm}	
	\end{equation}
	and $ W^n(\Ll)\!\subseteq\! W^n_\Psi(\Ll)$. If $\dom \Ll(\lambda)\!=:\!\Dd_\Ll$, $\lambda\!\in\!\Omega$, is constant, we can also~\vspace{-1mm}write
	\[
	W^{n}(\Ll)\defeq 
			\bigcup_{f\in\Dd_\Ll\cap{\mathcal S}^n} \sigma \big( \Ll_{f} \big ).
	\vspace{-2mm}		
	\]
\end{rem}

There are several other possible ways to define the pseudo block numerical range. In the following we show that, in general, they inevitably fail to contain the approximate point spectrum of an operator matrix family. 

\begin{defi} 
\label{eq:alt.def.psi0}
	Define 
	\begin{equation}
		W_{\Psi,0}^{n}(\Ll)\!\defeq\!\set{\lambda\!\in\!\Omega}{0\in\overline{W^{n}(\Ll(\lambda))}}, \quad 
		W_{\Psi,i}^{n}(\Ll)\!\defeq\!\bigcap_{\varepsilon>0}W_{\varepsilon,i}^{n}(\Ll), \ i\!=\!1,2,
	\vspace{-2mm}	
	\end{equation}
	where, for $\varepsilon>0$,
	\begin{equation}
	\begin{aligned}
		W_{\varepsilon,1}^{n}(\Ll) & \!\defeq\! \set{\lambda\in\Omega}{\exists\, f\in\dom\Ll(\lambda) \cap {\mathcal S}^n, \abs{\det(\Ll(\lambda)_{f})}<\varepsilon}\!, \\
		W_{\varepsilon,2}^{n}(\Ll) & \!\defeq\! \!\!\bigcup_{B_i\in L(\Hh_i),\norm{B_i}<\varepsilon} \!\!W^{n}\big(\Ll+\diag(B_1,\dots,B_n)\big).
	\end{aligned}
	\end{equation}
\end{defi}

While for the pseudo numerical range, analogous concepts as in Definition~\ref{eq:alt.def.psi0} coincide by Proposition \ref{prop:pseudo.num}, this is not true for the pseudo block numerical range. Here, in general, we only have the following inclusions.
	
\begin{prop}
\label{prop:nested.def.pseudo.qnr}
	The pseudo block numerical range $W^{n}_\Psi(\Ll)$ satisfies
	\begin{equation}
	\label{eq:pbnri}
		W^n(\Ll) \subseteq W^{n}_{\Psi,1}(\Ll)\subseteq  W_{\Psi,0}^{n}(\Ll)\subseteq W^{n}_{\Psi,2}(\Ll)\subseteq W^{n}_\Psi(\Ll).
	\end{equation}
\end{prop}
	
\begin{proof}
	We consider the case $n=2$; the proofs for $n>2$ are analogous. The leftmost and rightmost inclusions are trivial by definition. For the remaining inclusions, it is sufficient to show that, for every $\varepsilon>0$,
	\begin{equation}
	\label{eq:nested.def.pseudo.qnr}
		W^2_{\varepsilon,1}(\Ll)
		\subseteq \set{\lambda\in\Omega}{0\in\operatorname{B}_{\sqrt{\varepsilon}}(W^2(\Ll(\lambda)))}
		\subseteq W^2_{\sqrt{\varepsilon},2}(\Ll).
	\end{equation}
	Then the respective claims follow by taking the intersection over all $\varepsilon>0$.
	
	Let $\varepsilon>0$ and $\lambda\in W_{\varepsilon,1}^2(\Ll)$. Then there exists $f\in\dom\Ll(\lambda) \cap {\mathcal S}^2$ with
	\begin{equation}
		\sigma(\Ll(\lambda)_{f})=\{\lambda_1,\lambda_2\}\subseteq W^2(\Ll(\lambda)), \qquad \abs{\lambda_1}\abs{\lambda_2}=\abs{\det\Ll(\lambda)_{f}}<\varepsilon.
	\end{equation}
	Now the first inclusion in \eqref{eq:nested.def.pseudo.qnr} follows from
	\begin{equation}
		\dist(0,W^2(\Ll(\lambda)))\le\min\{\abs{\lambda_1},\abs{\lambda_2}\}< \sqrt{\varepsilon}.
	\end{equation}
		
	For the second inclusion, let $\lambda\!\in\!\Omega$ with $\dist(0,W^2(\Ll(\lambda)))\!<\!\!\sqrt{\varepsilon}$, i.e.\ there exists $\mu\!\in\!\C$, $\abs{\mu}\!<\!\!\sqrt{\varepsilon}$, with $\mu\!\in\! W^2(\Ll(\lambda))$ or, equivalently, $0\!\in\! W^2(\Ll(\lambda)\!-\!\mu\Ii_\Hh)$. By \eqref{eq:qnr.equiv}, the latter is in turn equivalent to 
	\begin{equation*}
		\lambda\in W^2(\Ll-\mu\Ii_{\Hh})\subseteq W^2_{\sqrt{\varepsilon},2}(\Ll). \qedhere
	\end{equation*}
\end{proof}

Clearly, in the simplest case $\Ll(\lambda)=\Aa-\lambda I_\Hh$, $\lambda\in\C$, with an $n\times n$ operator matrix $\Aa$ in $\Hh$ we \vspace{-1mm} have
\begin{equation}
\label{eq:ex-lin}
	W_{\Psi,0}^{n}(\Ll)=\overline{W^{n}(\Ll)} =\overline{W^{n}(\Aa)};
\vspace{-2mm}	
\end{equation}
this shows that $W_{\Psi,0}^n(\Ll)$ fails to enclose the spectrum of $\Ll$ whenever $\overline{W^n(\Aa)}$ does.
	
The following example shows that, already in this simple case, in fact \emph{none} of the subsets $W^n_{\Psi,1}(\Ll)\subseteq  W_{\Psi,0}^n(\Ll)\subseteq W^n_{\Psi,2}(\Ll)$ of the pseudo block numerical range $W^n_\Psi(\Ll)$, see \eqref{eq:pbnri}, is large enough to contain the approximate point spectrum $\sigma_{\rm ap}(\Ll)$.	

\begin{exple}
\label{ex:jordan}
	Let $\Hh\!=\!\ell^2(\N)\oplus\ell^2(\N)$ and $\Ll(\lambda)\!=\!\Aa-\lambda I_\Hh$, $\lambda\in\C$, with
	\begin{equation}
	\Aa	\!\defeq\!\left(\begin{array}{cc}
	\!0 \!&\! \diag(m^2\!-\!1:m\!\in\!\N)\!\! \\
	\!0 \!&\! 0\!\!
	\end{array}\right), \ \ \dom \Aa\!\defeq\!\ell^2(\N)\,\oplus\,\dom \diag(m^2\!-\!1:m\!\in\!\N),
	\end{equation}
	where $\diag(m^2-1:m\!\in\!\N)$ is the unbounded maximal multiplication operator in $\ell^2(\N)$ with domain
	\begin{align*}
		&\dom \diag(m^2\!-\!1:m\!\in\!\N)  := \big\{\{x_m\}_m \in \ell^2(\N): \{(m^2\!-\!1)x_m\}_m \in \ell^2(\N) \big\}.
		\vspace{-2mm}
	\end{align*}
	Clearly, $W^2(\Ll)=W^2(\Aa)=\{0\}$. We will now show that
	\begin{equation}
		\{0\} \!= W_{\Psi,1}^2(\Ll)\!=\!W_{\Psi,0}^2(\Ll)\!=\!W_{\Psi,2}^2(\Ll) 
		\ne W^2_\Psi(\Ll)\!=\!\app(\Ll)\!=\!\C.
	\end{equation}
%By \eqref{eq:ex-lin} and since $\Aa$ is triangular, we have $W_{\Psi,0}^2(\Ll)\!=\!\overline{W^2(\Ll)}\!=\!\overline{W^2(\Aa)}\!=\!W^2(\Aa)\!=\!\{0\}$. Since $W_{\Psi,1}^2(\Ll)\! \subseteq 	\!W_{\Psi,0}^2(\Ll)$ by  \eqref{eq:pbnri}, the first and second equality from the left follow. The third equality from the left follows from the definition of $W_{\Psi,0}^2(\Ll)$ and $W^2_{\Psi,2}(\Ll)$ since $\Aa$ is triangular. 
By the definition of $W_{\Psi,2}^2(\Ll)$ and since $W_{\varepsilon,2}^2(\Ll) \!\subseteq\!B_\varepsilon(0)$, $\varepsilon\!>\!0$, it follows that   
$W_{\Psi,2}^2(\Ll)=\{0\}$ which, together with \eqref{eq:pbnri}, proves the first three equalities. 
	To prove the two equalities on the right, and hence the claimed inequality, let $\lambda\!\in\!\C$ be arbitrary. If $\lambda\!=\!0$, then $\lambda \in W^{2}_\Psi(\Ll)$ by \eqref{eq:nested.def.pseudo.qnr}. If $\lambda\!\neq\! 0$, we define the bounded operator matrices
	\begin{equation}
		\Bb_{k}	\defeq\left(\begin{array}{cc}
		-\diag(\frac{\lambda}{m}\delta_{mk}:m\!\in\!\N) & 0 \\[2mm]
		-\diag(\frac{\lambda^2}{m^2}\delta_{mk}:m\!\in\!\N) & \diag(\frac{\lambda}{m}\delta_{mk}:m\!\in\!\N)
		\end{array}\right), \quad k \in\N,
	\end{equation}
	where $\delta_{mk}$ denotes the Kronecker delta. Then $\norm{\Bb_{k}}\to0$ as $k\to\infty$ and a straightforward calculation shows \vspace{-1mm} that
	\begin{equation}
		(\Aa-\lambda I_\Hh) f_k \!=\!\Bb_{k}
		f_k , \quad f_k \!\defeq\!\frac{\widetilde f_k}{\|\widetilde f_k\|}\in\dom\Aa, \quad \widetilde f_k 
		\!=\! \binom{\frac{k(k+1)}{\lambda}e_{k}
		}{e_{k}}, \quad k\in\N.
		\vspace{-1mm} 
	\end{equation}
	On the one hand, for arbitrary $\varepsilon>0$, this implies that there exists $N\in\N$ such that $\norm{\Bb_N}<\varepsilon$ and $0\in\sigma_{{\rm{p}}}(\Aa-\lambda I_{\Hh}-\Bb_N)=\sigma_{\operatorname{p}}(\Ll(\lambda)-\Bb_N)$, whence
	\begin{equation}
		\lambda\in\sigma_{\operatorname{p}}(\Ll-\Bb_N)\subseteq W^2(\Ll-\Bb_N)\subseteq W_\varepsilon^2(\Ll)
	\end{equation}
	and thus $\lambda\in W_\Psi^2(\Ll)$ by intersection over all $\varepsilon>0$. On the other hand, $\lambda\in\app(\Ll)$ since the normalised sequence $\{f_k\}_{k}\subseteq\dom\Ll(\lambda)$ satisfies
	\begin{equation}
		\norm{(\Aa-\lambda) f_k }=\norm{\Bb_{k}f_k}\le\norm{\Bb_{k}}\to 0, \quad k \to\infty.
	\end{equation}
\end{exple}

With one exception, we now focus on the most important case $n\!=\!2$ for which the notation
\begin{equation}
\label{eq:n=2}
\begin{aligned}
&\Ll(\lambda) \!\defeq \!\begin{pmatrix} A(\lambda) \!&\! B(\lambda) \\ C(\lambda) \!&\! D(\lambda) \end{pmatrix} \ \mbox{ in } \Hh=\Hh_1\oplus \Hh_2, \\ 
&\dom \Ll(\lambda) \!\defeq\! \big( \dom A(\lambda) \cap \dom C(\lambda) \big) \oplus \big( \dom B(\lambda) \cap \dom D(\lambda) \big),
\end{aligned}
\end{equation}
is more  customary. We establish various inclusions between the (pseudo) quadratic numerical range $W^2_{(\Psi)}(\Ll)$ and the (pseudo) numerical ranges of the diagonal operator functions $A$, $D$, as well as between $W^2_{(\Psi)}(\Ll)$ and the (pseudo) numerical ranges of the Schur complements of $\Ll$.

\begin{prop}
	\label{prop:op.mat.fam.num}
	\begin{enumerate}
		\item The quadratic numerical range and the pseudo quadratic numerical range satisfy
		\begin{equation}
			W^2(\Ll)\subseteq W(\Ll), \quad W^2_\Psi(\Ll)\subseteq W_\Psi(\Ll).
		\end{equation}
		
		\item Let $\Omega_1:=\{\lambda \in \Omega:\Dd_1(\lambda)=\dom A(\lambda)\}$ and suppose $\dim\Hh_2 >1$.	Then
		\[
			W(A) \cap \Omega_1 \subseteq %\subset 
			W^2(\Ll), \quad  W_\Psi(A) \cap \Omega_1 \subseteq W_{\Psi,2}^2(\Ll) \subseteq W_\Psi^2(\Ll);
		\]
		if $\,\Dd_1(\lambda)\!=\!\dom A(\lambda)$ for all $\lambda\!\in\! W(A)$ or $\lambda\!\in\! W_\Psi(A)$, respectively, then
		\begin{equation}
			W(A)\subseteq W^2(\Ll), \quad  W_\Psi(A) \subseteq W_{\Psi,2}^2(\Ll) \subseteq W_\Psi^2(\Ll).
		\end{equation}
		\item Let $\Omega_2\!:=\!\{\lambda \!\in\! \Omega:\Dd_2(\lambda)\!=\!\dom D(\lambda)\}$ and suppose $\dim\Hh_1>1$. Then 
		\[
			W(D) \cap \Omega_2  \subseteq %\subset 
			W^2(\Ll), \quad  W_\Psi(D) \cap \Omega_2 \subseteq W_{\Psi,2}^2(\Ll) \subseteq W_\Psi^2(\Ll);
		\]
		if $\,\Dd_2(\lambda)\!=\!\dom D(\lambda)$ for all $\lambda\!\in\! W(D)$ or $\lambda\!\in\! W_\Psi(D)$, respectively, then
		\begin{equation}
			W(D)\subseteq W^2(\Ll), \quad  W_\Psi(D) \subseteq W_{\Psi,2}^2(\Ll) \subseteq W_\Psi^2(\Ll).
		\end{equation}
	\end{enumerate}
\end{prop}

\begin{proof}
	The claims for the quadratic numerical range are consequences of \eqref{eq:qnr.equiv} and of the corresponding statements \cite[Prop.\ 3.2, 3.3 (i),(ii)]{Tretter-2009} for operator matrices. So it remains to prove the claims (i) and (ii) for the pseudo quadratic numerical range; the proof of claim (iii) is completely analogous.

	(i) The inclusion for the quadratic numerical range in (i) app\-lied to $\Ll\!+\!\Bb$ with $\norm{\Bb}\!<\!\varepsilon$ yields $W^2_\varepsilon(\Ll)\!\subseteq\! W_\varepsilon(\Ll)$ for any $\varepsilon\!>\!0$. The claim for the pseu\-do quadratic numerical range follows if we take the intersection over all~$\varepsilon\!>\!0$.
	
	(ii) Let $\lambda\!\in\! W_\varepsilon(A) \cap \Omega_1$ with $\varepsilon\!>\!0$ arbitrary. Then there exists a bounded operator $B_\varepsilon$ in $\Hh_1$ with $\norm{B_\varepsilon}\!<\!\varepsilon$ and $\lambda\!\in\! W(A+B_\varepsilon)$. Since $\dom (A(\lambda)+B_\varepsilon)$ $= \dom A(\lambda) \subseteq\dom C(\lambda)$, the inclusion for the quadratic numerical range in (ii) applied to $\Ll + \diag(B_\varepsilon,0_{\Hh_2})$ shows that
	\begin{equation}
		\lambda\in W^2(\Ll+\diag(B_\varepsilon,0_{\Hh_2})) \subseteq W_{\varepsilon,2}^2(\Ll) \subseteq W^2_\varepsilon(\Ll).
	\end{equation}
	By intersecting over all $\varepsilon>0$, we obtain $\lambda\in W^2_{\Psi,2} (\Ll) \subseteq W_\Psi^2(\Ll)$. The second claim is obvious from the first one since then $\Omega_1 \subseteq W_\Psi (A)$.
\end{proof}

Both qualitative and quantitative behaviour of operator matrices are closely linked to the properties of their so-called Schur complements, see e.g.\ \cite{Tretter-2009}; the same is true for operator matrix functions, see e.g.\ \cite{Tretter-2010} for the case of bounded operator values.

\begin{defi}
	The Schur complements of the $2\times 2$ operator matrix family $\Ll=\set{\Ll(\lambda)}{\lambda\in\Omega}$ in $\Hh=\Hh_1\oplus \Hh_2$ as in \eqref{eq:n=2} are the \vspace{-1mm} families
	\begin{alignat*}{2}
		S_1(\lambda) & \defeq A(\lambda)-B(\lambda)D(\lambda)^{-1}C(\lambda), \quad && \lambda\in\rho(D), \\
		S_2(\lambda) & \defeq D(\lambda)-C(\lambda)A(\lambda)^{-1}B(\lambda), \quad &&\lambda\in\rho(A),
	\end{alignat*}
	of linear operators in $\Hh_1$ and $\Hh_2$, respectively, with \vspace{-1mm} domains
	\begin{alignat*}{2}
		\dom S_1(\lambda) & \defeq\set{f\in\Dd_1(\lambda)}{D(\lambda)^{-1}C(\lambda)f\in\dom B(\lambda)}, \quad &&\lambda\in\rho(D), \\
		\dom S_2(\lambda) & \defeq\set{f\in\Dd_2(\lambda)}{A(\lambda)^{-1}B(\lambda)f\in\dom C(\lambda)}, \quad && \lambda\in\rho(A).
	\end{alignat*}
\end{defi}

The following inclusions between the numerical ranges and pseudo numerical ranges of the Schur complements $S_1$, $S_2$ and the quadratic numerical range and pseudo quadratic numerical range, respectively, of $\Ll$~hold.

\begin{prop}
\label{prop:schur.num.incl.qnr}
\label{prop:op.mat.fam.pseudo.num}
	The numerical ranges and pseudo numerical ranges of the Schur complements satisfy
	\begin{equation}
	\label{eq:pseudo.schur.incl.pseudo.qnr}
		W(S_1)\cup W(S_2)\subseteq W^2(\Ll), \quad W_\Psi(S_1)\cup W_\Psi(S_2) \subseteq W_{\Psi,2}^2(\Ll) \subseteq W_\Psi^2(\Ll).
	\end{equation}
\end{prop}

\begin{proof}
	The first claim follows from \eqref{eq:qnr.equiv} and the corresponding statement \cite[Thm.\ 2.5.8]{Tretter-2008} for unbounded operator matrices.
	
	Using the first claim, the second claim can be proven in a similar way as the claim for the pseudo numerical range in Proposition \ref{prop:op.mat.fam.num} (ii).
\end{proof}

The following spectral enclosure properties of the block numerical range and pseudo block numerical range hold for operator matrix functions. They generalise results for the case of bounded operator values from \cite{Wagenhofer-PhD-2007}, see also \cite{Tretter-2010} for $n=2$, as well as the results for the operator function case, i.e.\ $n=1$, in Proposition \ref{thm:spec.incl.pseudo.num.ran}.

\begin{prop}
\label{prop:point.spec.incl.pseudo.qnr}
	Let $\Ll$ be a family of operator matrices. Then
	\begin{equation}
		\sigma_{\operatorname{p}}(\Ll)\subseteq W^{n}(\Ll)\subseteq W_\Psi^{n}(\Ll).
	\end{equation}
\end{prop}
 
\begin{proof}
	The proof of the first inclusion is analogous to the bounded case, see \cite[Thm.\ 2.14]{Wagenhofer-PhD-2007} or \cite[Thm.\ 3.1]{Tretter-2010} for $n\!=\!2$; the second inclusion is obvious, see Remark \ref{rem:4.2}.
\end{proof}

\begin{thm}
	\label{thm:spec.incl.pseudo.qnr}
	Let $\Ll$ be a family of operator matrices in $\Hh=\Hh_1\oplus  \dots \oplus \Hh_n$. For \vspace{-1.5mm} every~$\varepsilon\!>\!0$,
	\begin{equation}
	\label{eq:pseudo.spec.qnr}
		\sigma_{{\rm ap},\varepsilon}(\Ll) \subseteq W_\varepsilon^n (\Ll), 
		\qquad
		\norm{\Ll(\lambda)^{-1}}\le\frac{1}{\varepsilon}, \quad \lambda\in\rho(\Ll)\setminus W_\varepsilon^n(\Ll),
		\vspace{-1.5mm}
	\end{equation}
	and \vspace{-1mm} hence
	\begin{equation}
	\label{eq:app.incl.pseudo.qnr}
		\app(\Ll)\subseteq W_\Psi^n (\Ll);
	\end{equation}
	if, for all $\lambda\in\Omega$, $\sigma(\Ll(\lambda))\subseteq\overline{W^n(\Ll(\lambda))}$, then
	\begin{equation}
		\sigma(\Ll)\subseteq W^n_{\Psi,0}(\Ll) \subseteq W_\Psi^n	(\Ll).
	\end{equation}
\end{thm}

\begin{proof}
	First let $\lambda \!\in\! \sigma_{{\rm ap},\varepsilon}(\Ll)$. Then there exists $f_\varepsilon\!\in\!\dom\Ll(\lambda)$, $\norm{f_\varepsilon}
	\!=\!1$, with $\norm{\Ll(\lambda) f_\varepsilon}\!<\!\varepsilon$. 
	The linear operator in $\Hh$ given \vspace{-1mm} by
	\begin{equation}
	\label{eq:app.sequ.pert}
		\Bb f	\defeq\begin{cases}
			\Ll(\lambda) \mu f_\varepsilon& {\rm if}~f= \mu f_\varepsilon	\in \ls f_\varepsilon, \\
			\ \ \ \ 0 & {\rm if}~f 	\perp f_\varepsilon,
		\end{cases}
	\end{equation}
	is bounded with $\norm{\Bb}\!=\!\norm{\Ll(\lambda)f_\varepsilon}\!<\!\varepsilon$ and $(\Ll(\lambda)\!-\!\Bb)f_\varepsilon\!=\!0$, i.e.\ $\lambda\!\in\!\sigma_{\operatorname{p}}(\Ll\!-\!\Bb)$. By Proposition \ref{prop:point.spec.incl.pseudo.qnr} and since $\|\Bb\|\!<\!\varepsilon$, we conclude that
	$\lambda\!\in\! W^n	(\Ll-\Bb)\!\subseteq\! W_\varepsilon^n(\Ll)$, which proves the first claim.
		
	The resolvent estimate in \eqref{eq:pseudo.spec.qnr} follows from the first claim and from the definition of $\sigma_{\rm{ap},\varepsilon}(\Ll)$, cf.\ the proof of Proposition \ref{thm:spec.incl.pseudo.num.ran}. 

	Taking the intersection over all $\varepsilon>0$ in the first claim, we obtain that $\app(\Ll)\subseteq W_\Psi^n(\Ll)$.
	
	Finally, the assumption that $\sigma(\Ll(\lambda))\!\subseteq\!\overline{W^n(\Ll(\lambda))}$  for all $\lambda\!\in\!\Omega$ implies~that $\sigma(\Ll)\subseteq W^n_{\Psi,0}(\Ll)$, see Definition \ref{eq:alt.def.psi0}. Now the second inequality in the last claim follows from the inclusion $W_{\Psi,0}^n(\Ll)\!\subseteq\! W_\Psi^n	(\Ll)$ by Proposition~\ref{prop:nested.def.pseudo.qnr}.
\end{proof}

\section{Spectral enclosures by pseudo numerical ranges of \\ Schur complements}
\label{sec:schur.app.encl}

In this section we establish a new enclosure of the approximate point spectrum of an operator matrix family $\Ll$ by means of the pseudo numerical ranges of the associated Schur complements and hence, by Proposition \ref{prop:op.mat.fam.pseudo.num}, in $W^2_{\Psi,2} (\Ll)$ and in the pseudo quadratic numerical range $W_\Psi^2(\Ll)$. Compared to earlier work, we no longer need restrictive dominance assumptions.

\begin{thm}
\label{thm:mat.spec.incl.schur.app}
	Let $\Ll$ be a family of operator matrices as in \eqref{eq:n=2}. If $\lambda\in\app(\Ll)\setminus(\sigma(A)\cup\sigma(D))$ is such that one of the conditions 
	\begin{enumerate}
		\item $C(\lambda)$ is $A(\lambda)$-bounded and $B(\lambda)$ is $D(\lambda)$-bounded;
		\item $A(\lambda)$ is $C(\lambda)$-bounded, $D(\lambda)$ is $B(\lambda)$-bounded 
		and both $C(\lambda)$ and $B(\lambda)$ are boundedly invertible; %.
	\end{enumerate}
	is satisfied, then $\lambda\in\app(S_1)\cup\app(S_2)$. 
	If for all $\lambda\in\rho(A)\cap\rho(D)$ one of the conditions {\rm (i)} or {\rm (ii)} is satisfied, then
	\begin{equation}
	\label{eq:mat.fam.schur.app.incl}
		\begin{aligned}
		\app(\Ll)\setminus(\sigma(A)\cup\sigma(D)) &\subseteq\app(S_1)\cup\app(S_2) \\
		&\subseteq W_\Psi(S_1)\cup W_\Psi(S_2) \subseteq W^2_{\Psi,2} (\Ll) \subseteq W^2_\Psi (\Ll).
		\end{aligned}
	\end{equation}
\end{thm}

\begin{proof}
	Let $\lambda\in\app(\Ll)$. Then there exists a sequence $\{(u_n,v_n)\}_n\subseteq\dom\Ll(\lambda)$ with $\norm{u_n}^2+\norm{v_n}^2=1$, $n\in\N$, 
	\vspace{-1mm} and
	\begin{alignat}{3}
	\label{eq:mat.app.seq.1} 
		A(\lambda)u_n+B(\lambda)v_n & \eqdef 	h_n & ~ \to ~0, \quad n & \to\infty, \\
		\label{eq:mat.app.seq.2} 
		C(\lambda)u_n+D(\lambda)v_n & \eqdef 	k_n & ~ \to ~0, \quad n & \to\infty.
	\end{alignat}
	The normalisation implies that $\liminf_{n\to\infty}\norm{u_n}\!>\!0$ or $\liminf_{n\to\infty}\norm{v_n}\!>\!0$. Let $\liminf_{n\to\infty}\norm{u_n}\!>\!0$, without loss of generality $\inf_{n\in\N}\norm{u_n}\!>\!0$. We show that, if $\lambda \in \rho(D)$, then $\lambda\!\in\!\app(S_1)$; if $\liminf_{n\to\infty}\norm{v_n}\!>\!0$, an analogous proof yields that, if $\lambda \in \rho(A)$, then $\lambda\!\in\!\app(S_2)$.
	
	First we assume that $\lambda$ satisfies (i). Since $\lambda\in\rho(D)$, \eqref{eq:mat.app.seq.2} implies~that
	\begin{equation}
		v_n=D(\lambda)^{-1}k_n-D(\lambda)^{-1}C(\lambda)u_n, \quad n\in\N.
	\end{equation}
	Inserting this into \eqref{eq:mat.app.seq.1} and using $\dom D(\lambda)\subseteq\dom B(\lambda)$, we conclude that
	\begin{equation}
	\label{eq:mat.schur.app1}
		S_1(\lambda)u_n+B(\lambda)D(\lambda)^{-1}k_n=h_n ~ \to~ 0, \quad n\to\infty. 
	\end{equation}
	Due to (i) $B(\lambda)D(\lambda)^{-1}$ is bounded and hence $B(\lambda)D(\lambda)^{-1}k_n\to0$, $n\to\infty$.
	Then \eqref{eq:mat.schur.app1} yields that $S_1(\lambda)u_n\to0$, $n\to\infty$. Because $\inf_{n\in\N}\norm{u_n}>0$, 
	we can \vspace{-2mm} set 
	\begin{equation}
		f_n\defeq \frac{u_n}{\norm{u_n}}\in\Dd_1(\lambda)=\dom S_1(\lambda), \quad n\in\N,
	\end{equation}
	and obtain  that $S_1(\lambda)f_n\to0$ for $n\to\infty$, which proves $\lambda\in\app(S_1)$.
	
	Now assume that $\lambda$ satisfies (ii). Since $C(\lambda)$ is invertible, \eqref{eq:mat.app.seq.2} shows~that
	\begin{equation}
	\label{eq:mat.schur.app3}
		u_n=C(\lambda)^{-1}k_n-C(\lambda)^{-1}D(\lambda)v_n \eqdef C(\lambda)^{-1}k_n-w_n, \quad n\in\N,
	\end{equation}
	where $w_n\defeq C(\lambda)^{-1}D(\lambda)v_n\in\dom S_1(\lambda)$ for $n\in\N$ since
	\begin{equation}
		w_n\in\Dd_1(\lambda)=\dom C(\lambda), \quad D(\lambda)^{-1}C(\lambda)w_n=v_n\in\Dd_2(\lambda)=\dom B(\lambda).
	\end{equation}
	Inserting \eqref{eq:mat.schur.app3} into \eqref{eq:mat.app.seq.1} and using $\dom C(\lambda)\subseteq\dom A(\lambda)$, we obtain that
	\begin{equation}
	\label{eq:mat.schur.app2}
		A(\lambda)C(\lambda)^{-1}k_n-S_1(\lambda)w_n=h_n ~ \to ~0, \quad n\to\infty.
	\end{equation}
	Since $C(\lambda)^{-1}$ is bounded, it follows that $C(\lambda)^{-1}k_n\!\to\!0$, $n\!\to\!\infty$. Thus $\inf_{n\in\N}\norm{u_n}>0$ and \eqref{eq:mat.schur.app3} show	that, without loss of generality, we can assume that $\inf_{n\in\N}\norm{w_n}>0$. \vspace{-1mm} Set
	\begin{equation}
		g_n\defeq\frac{w_n}{\norm{w_n}}\in\dom S_1(\lambda), \quad n\in\N.
	\end{equation}
	By (ii) $A(\lambda)C(\lambda)^{-1}$ is bounded and so $A(\lambda)C(\lambda)^{-1}k_n\!\to\! 0$, $n\!\to\!\infty$.
	Now~\eqref{eq:mat.schur.app2} yields $S_1(\lambda)w_n\!\to\!0$ and thus $S_1(\lambda)g_n 
	\!\to\!0$, $n\!\to\!\infty$, which proves~$\lambda\!\in\!\app(S_1)$.
	
	Finally, the first inclusion in \eqref{eq:mat.fam.schur.app.incl} is obvious from what was already shown; the second inclusion in \eqref{eq:mat.fam.schur.app.incl} follows from Proposition \ref{thm:spec.incl.pseudo.num.ran} and the last two inclusions from Proposition \ref{prop:schur.num.incl.qnr}.
\end{proof}

\begin{rem}
\label{rem:nr.qnr.incl}
	If under the assumptions of Theorem \ref{thm:mat.spec.incl.schur.app}, the Schur complements $S_1$ and $S_2$
	satisfy the assumptions of Theorem  \ref{thm:markus.B} or \ref{thm:markus.A} 
	on every connected component of $\rho(D)$ and $\rho(A)$, respectively, then
	\begin{equation}
		\app(\Ll)\setminus(\sigma(A)\cup\sigma(D))\subseteq\overline{W(S_1)}\cup\overline{W(S_2)}\subseteq\overline{W^2(\Ll)},
	\end{equation}
	see Proposition \ref{prop:schur.num.incl.qnr} for the second inclusion. 
\end{rem}

For operator matrix families $\Ll$ with  off-diagonal entries that are symmetric or anti-symmetric to each other, 
we now establish conditions ensuring that the approximate point spectrum of $\Ll$ is con\-tained in the union of the approximate point spectrum of one Schur complement and the pseudo numerical range of the corresponding diagonal entry, i.e.\ $S_1$ and $D$ or $S_2$ and~$A$.

\pagebreak

\begin{thm}
\label{thm:spec.incl.def.indef}
	Let $\Ll$ be an operator matrix family as in \eqref{eq:n=2}.
	\begin{enumerate}
		\item If $\,\lambda\!\in\!\app(\Ll)\!\setminus\!\sigma(D)$ is such that $C(\lambda)\!\subseteq\! \pm B(\lambda)^*\!$, $A(\lambda)$ is accretive, $\mp D(\lambda)$ sectorial with vertex $0$ and $B(\lambda)$ is $D(\lambda)$-bounded, then $\lambda\!\in\!\app(S_1)\cup W_\Psi(D)$. If these conditions hold for all $\lambda\!\in\!\rho(D)$, then
		\begin{equation}
		\label{eq:BB*.def.indef.incl}
			\app(\Ll)\!\setminus\!\sigma(D) \!\subseteq\! \app(S_1)\!\cup\! W_\Psi(D) \!\subseteq\! W_\Psi(S_1)\cup W_\Psi(D);
		\end{equation}
		if $\dim \Hh_1 > 1$, \vspace{-1mm} then 
		\begin{equation}
		\label{eq:BB*.def.indef.incl.pseudo.qnr}
		   \app(\Ll)\!\setminus\!\sigma(D) \subseteq W_{\Psi,2}^2(\Ll)\!\subseteq\! W_\Psi^2(\Ll).
		\end{equation}
		\item If $\lambda\!\in\!\app(\Ll)\!\setminus\!\sigma(A)$ is such that $C(\lambda)\!\subseteq\! \pm B(\lambda)^*\!$, $A(\lambda)$ is sectorial with vertex $0$, $\mp D(\lambda)$ accretive and $C(\lambda)$ is $A(\lambda)$-bounded, then $\lambda\!\in\!\app(S_2)\cup W_\Psi(A)$. If these conditions hold for all $\lambda\!\in\!\rho(A)$, then
		\begin{equation}
			\hspace{9mm} \app(\Ll)\!\setminus\!\sigma(A) \!\subseteq\! \app(S_2)\!\cup\! W_\Psi(A) \!\subseteq\! W_\Psi(S_2)\cup W_\Psi(A);
		\end{equation}
		if $\dim \Hh_2 > 1$, \vspace{-1mm} then 
		\[
		   \app(\Ll)\!\setminus\!\sigma(A) \subseteq W_{\Psi,2}^2(\Ll)\!\subseteq\! W_\Psi^2(\Ll).				
		\]
 	\end{enumerate}
\end{thm}

Note that here we do not assume that the entries of $\Ll$ are holomorphic. In the next section Theorem~\ref{thm:spec.incl.def.indef} 
	will be applied with $B(\lambda) = \e^{\i \omega(\lambda)} B$ and $C(\lambda) = \e^{-\i \omega(\lambda)} C$, where $C \subseteq B^*$ are constant and $\omega$ is real-valued, see the proof of Theorem~\ref{thm:spec.incl.BB*}.

The following corollary is immediate from Theorem {\rm \ref{thm:spec.incl.def.indef}} due to Proposition~\ref{prop:op.mat.fam.num} and Proposition \ref{prop:op.mat.fam.pseudo.num}.

\begin{cor}
\label{cor:spec.incl.def.indef}
	Under the assumptions of Theorem {\rm \ref{thm:spec.incl.def.indef}}, if in {\rm (i)} additionally $\sigma(D) \!  \subseteq %\subset
	\! W_\Psi(D)$, then
	\[
		\app(\Ll) \!\subseteq\! \app(S_1)\!\cup\! W_\Psi(D) \!\subseteq\! W_\Psi(S_1)\cup W_\Psi(D)
		\!\subseteq\! W_{\Psi,2}^2(\Ll)\!\subseteq\! W_\Psi^2(\Ll),
	\]
	and if in {\rm (ii)} additionally $\sigma(A) \! \subseteq %\subset
	\! W_\Psi(A)$, then
	\[
		\app(\Ll) \!\subseteq\! \app(S_2)\!\cup\! W_\Psi(A) \!\subseteq\! W_\Psi(S_2)\cup W_\Psi(A)
		\!\subseteq\! W_{\Psi,2}^2(\Ll)\!\subseteq\! W_\Psi^2(\Ll).
	\]
\end{cor}

\begin{proof}[Proof of Theorem {\rm \ref{thm:spec.incl.def.indef}}.]
	We only prove (i); the proof of (ii) is analogous. Let $\lambda\!\in\!\app(\Ll)\!\setminus\!\sigma(D)$. In the same way as at the beginning of the proof of Theorem \ref{thm:mat.spec.incl.schur.app} we conclude that if $\liminf_{n\to\infty}\norm{u_n}\!>\!0$, then $\lambda \!\in\! \app(S_1)$. It remains to be shown that in the case $\liminf_{n\to\infty}\norm{v_n}\!>\!0$,  without loss of generality $\inf_{n\in\N}\norm{v_n}\!>\!0$, it follows that $\lambda \!\in\! W_\Psi(D)$. 

	Taking the scalar product with $u_n$ in \eqref{eq:mat.app.seq.1} and with $v_n$ in \eqref{eq:mat.app.seq.2}, respectively, we conclude that	
	\begin{alignat}{3}
		\label{eq:pm.1} (A(\lambda)u_n,u_n) && +(B(\lambda)v_n,u_n) & =(h_n,u_n), \quad && n\in\N, \\ 
		\label{eq:pm.2} \pm (u_n,B(\lambda)v_n) && +(D(\lambda)v_n,v_n) & =(k_n,v_n), \quad && n\in\N.
	\end{alignat}
	By subtracting  from \eqref{eq:pm.1}, or adding to \eqref{eq:pm.1}, the complex conjugate of \eqref{eq:pm.2}, we deduce that
	\begin{equation}
		(A(\lambda)u_n,u_n) \mp \overline{(D(\lambda)v_n,v_n)}=(h_n,u_n) \mp \overline{(k_n,v_n)}\to0, \quad n\to\infty.
	\end{equation}
	Taking real parts and using the accretivity of $A(\lambda)$ and $\mp D(\lambda)$, we obtain
	\begin{equation}
		0\le\re(\mp D(\lambda)v_n,v_n)\le\re(A(\lambda)u_n,u_n)\mp\re(D(\lambda)v_n,v_n)\to 0, \quad n\to\infty.
	\end{equation}
	Since $\mp D(\lambda)$ is sectorial with vertex $0$ by assumption, this implies that $(\mp D(\lambda)v_n,v_n)\to0$ and hence $(D(\lambda)v_n,v_n)\to 0$, $n\to\infty$, which proves that $\lambda\!\in\! W_\Psi(D)$ by Proposition \ref{prop:pseudo.num}.
	
	Finally, the first inclusion in \eqref{eq:BB*.def.indef.incl} is obvious from what was already proved; 
	the second inclusion in \eqref{eq:BB*.def.indef.incl} follows from Proposition \ref{thm:spec.incl.pseudo.num.ran}. 
	The last claim in \eqref{eq:BB*.def.indef.incl.pseudo.qnr} is then a consequence of Propositions \ref{prop:op.mat.fam.num} (iii) and \ref{prop:schur.num.incl.qnr}.
\end{proof}

\begin{rem}
	\begin{enumerate}
		\item Sufficient conditions for the inclusions $\sigma(A)\!  \subseteq %\subset
		\! W_\Psi(A)$ or $\sigma(D)\!  \subseteq %\subset
		\! W_\Psi(D)$, respectively, may be found e.g.\ in Theorem \ref{thm:pseudo.dense.hol.fam} or Pro\-po\-sition~\ref{thm:spec.incl.pseudo.num.ran}.
		\item An analogue of Remark \ref{rem:nr.qnr.incl} also holds for Theorem \ref{thm:spec.incl.def.indef}; the details of all possible combinations of assumptions and corresponding inclusions are left to the reader.
	\end{enumerate}
\end{rem}

\section{Application to structured operator matrices}
\label{sec:BB*}

In this section, we apply the results of the previous section to prove new spectral enclosures and resolvent estimates for non-selfadjoint operator matrix functions exhibiting a certain dichotomy.

More precisely, we consider a linear monic family $\Ll(\lambda)=\Aa-\lambda I_\Hh$, $\lambda\in\C$, with a densely defined operator matrix
\begin{equation}
\label{eq:op.mat}
	\Aa\!=\!\left(\begin{array}{cc}
		A & B \\
		C & D
	\end{array}\right), \quad \dom \Aa\!=\! \big( \dom A \cap \dom C \big) \!\oplus\! \big( \dom B
	\cap \dom D \big)
\end{equation}
with $C\!\subseteq %\subset
\! B^*$ in $\Hh\!=\!\Hh_1\oplus\Hh_2$. We assume that the entries of $\Aa$ 
are densely defined closable linear operators acting between the respective spaces $\Hh_1$ and/or $\Hh_2$, 
and that $A$, $-D$ are accretive or even sectorial with vertex $0$.
This means that their numerical ranges lie in 
closed sectors $\Sigma_\omega$ with semi-axis $\R_+$ and semi-angle $\omega = \pi/2 $ or $\omega \in[0,\pi/2)$,
respectively, \vspace{-1mm} given by
\begin{equation}
	\Sigma_\omega\defeq\set{z\in\C}{\abs{\arg z}\le\omega}, \quad \omega\in[0,\pi/2];
\end{equation}
here $\arg:\C\to(-\pi,\pi]$ is the argument of a complex number with $\arg0=0$. 

The next theorem no longer requires bounds on the dominance orders among the entries in the columns of $\Aa$, in contrast to earlier results in \cite[Thm.\ 5.2]{Tretter-2009} where the relative bounds had to be $0$.

\begin{thm}
\label{thm:spec.incl.BB*}
	Let $\Aa$ be an operator matrix as in \eqref{eq:op.mat} with $C\subseteq B^*$. Assume that there exist $\alpha$, $\delta \in \R
	$ and semi-angles $\varphi,\psi\in[0,\pi/2]$ with
	\begin{equation}
	\label{eq:sec.diag.entries}
		\re W(D)\le\delta<0<\alpha\le\re W(A),  \quad W(A)\subseteq\Sigma_\varphi, \quad W(D)\subseteq-\Sigma_\psi.
	\end{equation}
	Suppose further that one of the following holds:
	\begin{enumerate}
		\item $A$, $-D$ are m-accretive, $C$ is $A$-bounded, $B$ is $D$-bounded,
		\item $A$, $-D$ are m-accretive, $A$ is $C$-bounded, $D$ is $B$-bounded and $B$, $C$ are boundedly \vspace{0.9mm} invertible,
		\item \!\!$-D$ is m-sectorial with vertex $0$, i.e.\ $\psi\!<\!\pi/2$, and $B$ is~$D$-bounded,
		\item  $A$ is m-sectorial with vertex $0$, i.e.\ $\varphi\!<\!\pi/2$, and $C$ is $A$-bounded.
	\end{enumerate}
	Then, with $\tau\defeq\max\{\varphi,\psi\}$,
	\begin{equation}
	\label{eq:app.incl.sigma}
		\app(\Aa)\subseteq(-\Sigma_\tau\cup\Sigma_\tau)\cap\set{z\in\C}{\re z\notin(\delta,\alpha)}\eqdef\Sigma;
	\end{equation}
	if, in addition, $\rho(\Aa)\cap\Sigma^{\operatorname{c}}\neq\emptyset$, then $\sigma(\Aa)\subseteq\Sigma$.
\vspace{-5mm}
\end{thm}

\begin{figure}[htbp]
	\centering
	\includegraphics[width=0.68\textwidth]{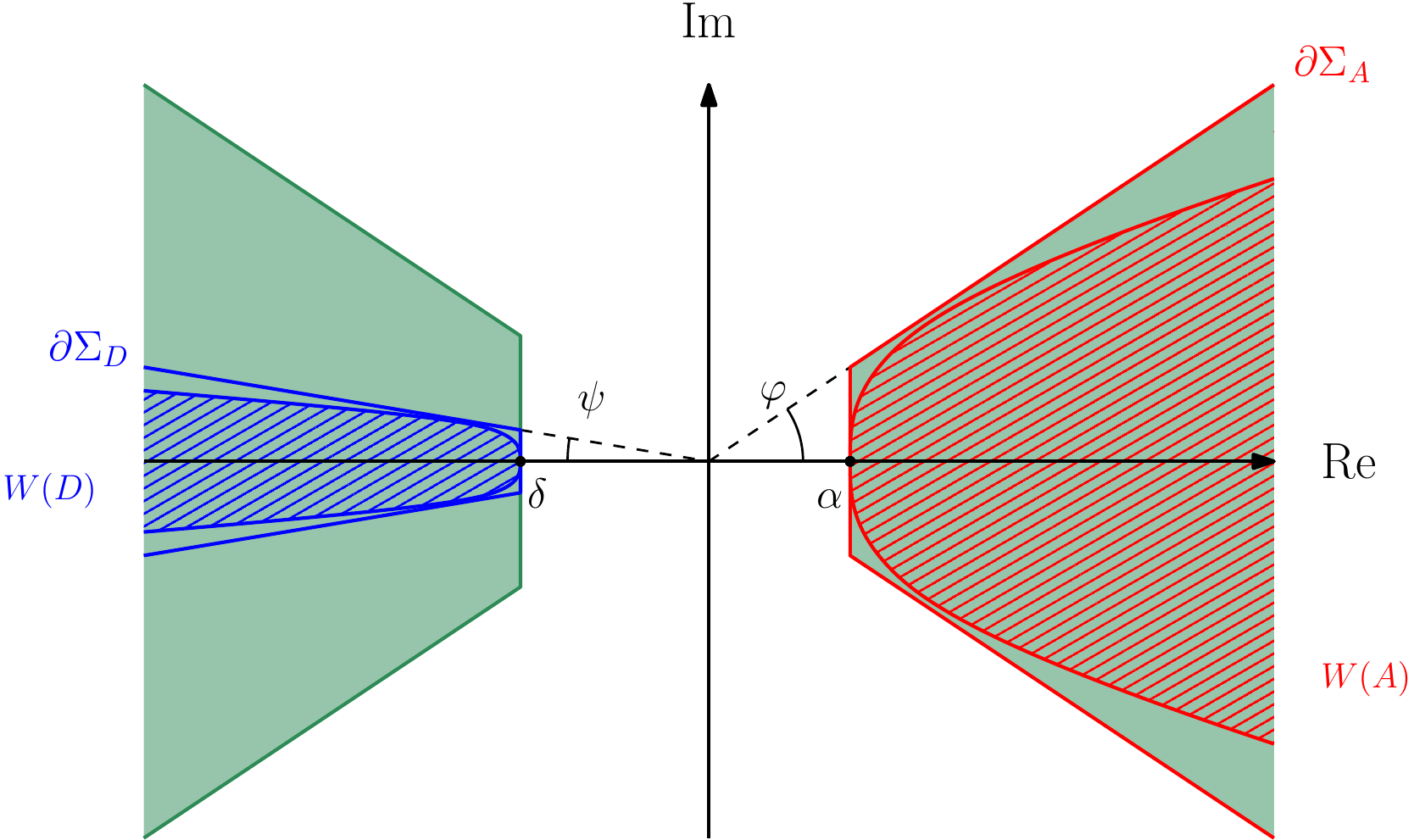}
	\caption{\small The set $\Sigma$ (green) enclosing $\app (\Aa)$, see \eqref{eq:app.incl.sigma}%,  bounded by black dash-dotted line
	; inside the sets $\Sigma_A\!\defeq\! \Sigma_\varphi \!\setminus\! S$ (bounded by red line%green
	) enclosing $W(A)$ (red, dashed) and $\Sigma_D\!\defeq\! -\Sigma_\psi \!\setminus\! S $ (bounded by blue line%orange
	) enclosing
    $W(D)$ (blue, dashed), separated by 	$S\!\defeq\! \{z\!\in\!\C:\re z\!\in\! (\delta,\alpha)\}$, 
	\vspace{-2mm}see~\eqref{eq:sec.diag.entries}.}
	\label{fig:sec}
\end{figure}

The proof of Theorem \ref{thm:spec.incl.BB*} relies on Theorems \ref{thm:mat.spec.incl.schur.app} and \ref{thm:spec.incl.def.indef}, and on the following enclosures for the pseu\-do numerical ranges of the Schur complements.

\begin{lem}
\label{lem:schur.compl.BB*}
	Let $\Aa$ be as in \eqref{eq:op.mat} with $C\!\subseteq\!B^*$ and let $\lambda\in\C$.
	\begin{enumerate}
		\item Suppose $A$,\,$-D$ are uniformly accretive,
		\begin{equation}
		\label{eq:A.D.unif.accr}
			\re W(D)\le\delta<0<\alpha\le\re W(A).
		\end{equation}
		If $\,\re \lambda \in (\delta,\alpha)$, then
		\begin{equation}
			\begin{aligned}
				\lambda\in\rho(D) & \implies \re\overline{W(S_1(\lambda))}\ge\alpha-\re \lambda>0, \\
				\lambda\in\rho(A) & \implies \re\overline{W(S_2(\lambda))}\le \delta-\re \lambda<0.
			\end{aligned}
		\end{equation}
		\item Suppose $A$,\,$-D$ are sectorial with vertex $0$,
		\begin{equation}
			W(A)\subseteq\Sigma_\varphi, \qquad W(D)\subseteq -\Sigma_\psi
		\end{equation}
		with $\varphi,\psi\!\in\![0,\pi/2)$ and let $\tau\!\defeq\!\max\{\varphi,\psi\}$. If $\,\arg\lambda\!\in\!(\tau,\pi-\tau)$, then
		\begin{equation}
		\hspace{6mm} \begin{aligned}
			\lambda\in\rho(D) & \ \implies \ \arg(\overline{W(S_1(\lambda))}+\lambda) \in [-\arg\lambda,\tau], \\
			\lambda\in\rho(A) & \ \implies \ \arg(\overline{W(S_2(\lambda))}+\lambda) \in (\!-\!\pi,-\arg\lambda]\cup[\pi-\tau,\pi];
		\end{aligned}
		\end{equation}
		if $\,\arg\lambda\!\in\!(-\pi+\tau,-\tau)$, then
		\begin{equation}
		\hspace{8mm} 	\begin{aligned}
			\lambda\in\rho(D) & \ \implies \ \arg(\overline{W(S_1(\lambda))}+\lambda) \in [-\tau,-\arg\lambda], \\
			\lambda\in\rho(A) & \ \implies \ \arg(\overline{W(S_2(\lambda))}+\lambda) \in (\!-\!\pi,-\pi+\tau]\cup[-\arg\lambda,\pi].
			\end{aligned}
		\end{equation}
	\end{enumerate}
\end{lem}

\begin{proof}
	We show the claims for $S_1$, the proofs for $S_2$ are analogous. It is easy to see that it suffices to prove the claimed non-strict inequalities for $W(S_1(\lambda))$. Let $\lambda\in\rho(D)$, $f\in\dom S_1(\lambda)\subseteq\dom A\cap\dom B^*$ with $\norm{f}=1$, and set $g\defeq(D-\lambda)^{-1}B^*f$. \vspace{-1mm} Then
	\begin{equation}
	\label{eq:lem.Schur.compl.BB*}
		(S_1(\lambda )f,f)=(Af,f)-\lambda	-\overline{(Dg,g)}+\overline{\lambda}	\norm{g}^2.
	\end{equation}
	
	(i) If $\re\lambda \,\in (\delta,\alpha)$, then \eqref{eq:lem.Schur.compl.BB*} and \eqref{eq:A.D.unif.accr}	show that
	\begin{equation}
		\re(S_1(\lambda)f,f)\ge\alpha-\re\lambda+(-\delta+\re\lambda)\norm{g}^2\ge\alpha-\re \lambda>0.
	\end{equation}	
	
	(ii) We consider $\arg\lambda\!\in\!(\tau,\pi\!-\!\tau)$, the case $\arg\lambda\!\in\!(-\pi\!+\!\tau,-\tau)$ can be shown analogously. By assumption, $\lvert\arg (Af,f)\rvert\!\le\!\varphi\!\le\!\tau$, $\lvert\arg\overline{(-Dg,g)}\rvert\!\le\!\psi\!\le\!\tau$. Together with $\arg(\overline{\lambda}\norm{g}^2)=-\arg\lambda \!\in\!(-\pi\!+\!\tau,-\tau)$, it follows from \eqref{eq:lem.Schur.compl.BB*} that
	\begin{equation*}
	  \arg \big( (S_1(\lambda)f,f)\!+\!\lambda\big) 
		\!=\! \arg\big((Af,f)\!+\!\overline{(-Dg,g)}\!+\!\overline{\lambda}\norm{g}^2\big)\in[-\arg\lambda,\tau].
	\qedhere	
	\end{equation*}
\end{proof}

%\vspace{2mm}

\begin{proof}[Proof of Theorem \textnormal{\ref{thm:spec.incl.BB*}}]
	First we use Lemma \ref{lem:schur.compl.BB*} to show that if $A$ or $-D$ are m-accretive, respectively, then
	\begin{equation}
	\label{eq:Schur.pseudo.nr.in.Sigma}
		W_\Psi(S_2)\subseteq\Sigma \quad \rm{or} \quad W_\Psi(S_1)\subseteq\Sigma.
	\end{equation}
	We prove the claim for $S_1$ by taking complements; the proof for $S_2$ is analogous. To this end, let $\lambda \in \Sigma^{\operatorname{c}} \subseteq \rho (D)$. Then $\re\lambda\in(\delta,\alpha)$ or $\abs{\arg\lambda}\in(\tau,\pi-\tau)$; note that the latter case only occurs if both $A$ and $-D$ are sectorial with vertex $0$, i.e.\ if $\tau < \pi/2$. If $\re\lambda\in(\delta,\alpha)$, Lemma \ref{lem:schur.compl.BB*} (i) implies $0\notin\overline{W(S_1(\lambda))}$, i.e.\ $\lambda\notin W_\Psi(S_1)$ by \eqref{eq:pseudo.num.id}. In the same way, if $\abs{\arg\lambda}\in(\tau,\pi-\tau)$, then $\lambda\notin W_\Psi(S_1)$ follows from Lemma \ref{lem:schur.compl.BB*} (ii); indeed, otherwise we would have $0\in\overline{W(S_1(\lambda))}$ and hence, e.g.\ if $\arg \lambda \in (\tau, \pi - \tau)$,
	\begin{equation}
		\arg (0 + \lambda) = \arg \lambda \in [-\arg \lambda, \tau] \cap (\tau, \pi - \tau) = \emptyset,
	\end{equation}
	and analogously for $\arg \lambda \in (-\pi + \tau,- \tau)$. This completes the proof of \eqref{eq:Schur.pseudo.nr.in.Sigma}.

	We show that assumptions (i) or (iii) imply \eqref{eq:app.incl.sigma}; the proof when assumptions (ii) or (iv) hold is analogous. 

	Assume first that (i) holds and let $\lambda\in\app(\Aa)$. If $\lambda\in\sigma(A)\cup\sigma(D)\subseteq\Sigma$, there is nothing to show. 
	If $\lambda\notin\sigma(A)\cup\sigma(D)$, then Theorem \ref{thm:mat.spec.incl.schur.app} (i) shows that
	$\lambda\in W_\Psi(S_1)\cup W_\Psi(S_2)$ and we conclude $\lambda\in \Sigma$ from \eqref{eq:Schur.pseudo.nr.in.Sigma}. 	
	
	Now assume that (iii) is satisfied. Then $-D$ is m-sectorial with vertex $0$ and $\sigma(D) \subseteq \overline{W(D)}\subseteq \Sigma$. In order to prove \eqref{eq:app.incl.sigma}, we  show $\app (\Aa) \cap \Sigma^{\operatorname{c}} = \emptyset$. To this end, it suffices to prove that
	\begin{equation}
	\label{eq:incl.thm.def.indef}
		\app (\Aa) \cap \Sigma^{\operatorname{c}}\subseteq W_\Psi (S_1) \cup W_\Psi (D-\cdot I_{\Hh_2});
	\end{equation}
	here, in the sequel, we write $D-\cdot I_{\Hh_2}$ for the operator family $D - \lambda I_{\Hh_2}$, $\lambda\in\C$. Indeed, if  \eqref{eq:incl.thm.def.indef} holds, then $W_\Psi(D-\cdot I_{\Hh_2}) = \overline{W(D)} \subseteq \Sigma$ and  \eqref{eq:Schur.pseudo.nr.in.Sigma} yield that $\app (\Aa) \cap \Sigma^{\operatorname{c}} \subseteq \Sigma$ and hence the claim.

	For the proof of \eqref{eq:incl.thm.def.indef}, we will use Theorem \ref{thm:spec.incl.def.indef} (i). To this end, for $\lambda \in \Sigma^{\operatorname{c}}$, we define a rotation angle
	%\bgcomm{maybe this formula for $\omega(\lambda)$ is easier/more intuitive?}
	\begin{equation}
		\omega(\lambda) \defeq \begin{cases}
			0, & \re\lambda \in (\delta, \alpha), \\
			\sgn (\arg \lambda) \big|\frac\pi2 - |\arg\lambda|\big|, %\sgn (\im \lambda\re\lambda) (\frac\pi2 - |\arg\lambda|), 
			& \re \lambda \notin (\delta, \alpha) \wedge |\arg \lambda| \in(\tau, \pi-\tau);
		\end{cases}
	\end{equation}
	note that the second case only occurs if $A$ is sectorial with vertex $0$, i.e.\ if $\tau < \pi/2$, and that then $\lambda \neq 0$ and $|\omega(\lambda)| \in (0,\pi/2-\tau)$. %$\arg \lambda \neq \pm \pi/2$.
	Define a rotated operator matrix family $\widetilde \Ll$ by
	\begin{equation}
		\widetilde \Ll(\lambda) \!\defeq\! \diag \big(\!\e^{\i \omega(\lambda)}\Ii_{\Hh_1}, \e^{-\i \omega(\lambda)}\Ii_{\Hh_2}\!\big) (\Aa-\lambda \Ii_\Hh), \ \ 
		\dom \widetilde \Ll (\lambda) \!\defeq\! \dom \Aa, \quad \lambda \!\in\!\Sigma^{\operatorname{c}}\!.
	\vspace{-1mm} 	
	\end{equation}%
Since, for fixed $\lambda \!\in\! \Sigma^{\operatorname{c}}$, the operator matrix $\diag (\e^{\i \omega(\lambda)}\Ii_{\Hh_1}, \e^{-\i \omega(\lambda)}\Ii_{\Hh_2})$ is bounded and boundedly invertible (even unitary), it is straightforward to show \vspace{-1mm}  that
\begin{equation}
	\lambda \in \app (\Aa) \, \iff \, 0 \in \app (\widetilde \Ll (\lambda)),
\end{equation}
which implies $\app (\widetilde \Ll) = \app (\Aa) \cap \Sigma^{\operatorname{c}}$. Moreover, 
%, if $\re \lambda \notin (\delta, \alpha)$, %
	%\bgcomm{$A-\lambda I_{\Hh_1}$ only accretive if $A$ accretive and $\re \lambda \in (\delta,\alpha)$}
	the angle $\omega(\lambda)$ is chosen such that $\e^{\i \omega(\lambda)}(A -\lambda I_{\Hh_1})$ is accretive, %and 
	$-\e^{-\i \omega(\lambda)}(D -\lambda I_{\Hh_2})$ is % are
	sectorial with vertex~$0$ and $\e^{-\i \omega(\lambda)}C \!\subseteq\! \e^{\i \omega(\lambda)}B^*$ for every $\lambda \!\in\! \Sigma^{\operatorname{c}}$. 
	In fact, if $\re \lambda \in (\delta, \alpha)$, this is obvious.  If $\re \lambda \notin (\delta, \alpha) $ and $|\arg \lambda| \in(\tau, \pi-\tau)$,  then  $\varphi < \pi /2$ and $|\omega(\lambda)| < \pi/2-\tau $ as mentioned above. From $\re W(A) \ge \alpha >0$ and $W(A) \subseteq \Sigma_\varphi$, it thus follows that $\e^{\i \omega(\lambda)}A$ is uniformly accretive and sectorial with vertex~$0$ % by the choice of $\omega(\lambda)$ 
	and, since  $\re (\e^{\i \omega(\lambda)}\lambda) \le 0$%$\e^{\i \omega(\lambda)} (-\lambda) \in \i\R$
	, the claim for $\e^{\i \omega(\lambda)}(A -\lambda I_{\Hh_1})$ holds. %Because $-D$ is sectorial \bg{with vertex $0$}, $\re W(-D) \ge \delta >0$ and $\re (\e^{-\i \omega(\lambda)}\lambda) \ge 0$%$-\e^{-\i \omega(\lambda)} (-\lambda) \in \i \R$
	The proof for $-\e^{-\i \omega(\lambda)}(D -\lambda I_{\Hh_2})$ is analogous.
	%\bgcomm{for fixed $\lambda$, either $\e^{\i \omega(\lambda)} \lambda \in \i\R$ or $\e^{-\i \omega(\lambda)} \lambda \in \i\R$, but not both}
	
	Therefore $\widetilde \Ll$ satisfies the assumptions of Theorem \ref{thm:spec.incl.def.indef} (i) and, because $\sigma (\e^{-\i \omega}(D -\cdot I_{\Hh_2})) = \sigma (D) \cap \Sigma^{\operatorname{c}} = \emptyset$, \eqref{eq:BB*.def.indef.incl} therein yields that
	\begin{equation}
		\app (\Aa)\cap \Sigma^{\operatorname{c}} = \app (\widetilde \Ll) \subseteq W_\Psi (\widetilde S_1) \cup W_\Psi (\e^{-\i \omega}(D -\cdot I_{\Hh_2})),
	\end{equation}
	where $\widetilde S_1$ is the first Schur complement of $\widetilde\Ll$. Now the claim \eqref{eq:incl.thm.def.indef} follows from the above inclusion and from the fact that, since $\e^{\i \omega(\lambda)}\!\ne\! 0$,
	\begin{equation}
		0 \!\in\! \overline{W(\widetilde S_1 (\lambda))}
		\!\iff\,
		0 \!\in\! \overline{W(\e^{\i \omega(\lambda)} S_1 (\lambda))} \!=\! \e^{\i \omega(\lambda)} \overline{W(S_1 (\lambda))} 
		\iff 
		0 \!\in\! \overline{W(S_1 (\lambda))}
\vspace{-2mm}		
	\end{equation}
 	for $\lambda \!\in\! \Sigma^{\operatorname{c}}\!$, and analogously for the family $\e^{-\i \omega}(D -\cdot I_{\Hh_2})$. This completes the proof that (i) and (iii) imply~\eqref{eq:app.incl.sigma}.

	Finally, if $\rho(\Aa)\cap\Sigma^{\operatorname{c}}\neq\emptyset$, then $\Aa$ is closed and $\sigma(\Aa) \!\subseteq\!\Sigma$ follows from $\app(\Aa)\subseteq\Sigma$, see \eqref{eq:app.incl.sigma}, and from the stability of Fredholm index, see \cite[Thm.\ IV.5.17]{Kato-1995}.
\end{proof}

In Proposition \ref{thm:full.spec.incl.BB*} below, we derive sufficient conditions for $\rho(\Aa)\cap\Sigma^{\operatorname{c}}\neq\emptyset$ in Theorem \ref{thm:spec.incl.BB*} for diagonally dominant and off-diagonally dominant operator matrices. For the latter, we use a result of \cite{Cuenin-Tretter-2016}, while for the former we employ the following lemma, inspired by an estimate in \cite[Prob.\ V.3.31]{Kato-1995} for accretive operators.

\begin{lem}
	\label{lem:sec.res.est}
	Let the linear operator $T$ in $\Hh$ be m-sectorial with vertex $0$ or m-accretive, i.e.\ there exists 
	$\omega\!\in\!\left[0,\pi/2\right)$ or 
	$\omega = \pi/2$, respectively, with $\sigma(T)\!\subseteq\!\overline{W(T)}\!\subseteq\!\Sigma_\omega$. 
	\vspace{-1mm} Then
	\begin{equation}
	\norm{T(T\!-\!\lambda)^{-1}}\!\le\! \frac 1{m_T(\arg\lambda)}\!:=\!
	\left\{\begin{array}{cl}
		 \displaystyle\!\!\frac{1}{\sin(\abs{\arg\lambda}\!-\!\omega)},\! & \!\abs{\arg\lambda}\!\in\!(\omega,\omega\!+\!\frac{\pi}{2}),
		 \\[3.5mm]
		 \!\!1, & \!\abs{\arg\lambda}\!\in\![\omega\!+\!\frac{\pi}{2},\pi],
		 \end{array}\right.   
		 \!\lambda\!\notin\!\Sigma_\omega.
	\end{equation}
\end{lem}

\begin{proof}
Let $\lambda\notin\Sigma_\omega$ and $\varepsilon\in(0,\abs{\lambda})$ be arbitrary. Then $\lambda\in\rho(T)$, $-\varepsilon\in\rho(T)$,  $\lambda\neq-\varepsilon$ and we can write
\begin{align}
		T(T-\lambda)^{-1} 
		& =(T+\varepsilon)(T+\varepsilon-(\lambda+\varepsilon))^{-1}-\varepsilon(T-\lambda)^{-1}, \\
		& \label{eq:sec.res.est} =-(\lambda+\varepsilon)^{-1}\left((T+\varepsilon)^{-1}-(\lambda+\varepsilon)^{-1}\right)^{-1}-\varepsilon(T-\lambda)^{-1}. \qquad
	\end{align}
	Since $\varepsilon>0$, it is easy to see that $T+\varepsilon$ is m-accretive or m-sectorial with semi-angle $\omega$ and vertex $0$, and hence so is $(T+\varepsilon)^{-1}$, cf.\ \cite[Prob.\ V.3.31]{Kato-1995} for the m-accretive case. Thus, by \cite[Thm.\ V.3.2]{Kato-1995} and \eqref{eq:sec.res.est}, we can \vspace{-1mm} estimate
	\begin{equation}
		\norm{T(T-\lambda)^{-1}}\le\frac{\abs{\lambda+\varepsilon}^{-1}}{\dist\left((\lambda+\varepsilon)^{-1},\Sigma_\omega\right)}+\frac{\varepsilon}{\dist\left(\lambda,\Sigma_\omega\right)}.
	\end{equation}
	The claim now follows by taking the limit $\varepsilon\to0$ and using the \vspace{-1mm} estimate
	\begin{equation}
	\label{eq:kato.sec.num.dist.est}
		\dist\left(\lambda^{-1},\Sigma_\omega\right)\ge\left\{\begin{array}{cl}
			\displaystyle\frac{\sin(\abs{\arg\lambda}-\omega)}{\abs{\lambda}}, & \,\abs{\arg\lambda}\in\left(\omega,\omega+\frac{\pi}{2}\right),
			\\[4mm]
			\displaystyle\frac{1}{\abs{\lambda}}, & \,\abs{\arg\lambda}\in\left[\omega+\frac{\pi}{2},\pi\right],
		\end{array}\right.
	\vspace{-1mm}	
	\end{equation}
	 cf.\ \cite[Thm.\ 2.2]{Kato-1961-I}.
\end{proof}

\begin{rem}
The inequality in Lemma \ref{lem:sec.res.est} is optimal, equality is  achieved e.g.\ for normal operators 
with spectrum on the boundary of $\Sigma_\omega$.
\end{rem}

\begin{prop}
\label{thm:full.spec.incl.BB*}
	Suppose that, under the assumptions of Theorem {\rm \ref{thm:spec.incl.BB*}}, we strengthen assumptions {\rm (i)} and {\rm (ii)} to
	\begin{enumerate}
		\item[{\rm (i${'}$)}] $A$, $-D$ are m-sectorial with vertex $0$, i.e.\ $\varphi$, $\psi\!<\!\pi/2$ in \eqref{eq:sec.diag.entries}, 
		$C$ is $A$-bounded~with relative bound $\delta_A$ and $B$ is $D$-bounded with relative bound $\delta_D$ such that
		\begin{align}
		\label{eq:BB*spec.incl.diag.cond} 
			\delta_A\delta_D & < \sin(\theta_{0}-\varphi)\sin(\theta_{0}+\psi) =: M_{\theta_0} \in (0,1]
		\end{align}
		where 
		\[
		 	\theta_0:= 
		 	\begin{cases}
		 		\max \big\{ \frac \pi 2 \!+\! \frac{\varphi-\psi}2, \tau \big\}, & \ \varphi \le \psi, \\
		 		\,\min \big\{ \frac \pi 2 \!+\! \frac{\varphi-\psi}2, \pi\!-\!\tau \big\}, & \ \psi < \varphi;
		 	\end{cases}
		\]
		\item[{\rm (ii${'}$)}] $A$, $-D$ are m-accretive, 
		$C\!=\!B^*$, $A$ is $C$-bounded with relative bound~$\delta_C$, $D$~is $B$-bounded with relative bound $\delta_B$ \vspace{-1mm}with
		\[
		   \delta_B \delta_C < 1,
		\]
		$B$, $C$ are boundedly in\-vert\-ible, and the relative boundedness constants $a_C$, $a_B \!\ge\! 0$, $b_C$, $b_B \!\ge\! 0$ in
		\begin{alignat*}{2}
		  \qquad &\|Ax\|^2\le a_C^2\|x\|^2+b_C^2\|Cx\|^2, \quad &&x\in \dom C, \\
			\qquad &\|Dy\|^2\le a_B^2\|y\|^2+b_B^2\|By\|^2, \ &&y \in \dom B,
		\end{alignat*}
	  	satisfy
		\[
			\sqrt{a_C^2\|B^{-1}\|^2+b_C^2} \sqrt{a_B^2\|B^{-1}\|^2+b_B^2}<1.
		\]
	\end{enumerate}
	Then $\rho(\Aa)\cap\Sigma^{\operatorname{c}}\neq\emptyset$ and hence
	\begin{equation}
		\sigma(\Aa)\subseteq(-\Sigma_\tau\cup\Sigma_\tau)\cap\set{z\in\C}{\re z\notin(\delta,\alpha)} = \Sigma.
	\end{equation}
\end{prop}

\begin{proof}
	By Theorem \ref{thm:spec.incl.BB*}, it suffices to show $\rho(\Aa)\cap\Sigma^{\operatorname{c}}\neq\emptyset$. 
	
	Suppose that (i${'}$) holds and let	$\lambda\!=\!r\e^{\i\theta}$ with $r\!>\!0$, $\theta \!\in\! (\tau,\pi\!-\!\tau)$ to be~cho\-sen later.
	Then $\lambda \!\in\! \rho(A) \cap \rho(D)$. Since $\frac 1{M_{\theta_0}} \delta_A \delta_D \!<\! 1$, there exists $\varepsilon\!>\!0$
	\vspace{-2mm}  so~that
	\begin{equation}
	\label{eq:3factors}
	  	\frac 1 {M_{\theta_0} - \varepsilon }  (\delta_A + \varepsilon ) (\delta_D + \varepsilon ) < 1.
	\end{equation} 
	Due to the relative boundedness assumption on $C$, there exist $a_A$, $b_A>0$, $b_A\in[\delta_A,\delta_A+\varepsilon)$ such that
	\begin{equation}
	\label{eq:spec.incl.BB*1}
		\norm{C(A-\lambda)^{-1}}\le a_A\norm{(A-\lambda)^{-1}}+b_A\norm{A(A-\lambda)^{-1}}.
	\end{equation}
	Since $A$ is m-sectorial with semi-angle $\varphi$ and vertex $0$, we have the estimate
	\begin{equation}
	\label{eq:spec.incl.BB*2}
		\norm{(A-\lambda)^{-1}}\le\frac{1}{\dist(\lambda,W(A))}\le\frac{1}{r m_A(\theta)},
	\end{equation}
	with $m_{A}	(\theta)$ defined as in Lemma \ref{lem:sec.res.est}, see \cite[Thm.\ V.3.2]{Kato-1995} or 
	\eqref{eq:kato.sec.num.dist.est}. Consequently, by \eqref{eq:spec.incl.BB*1}, \eqref{eq:spec.incl.BB*2} and Lemma \ref{lem:sec.res.est}, we obtain
	\begin{equation}
	\label{eq:spec.incl.BB*5}
		\norm{C(A-\lambda)^{-1}}
		\le\frac{a_A}{r m_A(\theta)}
		    +\frac{b_A}{m_A(\theta)}.
	\end{equation}
	Similarly, since $-D$ is m-sectorial with semi-angle $\psi$ and vertex $0$, and using Lemma \ref{lem:sec.res.est} as well as \eqref{eq:kato.sec.num.dist.est} and $|\arg (-\lambda)| = \pi-\theta$, we conclude that there exist $a_D$, $b_D>0$, $b_D\in[ \delta_D,\delta_D+\varepsilon)$ with
	\begin{equation}
	\label{eq:spec.incl.BB*6}
		\norm{B(D-\lambda)^{-1}}\le\frac{a_D}{rm_{-D}	(\pi -\theta)}+\frac{b_D}{{m_{-D}(\pi -\theta)}}
	\end{equation}
	with $m_{-D}(\pi -\theta)$ defined as in Lemma \ref{lem:sec.res.est} and hence
	\begin{equation}
	\label{eq:4factors}
		\| C(A-\lambda)^{-1} B(D-\lambda)^{-1} \| \!\le\! \frac {b_A b_D}{M_\theta\!} \Big( \frac{a_A}{r b_A} \!+\! 1 	\Big) 	\Big( \frac{a_D}{r b_D} 	\!+\! 1  \Big).
	\end{equation}
	Here the function
	\[  
	   [\varphi, \pi-\psi] \to [0,1], \quad \theta \mapsto	   M_\theta \defeq m_A(\theta) m_{-D}(\pi -\theta), 
	\]
	is continuous, monotonically increasing for $\theta \le  \widetilde \theta_0 := \frac \pi 2 + \frac{\varphi-\psi}2 \in [\varphi, \pi-\psi]$ and decreasing for $\theta \ge \widetilde \theta_0$. Hence, the restriction of $\theta \mapsto M_\theta$ to $[\tau, \pi-\tau]$ attains its maximum at $\theta_0$ and we can choose $\delta>0$ such that $M_{\theta_0} - \varepsilon	< M_\theta$ for 
	$\theta \in (\theta_0-\delta,\theta_0+\delta) \cap (\tau,\pi-\tau)$. Now we fix such a $\theta$. Using \eqref{eq:4factors} and \eqref{eq:3factors}, we conclude that there exists 
	$r>0$ so large that
	\begin{equation}
	  	\| C(A-\lambda)^{-1} B(D-\lambda)^{-1} \| \!\le\! \frac {(\delta_A+\varepsilon)(\delta_D+\varepsilon)} {M_{\theta_0}\!\!-\!\varepsilon	} \Big( \frac{a_A}{r b_A }  \!+\! 1	\Big) \Big( \frac{a_D}{r b_D}\!+\! 1\Big) \!<\! 1.\!
	\end{equation}
	This implies $1 \!\in\! \rho( C(A\!-\!\lambda)^{-1} B(D\!-\!\lambda)^{-1})\!$  and thus $\lambda \!\in\! \rho(\Aa)$ by \cite[Cor.~2.3.5]{Tretter-2008}.
	
	Suppose that (ii${'}$) is satisfied. By the assumptions on $B$, $C$, the operator $\Ss\!\defeq\!\Ss_1$ is self\-adjoint and has a spectral gap
	$(-\|B^{-1}\|^{-1},\|B^{-1}\|^{-1})$ around~$0$. Then \cite[Thm.\ 4.7]{Cuenin-Tretter-2016} with $\beta_T = 1/\norm{B^{-1}}$ therein implies that $\i\R \subseteq \rho(\Aa)$.
\end{proof}

\section{Application to damped wave equations in $\Rd$ with unbounded damping}
\label{sec:dwe}

In this section we use the results obtained in Section \ref{subsec:pseudo.nr.spec.encl} to derive new spectral enclosures for linearly damped wave equations with non-negative possibly singular and/or unbounded damping $a$ and potential 
$q$.

Our result covers a new class of unbounded dampings which are $p$-subord\-inate to $-\Delta+q$, a notion going back to \cite[\S I.7.1]{MR0342804}, \cite[\S 5.1]{Markus-1988}, cf.\ \cite[Sect.~3]{Tretter-Wyss-2014}.

\begin{thm}
	\label{thm:pencil.spec.incl}
	Let $\formt$ be a quadratic pencil of sesquilinear forms given by
	\begin{equation}
		\formt(\lambda)\defeq\formt_0+2\lambda\forma+\lambda^2, \quad \dom\formt(\lambda)\defeq\dom\formt_0, \quad \lambda\in\C,
	\end{equation}
	where $\formt_0$ and $\forma$ are densely defined sesquilinear forms in $\Hh$ such that $\formt_0$ is closed, $\formt_0\ge\kappa_0\ge0$, $\forma\ge\alpha_0\ge0$ and $\dom\formt_0\subseteq\dom\forma$.  Suppose that there exist $\kappa \le \kappa_0$ and $p\in(0,1)$ such that $\forma$ is $p$-form-subordinate with respect to $\formt_0-\kappa\ge0$, i.e. there is $C_{p}>0$ with
	\begin{equation}
	\label{eq:pencil.subordinate}
	\forma[f]\le C_{p}\big((\formt_0-\kappa)[f]\big)^p \big(\norm{f}^2\big)^{1-p}, \quad f\in\dom\formt_0.
	\end{equation}
	Then the family $\formt$ is holomorphic of type \textnormal{(a)}. If $\,T$ denotes the associated holomorphic family of type \textnormal{(B)}, then
	\[
 		\sigma(T) \subseteq W_\Psi(T)  \subseteq \big\{ z\!\in\!\C\!:  \re z\le 0 \big\}
	\]
	and the following more precise \vspace{1mm} spectral enclosures hold:
	\begin{enumerate}
		\item The non-real spectrum of $\,T$ is \vspace{-1.5mm}contained~in
		\begin{align*}
		\label{eq:pencil.spec.incl}
			\sigma(T)\setminus \R \subseteq W_\Psi(T) \setminus \R \subseteq \!
			\bigg\{ & z\!\in\!\C\!: \re z\le-\alpha_0, \, |z| \ge \sqrt{\kappa_0}
			, \\[-3mm]
			&\abs{\im z}\!\ge\! \sqrt{\max\!\Big\{0,C_{p}^{-\frac{1}{p}}\!\abs{\re z}^\frac{1}{p}\!\!-\!\abs{\re z}^2\!\!+\!\kappa\Big\}}\bigg\}; 
			\hspace{-10mm}
			\\[-6mm]
		\end{align*}
		\item 
		if $\,p\!<\!\frac 12$ or if $\,p\!=\!\frac 12$ and $C_{\frac 12}\! <\!1$ or if $p=\frac12$ and $C_\frac12 = 1$ and $\kappa>0$, the real spectrum of $\,T$ \vspace{-1mm} satisfies~either
		\[
	  		\sigma(T)\cap \R = \emptyset \quad \mbox{ or } \quad	\sigma(T) \cap \R \subseteq	[s^-,s^+%s_+^-,s_+^+
	  		], 
			\vspace{-1mm}
		\]
		if $\,p \!>\! \frac 12$ or if $\,p\!=\!\frac 12$ and $C_{\frac 12}\!>\!1$ or if $p=\frac12$ and $C_\frac12 = 1$ and $\kappa\le 0$, the real spectrum of $\,T$ satisfies \vspace{-1mm} either
		\[
	 		\sigma(T)\cap \R \subseteq (-\infty, r^+] \cup	[s^-\!, s^+]  
	  		\ \ \mbox{ or } \ \
	 		\sigma(T) \cap \R \subseteq 
	  		(-\infty,s^+],  
		\vspace{-1mm}		
		\]
		where $\infty<r^+< s^- \!\le\! s^+ \!\le\! 0$ depend on $p$, $C_p$, $\kappa_0$ and $\kappa$; 
		\vspace{1.5mm}
		\item if $\kappa=0$ and $p < \frac 12$, \vspace{-0.5mm} then
		%\marginpar{\ct{\footnotesize adaption was needed?} \\
		%\bg{\footnotesize yes, previous bounds were not sharp}}
 		\begin{align*}
 		\qquad \ \ 
 		&\sigma(T)\cap \R = \emptyset   \hspace{6.5cm}  \mbox{ if }  (C_p^2)^{\frac 1{1-2p}} \!<\! \kappa_0, \\[-1mm]
 		&\sigma(T)\cap \R \subseteq \!\!\Big[
 		\!-\!C_pt_0^p\!-\!\sqrt{C_{p}^pt_0^{2p}\!-\!t_0 }, -C_p\kappa_0^p \!+\!\sqrt{C_p^2\kappa_0^{2p}\!-\!\kappa_0}\Big)
 		\Big]\!\! \\[-1mm]
 		& \hspace{8.7cm}  \mbox{ if }   (C_p^2)^{\frac 1{1-2p}} \!\ge\! \kappa_0,\\[-10mm]
		\end{align*}
		where $t_0:= \max \big\{ \big( 4C_p^2p(1\!-\!p) \big)^{-\frac 1{2p-1}}\!\!,\kappa_0 \big\}$;
		\item if $\kappa=0$ and $p= \frac 12$, \vspace{-0.1mm} then
 		\begin{alignat*}{2}
 			&\sigma(T)\cap \R = \emptyset && \  \mbox{ if }  C_{\frac 12}\!<\!1 \mbox{ and } \kappa_0 \!>\! 0, \\
 			&\sigma(T) \cap \R \subseteq \{0\} && \  \mbox{ if } C_\frac 12 \!<\!1 \mbox{ and } \kappa_0 \!=\! 0, \\[-1mm]
 			\hspace{6mm} &\sigma(T)\cap \R \subseteq \!\Big(\!\!-\!\infty,-\Big(C_\frac12 \!-\!\sqrt{C_\frac12^2 \!-\! 1}\Big) \kappa_0^{\frac 12}	\Big] & &\  \mbox{ if }  C_{\frac 12} \!\ge\! 1;
 		\end{alignat*}
		\item if $\kappa=0$ and $p> \frac 12$, \vspace{-0.1mm} 
		then 
		%\marginpar{\ct{\footnotesize Not sure this is okay, $g_+$ not mon.\ \& your bound seems to be $g_+(x_+)$, I get $g_+(t_0)$, see below.} \\
		%\bg{\footnotesize indeed now it is correct, I made a mistake here}}	
		%\vspace{-5mm}
		\begin{alignat*}{2}
			&\sigma(T)\cap \R \subseteq \Big(\!\!-\!\infty,  -C_pt_0^p+\sqrt{C_{p}^2t_0^{2p}\!-\!t_0 }\,\Big] && \ \mbox{ if }  \kappa_0 > 0, \\[-1mm]
			& \sigma(T)\cap \R \subseteq \Big(\!\!-\!\infty, -C_pt_0^p+\sqrt{C_{p}^2t_0^{2p}\!-\!t_0 }\,\Big] \cup \{0\} && \  \mbox{ if }  \kappa_0 = 0,
		\\[-8mm]	
		\end{alignat*}
		where $t_0:= \max \big\{ \big( 4C_p^2p(1\!-\!p) \big)^{-\frac 1{2p-1}}\!\!,\kappa_0 \big\}$.
		\end{enumerate}
\end{thm}

\begin{figure}[h]
	\hspace*{-4mm}
	\subcaptionbox*{\hspace{5.5mm}{\scriptsize (a)\,$p\!=\!0.4$,\,$C_p\!=\!1.3$,\,$\kappa\!=\!-2$, \\ 
	\hspace*{9mm} $\alpha_0\!=\!2.5$,\,$\kappa_0\!=\!5$}}{\includegraphics[width=0.322\textwidth]{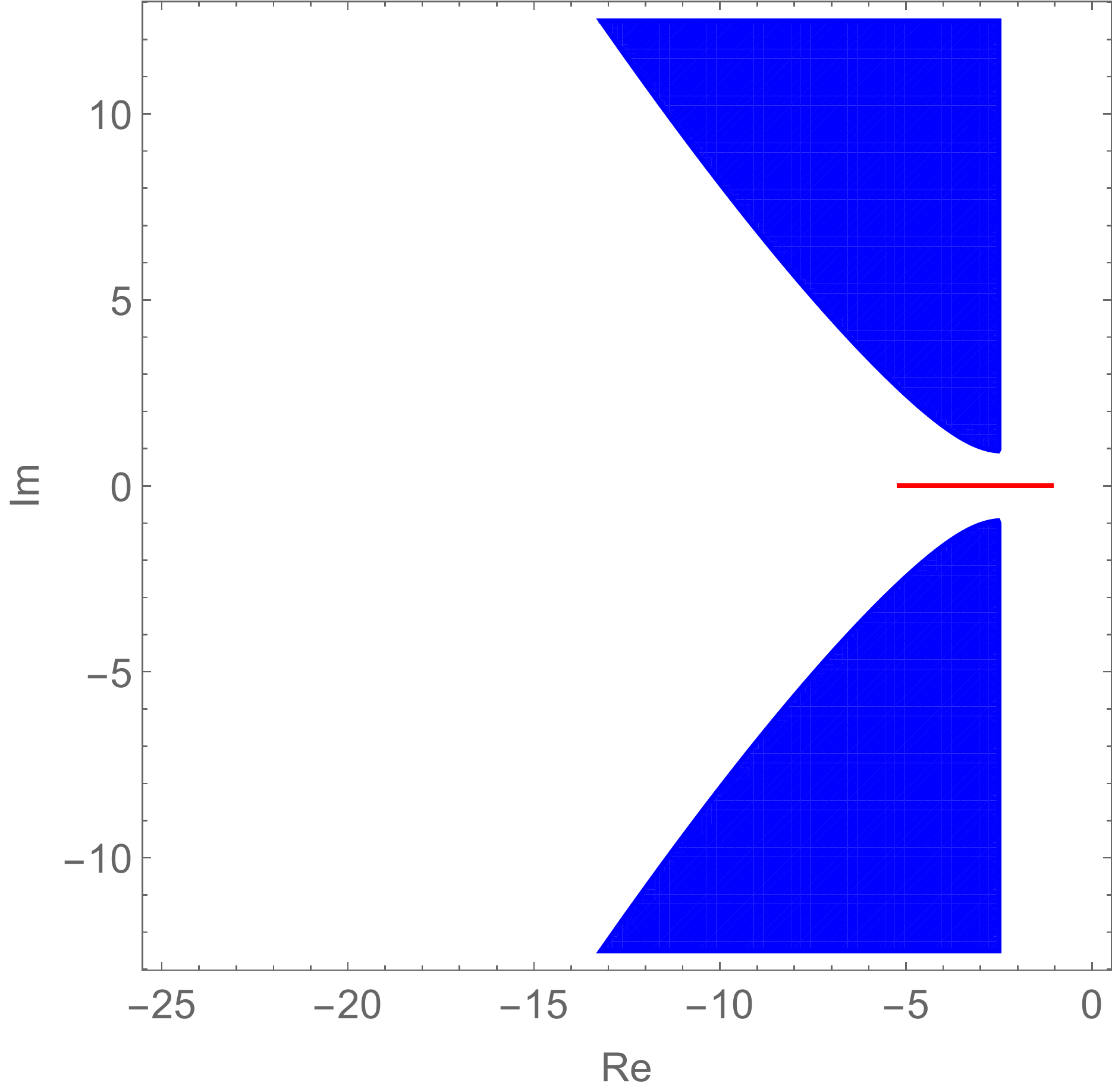}} \ %\\[2mm]
	\subcaptionbox*{{\hspace{4.5mm}  \scriptsize (b)\,$p\!=\!0.5$,\,$C_p\!=\!0.7$,\,$\kappa\!=\!3$, \\
	\hspace*{9mm} $\alpha_0\!=\!0.5$,\,$\kappa_0\!=\!6$}}{\includegraphics[width=0.322\textwidth]{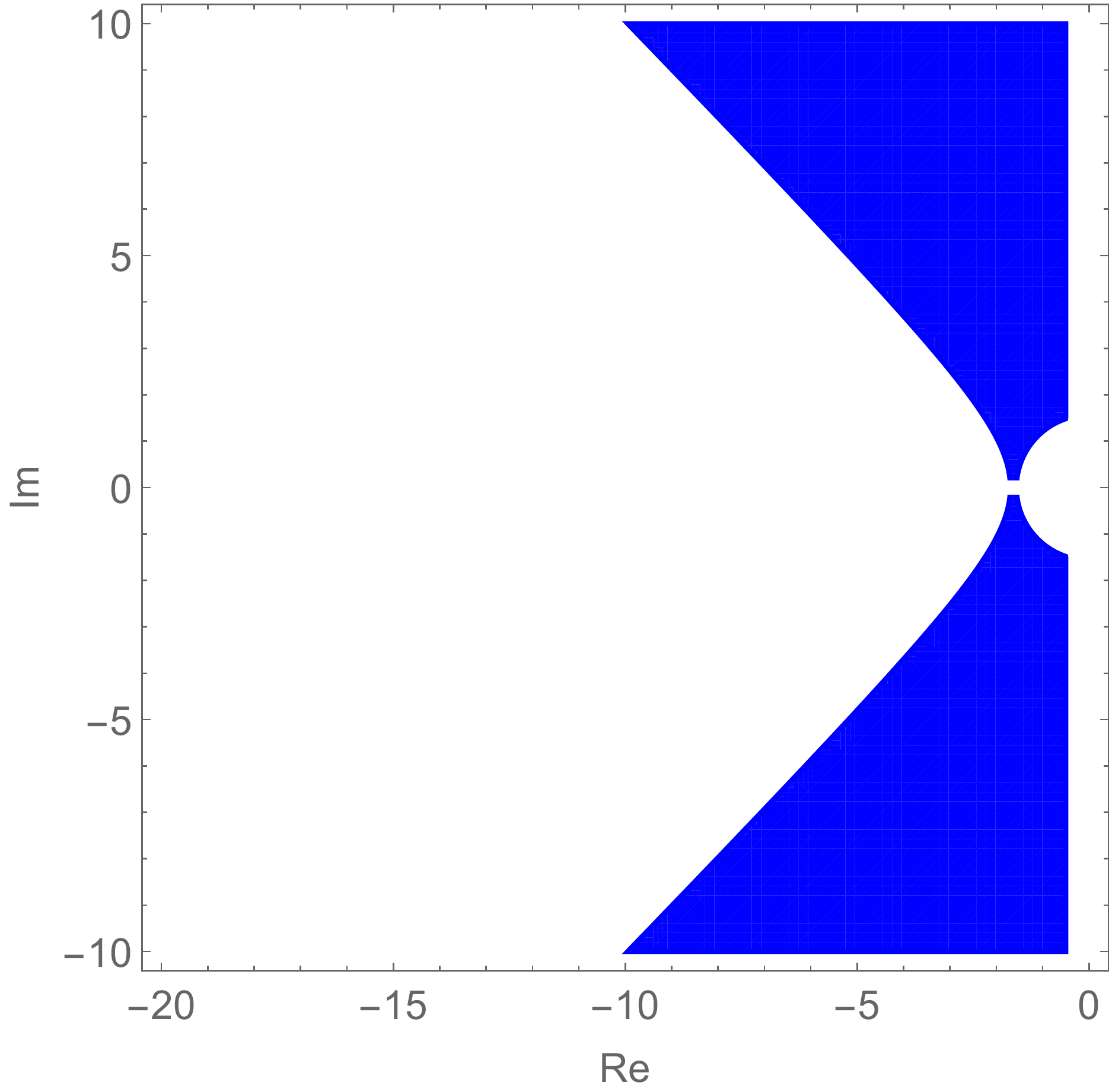}} \ %\\[2mm]
	\subcaptionbox*{{\hspace{3mm} \scriptsize (c)\,$p\!=\!0.65$,\,$C_p\!=\!0.5$,\,$\kappa\!=\!-5$, \\
	\hspace*{8mm}$\alpha_0\!=\!1$,\,$\kappa_0\!=\!0$}}{\includegraphics[width=0.322\textwidth]{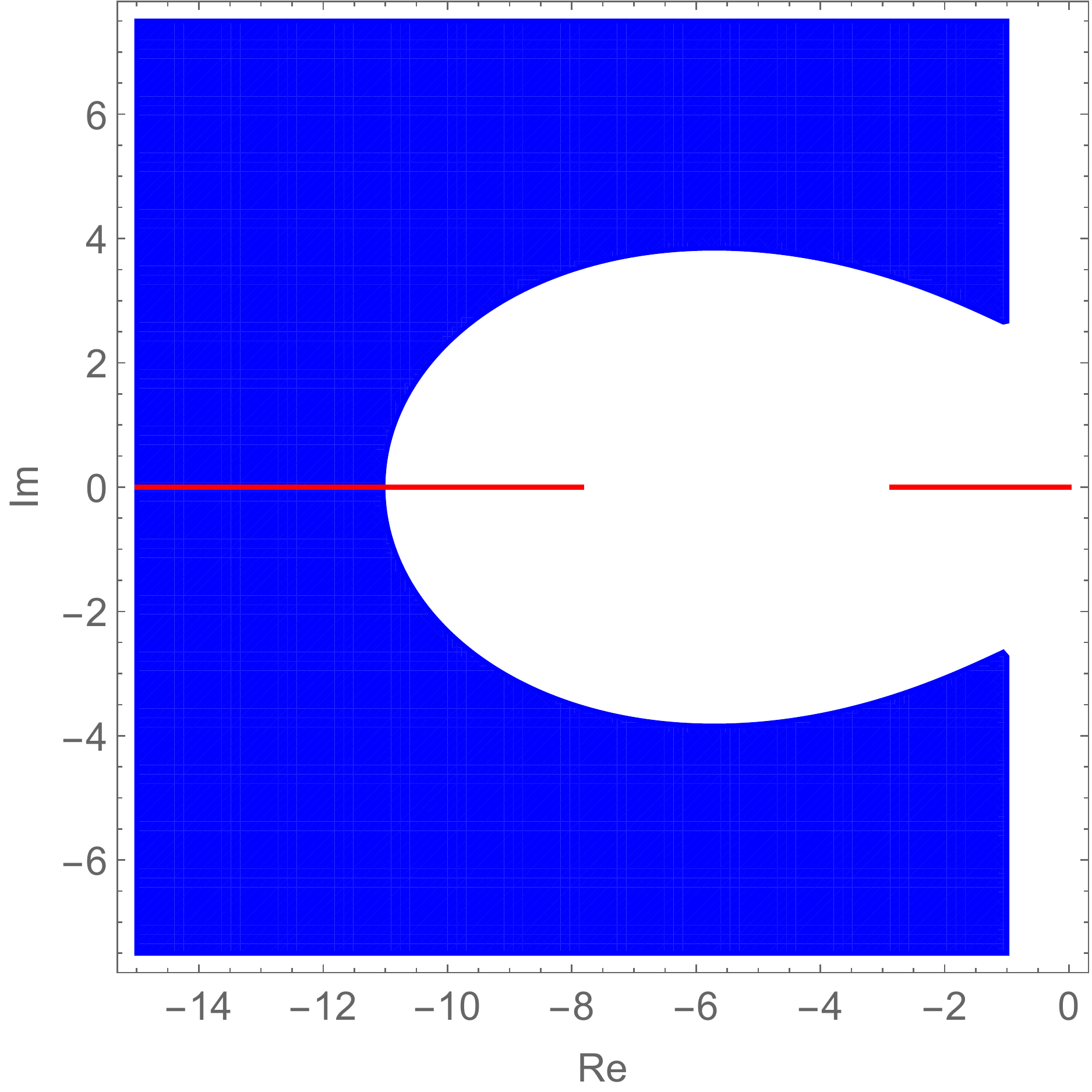}} %\\[2mm]
	\caption{{\small 
	\! Enclosures for $\sigma(T) \!\setminus\! \R$ in Theorem~\ref{thm:pencil.spec.incl}~(i) (blue) and 
	for $\sigma(T)\cap\R$ in Theorem~\ref{thm:pencil.spec.incl}~(ii)-(v) (red in (a),\,(c), empty~in~(b)).}}
\vspace{-4mm}	
	\label{fig:dwe.spec.incl}
\end{figure}

\vspace{-1mm}	

\begin{rem}
		If \eqref{eq:pencil.subordinate} holds with $p=0$, then $\forma$ is bounded and $\norm{\forma} \le C_p= C_0$. In this case, 
		the spectrum of $T$ lies in a strip to the left of the imaginary axis; more precisely,
		the non-real spectrum of $T$ \vspace{-0.5mm} satisfies
		\begin{equation}
			\sigma(T)\setminus\R \subseteq \set{z \in \C}{-C_0 \le \re z \le -\alpha_0, \, |z| \ge \sqrt{\kappa_0} },
		\end{equation}
	while the real spectrum \vspace{-0.5mm}  satisfies
	\begin{equation}
		\sigma (T) \cap \R  \begin{cases}
			= \emptyset & \quad {\rm if } \,\, C_0^2 < \kappa_0, \\
			\subseteq [-C_0 - \sqrt{C_0^2-\kappa_0}, -C_0 + \sqrt{C_0^2-\kappa_0}] & \quad {\rm if } \,\, C_0^2 \ge  \kappa_0;
		\end{cases}
	\end{equation}
	notice that the latter corresponds to Theorem~\ref{thm:pencil.spec.incl} (iii) with $p=0$.
\end{rem}

\begin{proof}[Proof of Theorem {\rm \ref{thm:pencil.spec.incl}}]
	Clearly, $\formt$ is holomorphic. For arbitrary $\varepsilon>0$, applying Young's inequality to \eqref{eq:pencil.subordinate}, we obtain
	\begin{equation}
	\label{eq:pencil.rel.bdd}
		\begin{aligned}
			\forma[f] & \le \left(\frac{\varepsilon}{p}\right)^p \!\!\! \big((\formt_0-\kappa)[f]\big)^p \left(\frac{p}{\varepsilon}\right)^p C_{p}\big(\norm{f}^2\big)^{1-p}\\[-1mm]
			& \le \varepsilon  \big((\formt_0-\kappa)[f]\big)+ (1\!-\!p)\left(\frac{p}{\varepsilon}\right)^\frac{p}{1-p} C_p^\frac{1}{1-p} \norm{f}^2
	\end{aligned}
	\end{equation}
	for all $f\in\dom\formt_0$, i.e.\ $\forma$ is $\formt_0$-bounded with relative bound $0$. Hence, for each $\lambda\in\C$, the form $\formt(\lambda)$ is densely defined, sectorial and closed, see e.g.\ \cite[Thm.\ VI.1.33]{Kato-1995}. This shows that $\formt$ is a holomorphic family of type (a). Since all enclosing sets in Theorem \ref{thm:pencil.spec.incl} are closed and 
	\[
		\sigma(T) \subseteq W_{\Psi} (T)=W_{\Psi} (\formt) = \overline{W(\formt)}
	\]
	by Theorem \ref{thm:pseudo.dense.hol.fam} with $k=2$ and $\mu\in\C$ arbitrary, it suffices to show that $W(\formt)\setminus \R$ and $W(\formt)\cap \R$ satisfy the claimed enclosures.
	
	Let $\lambda_0\in W(\formt)$, i.e.\ there exists $f\in\dom\formt_0$, $\norm{f}=1$, with $\formt(\lambda_0)[f]=0$. Taking real and imaginary part in this equation, we conclude that
	\begin{align}
	\label{eq:pencil.re}
		\formt_0[f]+2\re\lambda_0\,\forma[f]+(\re\lambda_0)^2-(\im\lambda_0)^2 & =0, \\
	\label{eq:pencil.im}
		2\im\lambda_0\,\forma[f]+2\re\lambda_0\im\lambda_0 & =0.
	\end{align}
	First assume that ${\lambda_0 \!\in}W(\formt)\setminus\R$. %Then $\forma[f]^2\!-\formt_0[f] \!<\! 0$. 
	Dividing \eqref{eq:pencil.im} by $2\im\lambda_0$ ($\ne 0$) and inserting this into \eqref{eq:pencil.re}, we find
  	\begin{align}
		&\re\lambda_0=-\forma[f]\le-\alpha_0\le 0,  \label{eq:pencil.re.num.ineq}\\
		&\abs{\lambda_0}^2=(\im\lambda_0)^2+(\re\lambda_0)^2=\formt_0[f] \ge \kappa_0. 
		\label{eq:pencil.im.num.ineq-0}
	\end{align}
	Using these relations and assumption \eqref{eq:pencil.subordinate}, we can further \vspace{-1mm} estimate
	\begin{equation}
	\label{eq:pencil.im.num.ineq}
		(\im\lambda_0)^2 =\formt_0[f] - |\re \lambda_0|^2 \ge \max\{ 0, {C_p^{-\frac 1p} } \abs{\re\lambda_0}^\frac{1}{p} - |\re \lambda_0|^2 +\kappa \},
	\end{equation}
	and hence $\lambda_0 \!\in W(\formt)\!\setminus\!\R$ satisfies all three claimed inequalities in (i).

	Now assume that $\lambda_0 \!\in \!W(\formt)\cap\R$. Then $\forma[f]^2\!-\!\formt_0[f]\!\ge\!0$ and thus, in par\-ti\-cular, $\forma[f] \!\ge\! \max\{\alpha_0,\sqrt{\kappa_0}\}$. Moreover, since $\im\lambda_0\!=\!0$, equality \eqref{eq:pencil.im}~trivially holds and \eqref{eq:pencil.re} implies 
%	$\lambda_0\!=\!-\forma[f]\pm\sqrt{\forma[f]^2\!-\!\formt_0[f]} \le 0$ 
\begin{equation}
	\label{eq:real-Wt}
		%\formt_0[f] \!\in\! D_{\le 0}, \quad 
		\lambda_0 = -\forma[f]\pm\sqrt{\forma[f]^2-\formt_0[f]}\le 0
	\end{equation}	
	because~$\formt_0 \!\ge\! 0$. This, together with $\forma \!\ge\! \alpha_0$ and assumption \eqref{eq:pencil.subordinate}, yields \vspace{1mm} that
	\begin{equation}
	\label{eq:ineq-new}
		\max\big\{\alpha_0^2,\kappa_0\big\} \le \max\{ \alpha_0^2, \formt_0[f]\} \le \forma[f]^2 \le  C_p^2 \big((\formt_0-\kappa)[f]\big)^{2p}. 
	\end{equation}
	If we define  
	\begin{align*} 
	   &d(x):= C_p^{-\frac 1p} x^{\frac 1{2p}} \!-\!x \!+\! \kappa, \quad x\!\in\! [0,\infty),\quad
	   D_{\leq 0} := \big\{ x \!\in\! [\kappa_0,\infty): d(x) \leq 0\big\},
	\end{align*}
	then it is easy to see that $\formt_0[f] \!\in\! D_{\le 0}$; 
	%\begin{equation}
	%\label{eq:real-Wt}
		%\formt_0[f] \!\in\! D_{\le 0}, \quad \lambda_0 = -\forma[f]\pm\sqrt{\forma[f]^2-\formt_0[f]};
	%\end{equation}
	in particular, $D_{\le 0} = \emptyset$ implies $W(\formt) \cap \R =\emptyset$. An elementary analysis shows that $d$ is either identically zero, has  no zero, one simple zero or two (possibly coinciding) zeros on $[0,\infty)$,	which we denote by $x_+$ and $x_-\le x_+$, respectively, if they exist.~Then	
	\begin{align}
    \label{eq:p.cases.1} 
   		p < &  
    	\frac 12  \mbox{ or } p = \frac12,  C_\frac12 < 1 \mbox{ or } p = \frac12, C_\frac12 = 1, \kappa>0 \\[1mm]
    	& \implies 	D_{\le 0} \!=\! \emptyset \mbox { or } D_{\le 0}  \mbox{ is bounded}, \ D_{\le 0} \!=\! [ \kappa_0,x_+] \mbox{ or } D_{\le 0}\!=\![x_-,x_+],%\\
\intertext{}			
    \label{eq:p.cases.2}
    	p > &\frac 12  \mbox{ or }  p = \frac12, C_\frac12 > 1 \mbox{ or } p = \frac12, C_\frac12 = 1, \kappa\le 0   \\[1mm]
    	& \implies D_{\le 0} \!\ne\! \emptyset \mbox{ is unbounded},\ D_{\le 0}\!=\![\kappa_0,\infty)  \mbox{  or } D_{\le 0} \!=\! [x_+, \infty)\\
    	& \hspace{5.75cm}
			 \mbox{ or }  D_{\le 0}\!=\![\kappa_0,x_-]\!\cup\![x_+,\infty).
   	\end{align}
	Which case prevails for fixed $p \!\in\! [0,1)$ can be characterised by means of in\-equalities involving the constants $\kappa_0$, $\kappa$ and $C_p$. For~estimating $\lambda_0$ in \eqref{eq:real-Wt} while respecting the restrictions in \eqref{eq:ineq-new}, we consider the functions 
	\[ 
		f_\pm(s,t):= -s \pm \sqrt{s^2 \!-\! t}, \quad s\!\in\![\alpha_0,\infty), \ t\!\in\! [\kappa_0,\infty), \ t \!\le\! s^2 \!\le\! C_p^2(t-\kappa)^{2p}.
	\]
	It is easy to check that $f_+$ is monotonically	increasing in $s$ and monotonically	decreasing in $t$, while $f_-$ is monotonically decreasing in $s$ and monotonically	increasing in $t$ and hence, since $s \le C_p(t-\kappa)^{p}$,
	\begin{equation}
	\label{eq:gpm}
		\begin{aligned}
			f_+(s,t) & \le f_+(C_p(t-\kappa)^p,t) \eqdef g_+(t), \\
			f_-(s,t) & \ge f_-(C_p(t-\kappa)^p,t) \eqdef g_-(t).%;
		\end{aligned}
	\end{equation}
	%\bgcomm{removed monotonicity properties of $g_\pm$; I did some plots/computations and they were not true in general; e.g.\ for $p=0.19$, $C_p=2.51$, $\kappa = 6.06$, $\kappa_0 = 7$ we %have $8 \in D_{\le 0}$ and $g_+'(8) >0$}
    %note that $g_\pm$ need not be monotonic, but $g_+$ is monotonically decreasing if $p\!<\!\frac 12$ and $g_-$ is monotonically decreasing if $p\!>\!\frac 12$.
    
    Now we distinguish the two qualitatively different cases \eqref{eq:p.cases.1} and \eqref{eq:p.cases.2}. 
		To obtain the claimed enclosures for $W(\formt)\cap \R$, we	use \eqref{eq:ineq-new}, 
		\eqref{eq:real-Wt} and \eqref{eq:gpm} to conclude that $g_-(t) \leq \lambda_0 \leq g_+(t)$ for some $t \in D_{\leq 0}$. 
	  \\ 
		If \eqref{eq:p.cases.1} holds, there are the following two possibilities: \\[1mm]
    (1) If $d$ has no zeros on $[0,\infty)$ or if $d$ has at least one zero and $x_+ \!<\!\kappa_0$,~then $D_{\le0} =\emptyset$ and \vspace{-2mm} thus
    \[
    	W(\formt)\cap \R = \emptyset.
    \] 
	(2) If $d$ has at least	one zero $x_+$ and $x_+ \ge \kappa_0$, then $D_{\le0} %=[\max\{\kappa_0,x_-\},x_+]
	$ is one bounded interval and
	%\bgcomm{$\max\{\kappa_0, x_-\}$ not well-defined if $d$ has only one zero; we specify $D_{\le0}$ below anyway}
	\begin{align*}
	\label{eq:spm}
		W(\formt)\!\cap \R \subseteq \!\big[s^-\!, s^+\big], \quad 
		& s^- \!\defeq\!\! \min_{t\in D_{\le0} %[\max\{\kappa_0,x_-\},x_+]
		}g_-(t), \quad  s^+ \!\defeq\!\! \max_{t\in D_{\le0} %[\max\{\kappa_0,x_-\},x_+]
		}g_+(t); %\,= g_+(\max\{\kappa_0,x_-\});
	\end{align*}
	here if $d$ has only one zero $x_+$ or if $d$ has two zeros $x_\pm$ and $x_-<\kappa_0$, then $D_{\le0} = [\kappa_0,x_+]$ and if $d$ has two zeros and $x_-\ge\kappa_0$, then \vspace{1mm} $D_{\le0} =[x_-,x_+]$.

	\noindent
	If \eqref{eq:p.cases.2} holds, there are the following two possibilities: \\[1mm]
	(3) If $d$ has two zeros $x_\pm$  on $[0,\infty)$ and $x_- \!\ge\! \kappa_0$, then $D_{\le0}\!=\! [\kappa_0,x_-] \cup [x_+,\infty)$ and we obtain
	\begin{align*}
		W(\formt)\cap \R \!\subseteq \! \big(\!-\!\infty, r^+	\big] \! \cup \! 
		\big[ s^-\!,s^+ \big], \ \ & r^+ \!\defeq\!\! 
		\max_{t \in [x_+,\infty)} g_+(t), \ s^+ \!\defeq\!\!\! \max_{t\in [\kappa_0,x_-]}g_+(t), \\[-1mm]
		& s^- \!\defeq\!\!\! \min_{t\in [\kappa_0,x_-]}g_-(t) %\!=\!g_-(x_-)
		;
	\end{align*}
	here $g_+$ attains a maximum on $[x_+,\infty)$ since $g_+(t)$ tends to $-\infty$ as $t\to\infty$, and analogously in the next case.
	
	\noindent
	(4) If $d$ has either at most one zero $x_+$ or two zeros $x_\pm$ on $[0,\infty)$ and $x_- < \kappa_0$, then $D_{\le0} = [\max\{\kappa_0, x_+\},\infty)$ and we conclude that 
	\[
		W(\formt)\cap \R \subseteq \big(-\infty, s^+\big], \quad s^+ \!\defeq\! \max_{t \in [\max\{\kappa_0, x_+\},\infty)} g_+(t). 
		\vspace{-2mm}
	\]
	This proves claim (ii). 
	
    Claim (iv) for $\kappa\!=\!0$ and $p\!=\! \frac 12$ follows from cases (1), (2) and (4) above if we note that then  $d(x)=(C_{\frac 12}^{-2}\!-\!1)x$, $x\!\in\![0,\infty)$, is either identically \vspace{-2mm} zero or
     has the only zero $x_+\!=\!0$   and, for case (4), $g_+(t)\!=\!-t^{\frac 12} \big(C_{\frac 12}\!+\! \sqrt{C_{\frac 12}^2\!-\!1}\big)$ \vspace{-2mm}  is montonically decreasing 
    so that $s^+=g_+(\kappa_0)$.
    
	Finally, if $\kappa\!=\!0$ and $p\!\ne\! \frac 12$,  the function $d$ has the two zeros \vspace{-1.5mm} $x_-\!=\!0$ and $x_+ \!=\! (C_p^2)^{\frac 1{1-2p}}$ on $[0,\infty)$, and the respective bounds $r^+$, $s^\pm$ above can be determined explicitly to deduce claims (iii) and (v). More precisely, claim (iii) follows from cases (1) and (2) if we note that, in (2), %\vspace{-1mm} 
	$D_{\le0} = [\kappa_0, x_+]$, $g_+$ is monotonically decreasing on $[0, x_+]$
	%\bgcomm{$g_+$ indeed decreasing for $\kappa=0$} 
	and $g_-$ %in (2) 
	attains its minimum \vspace{-1mm} on $[0,x_+]$ at $t=\big( 4C_p^2p(1\!-\!p) \big)^{-\frac 1{2p-1}}$. Claim (v) follows from cases (4) if $\kappa_0>0$ and (3) if $\kappa_0=0$; note that, for $\kappa\!=\!0$,  case (3) where $p\!>\!1/2$ can only occur if $\kappa_0\!=\!0$. In both cases, we use \vspace{-1mm} that $g_+$ attains its maximum on $[x_+ %0
	,\infty)$ at $t=\big( 4C_p^2p(1\!-\!p) \big)^{-\frac 1{2p-1}}$.
\end{proof}

\begin{rem}
If \eqref{eq:pencil.subordinate} holds with $\kappa \le \kappa_0$ and $p\in[0,1)$, then it holds for every $q\in(p,1)$ with $\kappa_1 \le \kappa$ such that $\kappa_1 < \kappa_0$.

Indeed, then $\formt_0\!-\!\kappa\! \le\! \formt_0\!-\!\kappa_1$ and $\formt_0\! -\!\kappa_1 \!\ge\! \kappa_0\!-\!\kappa_1 \!>\!0$ which implies that $(\|f\|^2)^{q-p} \!\le\! (\kappa_0\!-\!\kappa_1)^{p-q} \big((\formt_0\!-\!\kappa_1)[f] \big)^{q-p}$\!\!, 
$f\!\in\! \dom \formt_0$. Hence \eqref{eq:pencil.subordinate} holds with $q$, $\kappa_1$ and $C_q=C_p (\kappa_0\!-\!\kappa_1)^{p-q}$. 
\end{rem}

\begin{rem}
\label{Jacob-Trunk}
	As a special case of Theorem \ref{thm:pencil.spec.incl} we obtain the enclosure for the non-real spectrum proved in \cite[Thm.\ 3.2,~Part 5]{Jacob-Trunk-2009} (where the damping was only assumed to be accretive) and we considerably improve the enclosure for the real spectrum therein since we obtain that the latter is, in fact, empty. The assumption in \cite[Thm.\ 3.2,~Part 5]{Jacob-Trunk-2009} is~that   
	\begin{equation}
	\label{eq:JT09}
 		\nu:= \sup_{f\in\dom\formt_0\setminus\{0\}} \frac{2\forma[f]}{\formt_0[f]^{1/2}\|f\|} \in (0,2).
	\end{equation}
	The parameters $a_0$, $\beta$ and $\nu$ in \cite[(5) and p.\ 83]{Jacob-Trunk-2009} correspond to the following special choices in Theorem~\ref{thm:pencil.spec.incl} and assumption \eqref{eq:pencil.subordinate}:
	\[
 		p=\frac 12, \quad C_{\frac 12} = \frac \nu 2, \quad \kappa=0, \quad \kappa_0 = a_0^2 >0, \quad \alpha_0 = \frac \beta 2.
	\]
	Under the assumption \eqref{eq:JT09} made in \cite[Thm.\ 3.2,~Part 5]{Jacob-Trunk-2009}, Theorem~\ref{thm:pencil.spec.incl}~(i) yields the spectral \vspace{-2mm} enclosure
	\[
		\sigma(T) \setminus \R \subseteq \! \bigg\{ z\!\in\C: \re z\!\le\!-\frac \beta 2, \, \abs{z} \!\ge\! a_0, \, \, \abs{\im z}\!\ge\! \sqrt{\frac 4{\nu^2}\!-\!1\,}\abs{\re z} \bigg\}.
	\]
	This enclosure is the same as in \cite[Thm.\ 3.2,~Part 5]{Jacob-Trunk-2009}. However, since $\nu\!<\!2$ is equivalent to $C_{\frac 12} \!<\!1$, the enclosure $\sigma(T)\cap \R \subseteq (-\infty,-\frac {a_0}\nu-\frac{4a_0}{\nu^3}]$ in \cite[Thm.\ 3.2,~Part 5]{Jacob-Trunk-2009} is considerably improved by Theorem~\ref{thm:pencil.spec.incl} (iv) to
	\[\sigma(T)\cap \R = \emptyset.\]
\end{rem}

%\marginpar{\ct{\footnotesize $\sigma(T) \cap \R$ does not look empty in (b)!} \\
%	\bg{\footnotesize inserted white line to indicate resovent set; can be made thicker if required}}

\begin{rem}
	In the second case in Theorem~\ref{thm:pencil.spec.incl}~(ii), i.e.\ if $p \!>\! \frac 12$ or $p \!=\! \frac12$, $C_\frac12 \!>\! 1$ or $p \!=\! \frac12$, $C_\frac12 \!=\! 1$, $\kappa\!\le\! 0$, the set  $W(\formt)\cap(-\infty,0]$ used to enclose the spectrum can, indeed, be unbounded if so is $\formt_0$. 
	
	In fact, if $W(\formt_0) \!=\! [\kappa_0,\infty)$, we can choose $\forma\!=\!C_p(\formt_0-\kappa)^p$.  Then there exist $f_n\!\in\!\dom\formt_0$, $\norm{f_n}\!=\!1$, with $\formt_0[f_{n}] \!\ge\! n$ for $n\!\in\!\N$.
	%\bgcomm{removed repetition of "then"}
	T%hen t
	he conditions on~$p$, $C_p$ and $\kappa$ ensure, comp.~\eqref{eq:p.cases.2}, that $C_p^2(\formt_0[f_{n}]-\kappa)^{2p}-\formt_0[f_{n}]\ge0$ for sufficiently large $n\in\!\N$ and thus
	\begin{equation}
	W(\formt)\cap(-\infty,0]\ni\lambda_0\!=\!-\formt_0[f_n]^p\!-\!\sqrt{\formt_0[f_n]^{2p}\!-\!\formt_0[f_n]}\le-\formt_0[f_n]^p \le -n^p \to -\infty,
	\vspace{-1mm}
	\end{equation}
	and hence $\inf \left(W(\formt)\cap(-\infty,0]\right)=-\infty$.
\end{rem}

In the next example we apply Theorem \ref{thm:pencil.spec.incl} to linearly damped wave equations with possibly unbounded and/or singular damping. 

\begin{exple}
\label{ex:damped.wave.PDE}
	Let $\Hh=L^2(\Rd)$ with $d\ge3$ and $a$, $q\in\Loneloc(\Rd)$, $a \neq 0$ and
	$a,q\ge0$ almost everywhere. If $\dom a^\frac{1}{2}$ and $\dom q^\frac{1}{2}$ denote the maximal domains of the multi\-plication operators $a^\frac{1}{2}$ and $q^\frac{1}{2}$ in $L^2(\Rd)$, respectively, we define the quadratic forms $\forma$ and $\formt_0$ in $L^2(\Rd)$ \vspace{-1mm}by
	\begin{equation}
		\begin{aligned}
			\forma[f] & \defeq\int_{\Rd}a\abs{f}^2\d x, \quad & \dom\forma& \defeq\dom a^\frac{1}{2}, \\
			\formt_0[f] & \defeq\int_{\Rd}\abs{\nabla f}^2\d x+\int_{\Rd}q\abs{f}^2\d x, \qquad  & \dom\formt_0 & \defeq H^1(\Rd)\cap\dom q^\frac{1}{2}.
		\end{aligned}
	\end{equation}
	Suppose that, for almost all $x\in\Rd$,
	\begin{equation}
	\label{eq:dwe.pot.damp.inequ}
	a(x)\le\sum_{j=1}^n\abs{x-x_j}^{-t}+ u(x) + v(x), \qquad v(x) \le c_1 q(x)^r+c_2,
	\end{equation}
	where $u \in L^s (\Rd)$ with $s>d/2$, $v\!\in\!\Loneloc(\Rd)$, $t\!\in\![0,2)$, $n\!\in\!\N_{0}$, $x_j\!\in\!\R^d$ for $j\!=\!1,\dotsc,n$, $c_1$, $c_2\!\ge\!0$ and $r\!\in\![0,1)$. Then $\forma$, $\formt_0$ are closed, $\forma$, $\formt_0\ge 0$ and, without further assumptions, we only know that $\alpha_0 \ge 0$, $\kappa_0 \ge 0$ in Theorem~\ref{thm:pencil.spec.incl}. In order to verify \eqref{eq:pencil.subordinate}, let $f\in\dom\formt_0$ with $\norm{f}=1$. By H\"older's and Hardy's inequality, for $1\le j\le n$,
	\begin{equation}
	\label{eq:dwe.Hardy}
		\begin{aligned}
			\int_{\Rd}\!\!\abs{x\!-\!x_j}^{-t}\abs{f}^2 \d x & \le \left(\int_{\Rd}\!\!\abs{x\!-\!x_j}^{-2}\abs{f}^2\d x\right)^\frac{t}{2} \!\le\! \frac{2^t}{(d-2)^t}\norm{\nabla f}^t.
		\end{aligned}
	\end{equation}
	Moreover, by Gagliardo-Nirenberg-Sobolev's inequality, there exists a constant $G_d>0$ depending only on the dimension $d$ such that
	\begin{equation}
		\norm{f}_{L^{2^*}(\Rd)} \le G_d \norm{\nabla f}, \quad f\in H^1 (\Rd), \quad 2^*\defeq \frac{2d}{d-2},
	\end{equation}
	where $2^*\!>\!2$ is the critical Sobolev exponent for the embedding $H^1 (\Rd) \!\hookrightarrow\! L^{2^*}(\Rd)$. Since $d/s \!\in\! (0,2)$, we can use H\"older's inequality with three terms to estimate
	\begin{equation}
		\int_{\Rd} u |f|^2 \d x \le \norm{u}_{L^s(\Rd)} \left(\int_{\Rd} |f|^{\frac ds \frac{2s}{d-2}} \d x\right)^\frac{d-2}{2s} \left(\int_{\Rd} |f|^{\left(2-\frac ds\right)\frac{2s}{2s-d}} \d x\right)^\frac{2s-d}{2s}.
	\end{equation}
	This inequality, together with the relations
	\begin{equation}
		\frac ds \frac{2s}{d-2} = 2^*, \quad \frac{d-2}{2s} = \frac{d}{2^*s}, \quad \left(2-\frac ds\right)\frac{2s}{2s-d} = 2,
	\end{equation}
	and $\norm{f}=1$, yields that
	\begin{equation}
	\label{eq:Sobolev}
		\int_{\Rd} u |f|^2 \d x \le \norm{u}_{L^s(\Rd)} \norm{f}_{L^{2^*} (\Rd)}^\frac ds \le \norm{u}_{L^s(\Rd)}  G_d^\frac ds \norm{\nabla f}^\frac ds.
	\end{equation}
	Next the bound on $v$ in \eqref{eq:dwe.pot.damp.inequ} with $r\in[%(
	0,1)$, H\"older's inequality with $1/r\in (1,\infty]$, $1/(1-r)\in [1,\infty) %>1
	$ and $\norm{f}=1$ give
	%\bgcomm{remains valid for $r=0$ with $1/0 := \infty$}
	\begin{equation}
	\label{eq:dwe.v.subord}
		\int_{\Rd} v |f|^2 \d x \le c_1\int_{\Rd}q^r\abs{f}^2\d x + c_2 \le c_1\left(\int_{\Rd}q\abs{f}^2\d x\right)^r + c_2.
	\end{equation}
	Combining the inequalities \eqref{eq:dwe.Hardy}, \eqref{eq:Sobolev} and \eqref{eq:dwe.v.subord}, we arrive at
	\begin{equation}
	\label{eq:dwe.damp.subord}
		\begin{aligned}
			\hspace{-2.5mm}\forma[f] \!&\le\! \frac{n2^t}{(d\!-\!2)^t}\norm{\nabla f}^t  \!\!+\! \norm{u}_{L^s(\Rd)}  G_d^\frac ds \norm{\nabla f}^\frac ds \!+\! c_1\left(\int_{\Rd} \!q\abs{f}^2\d x\right)^{\!r}\!\!\!+\!c_2  \hspace{-4mm}
			\\
		&=: \alpha_1 (\norm{\nabla f}^2)^{\frac t2} \!\!+\! \alpha_2 (\norm{\nabla f}^2)^{\frac d{2s}} \!+\! \alpha_3 \left(\int_{\Rd} \!q\abs{f}^2\d x\right)^{r} \!+\! \alpha_4.	
		\end{aligned}
	\end{equation}
	In order to further bound \eqref{eq:dwe.damp.subord}, we estimate $\alpha_1 x_1^{p_1} \!\!+\! \alpha_2 x_2^{p_2} \!+\! \alpha_3 x_3^{p_3} \!+\! \alpha_4$~with $x_i \!\ge\! 0$, $p_i \!\in\! [0,1)$, $i\!=\!1,2,3$, and $\alpha_i \!\ge\! 0$, $i\!=\!1,2,3,4$; note that $x_1\!=\!x_2=\|\nabla f\|^2$ in~\eqref{eq:dwe.damp.subord}. If we set $p \!\defeq\! \max \{p_1,p_2,p_3\}$ and maximise $\delta(x) \!\defeq\! x^{p_i} \!-\! x^p\!$,  $x\in [0,1]$, $i\!=\!1,2,3$, we \vspace{-2mm}find~that
    \begin{equation}
	\label{eq:delta}
    x_i^{p_i} \!\le\! x_i^{p} \!+\!\delta_i, \quad 		\delta_i \defeq \begin{cases}
			\hspace{8mm} 0 & \mbox{if } p_i = p, \\
			\frac{p-p_i}{p} \left(\frac{p_i}{p}\right)^\frac{p_i}{p-p_i} & \mbox{if } p_i < p,
		\end{cases}
	\quad i=1,2,3.	
	\end{equation}
	If $\max \{\alpha_1, \alpha_2, \alpha_3\} \neq 0$, then
	\begin{equation}
	\label{eq:gamma.p}
		\gamma_p \defeq \alpha_1 (1 \!+\! \delta_1) \!+\! \alpha_2 (1 \!+\! \delta_2) \!+\! \alpha_3 (1 \!+\! \delta_3) \!+\! \alpha_4 \neq0.
	\end{equation}
	If we use \eqref{eq:delta}, the concavity of  $x\!\mapsto\! x^{p}$ on $[0,\infty)$ and $x_1\!=\!x_2$, %\vspace{-3mm}
	we~obtain
	\begin{align*}
		& \alpha_1 x_1^{p_1} \!\!+\! \alpha_2 x_2^{p_2} \!+\! \alpha_3 x_3^{p_3} \!+\! \alpha_4
		\, \le \,\alpha_1 (x_1^p \!+\! \delta_1) \!\!+\! \alpha_2 (x_2^{p} 	\!+\!\delta_2 ) \!+\! \alpha_3 (x_3^p \!+\! \delta_3) \!+\! \alpha_4 \\[1mm]
		&= \gamma_p \Big( \frac{\alpha_1}{\gamma_p 	} x_1^{p} 	\!\!+\! \frac{\alpha_2}{\gamma_p } x_2^p + \frac{\alpha_3}{\gamma_p} x_3^p\!\!+\! \frac{\alpha_1\delta_1 + \alpha_2 \delta_2\!+\!\alpha_3\delta_3 + \alpha_4}{\gamma_p } \Big) \\
		&\le \gamma_p 	\Big( \frac{\alpha_1}{\gamma_p } x_1 \!+\! \frac{\alpha_2}{\gamma_p } x_2 + \frac{\alpha_3}{\gamma_p} x_3 \!+\! \frac{\alpha_1\delta_1 + \alpha_2  \delta_2 \!+\!\alpha_3 \delta_3 + \alpha_4}{\gamma_p } \Big)^{p} \\[1mm]
		&= \gamma_p ^{1-p} \big( (\alpha_1 + \alpha_2) x_1 \!+\! \alpha_3 x_3 + \alpha_1\delta_1 + \alpha_2 \delta_2 \!+\!\alpha_3 \delta_3 + \alpha_4 \big)^{p} \\[1mm]
		&\le \gamma_p ^{1-p} \max\{ \alpha_1 +\alpha_2, \alpha_3\}^{p} \Big( x_1 \!+\! x_3 \!+\! \frac{\alpha_1\delta_1 + \alpha_2 \delta_2 \!+\!\alpha_3 \delta_3+\alpha_4}{\max\{ \alpha_1+\alpha_2, \alpha_3\}} \Big)^{p}.  
	\end{align*}
    If $\max\{n,\norm{u}_{L^s(\Rd)},c_1\} \!\neq\! 0$, we can apply this estimate to \eqref{eq:dwe.damp.subord}  with $p_1 \!=\! t/2$, $p_2 \!=\! d/(2s)$, $p_3 \!=\! r$, $\delta_i$, $i\!=\!1,2,3$, as in \eqref{eq:delta} to obtain that $\dom\formt_0\subseteq\dom\forma$ and assumption \eqref{eq:pencil.subordinate} holds with the %\vspace{-2mm}
    parameters
	\begin{equation}
	\label{vne0}
		\begin{aligned}
			&p\!=\!\max\Big\{\frac{t}{2},\frac{d}{2s}, r\!\Big\}, \quad
			C_p \!=\! \gamma_p^{1-p} \max\Big\{ \frac{n2^t}{(d\!-\!2)^t} + \norm{u}_{L^s(\Rd)} G_d^{\frac ds},c_1 \Big\}^{\!p}\!\!, \\
			&\kappa \!=\!- \frac{ n2^t \delta_1 + (d-2)^t (\norm{u}_{L^s(\Rd)} G_{d}^{\frac ds} \delta_2 + c_1 \delta_3 + c_2)}{\max\{ n2^t \!+\! (d\!-\!2)^t \norm{u}_{L^s(\Rd)} G_d^{\frac ds},\,(d-2)^t c_1\}},
		\end{aligned}
	\end{equation}
	where, according to \eqref{eq:gamma.p},
	\begin{equation}
			\gamma_p = \frac{n2^t}{(d-2)^t} (1 \!+\! \delta_1) \!+\! \norm{u}_{L^s(\Rd)} G_d^\frac ds (1 \!+\! \delta_2) \!+\! c_1 (1 \!+\! \delta_3) \!+\! c_2.
	\end{equation}
	If $\max\{n,\norm{u}_{L^s(\Rd)},c_1\} \!=\! 0$, i.e.~$n=0$, $u\equiv0$ and $c_1=0$, then the damping $a=v$ is bounded, our assumption $a\neq 0$ implies $c_2>0$ and \eqref{eq:pencil.subordinate} trivially holds with $p=0$, $C_0=c_2 =\|a\|_\infty$ and $\kappa \le d=\kappa_0$ arbitrary.
	
	The constants in \eqref{vne0} in the general case $\max\{n,\norm{u}_{L^s(\Rd)},c_1\} \!\neq\! 0$ simplify substantially
	if either $n\!=\!0$, $u\!\equiv\! 0$ or $v \!\equiv\! 0$. If e.g.~two of $n$, $u$ or $v$ vanish, the constants $p$, $C_p$ and $\kappa$, which may be read off from \eqref{eq:dwe.Hardy}, \eqref{eq:Sobolev} or \eqref{eq:dwe.v.subord},  are also obtained as special cases of \eqref{vne0}. For \vspace{-1.5mm} instance, 
	\begin{alignat*}{4}
		&p=\frac t2, \quad &&C_{\frac t2} = \frac{n2^t}{(d-2)^t}, \quad &&\kappa = 0
		&&\quad \mbox{if $n\neq 0$, $u\equiv0$ and $v\equiv0$},\\[-1mm]
		&p=\frac d{2s}, \ &&C_\frac{d}{2s} = \norm{u}_{L^s(\Rd)} G_d^\frac ds, \ &&\kappa = 0
		&&\quad \mbox{if $n= 0$, $u\not\equiv0$ and $v\equiv0$},\\
		&p=r, \quad &&C_r = (c_1\!+\!c_2)^{1-r}  c_1^r, \quad &&\kappa = -\frac{c_2}{c_1}
		&&\quad \mbox{if $n= 0$, $u\equiv0$ and $v\not\equiv0$, $c_1\!>\!0$};
	\end{alignat*}
	in \eqref{vne0} these are the 3 cases $\delta_1 = 0$ with $c_1=c_2=r=0$ and $s$ sufficiently large such that $d/(2s)<r$,
	$\delta_2 = 0$ with $t=c_1=c_2=r=0$, and $\delta_3=0$ with $t=0$ and $s$ sufficiently large, respectively. The cases where only one of $n$, $u$ or $v$ vanishes are similar and are left to the reader.	

	As a special case, we \vspace{-1.5mm} consider 
	\begin{equation}
		a(x)=\abs{x}^{k} \ \mbox{with } k \in[0,2), \quad q(x)=\abs{x}^2, \quad x\in\Rd.
	\vspace{-1.5mm}	
	\end{equation}
	Here $\alpha_0\!=\!0$ and  we can choose $\kappa_0 >0$ as the ground energy of the harmonic oscillator, cf. \cite[Sec.\ XIII.12]{Reed-Simon-1978}, \vspace{-1mm}i.e.\
	\begin{equation}
	\kappa_{0}=\inf_{f\in \dom\formt_0}\frac{\formt_0[f]}{\norm{f}^2}=\frac{\formt_0[f_0]}{\norm{f_0}^2}=d,
	\vspace{-1mm}
	\end{equation}
	where $f_0(x)=\exp(-\abs{x}^2\!/2)$, $x\in\Rd$, is the (non-normalised) ground state of the harmonic oscillator.
	Moreover, in this special case $a$ satisfies \vspace{-1.5mm} \eqref{eq:dwe.pot.damp.inequ}~with
	\[
		n\!=\!0, \quad t\!=\!0, \quad u \equiv 0, \quad v \equiv a, \quad r = \frac{k} 2, \quad c_1 = 1, \quad c_2=0,
	\vspace{-1mm}	
	\]
	and by what was shown above, condition \eqref{eq:pencil.subordinate} holds\vspace{-1mm}  with
	\[
  		p=\frac k 2, \quad C_p=1, \quad \kappa=0.
	\vspace{-1mm}
	\]
	Hence the results in Theorem \ref{thm:pencil.spec.incl} (iii), (iv) and (v) yield \vspace{-1mm} that
	\begin{equation}
	\label{eq:dwe.comp.incl}
		\sigma(T) \setminus\R\subseteq\Big\{z\!\in\!\C:\re z\!\le\!0, \, \abs{z}\!\ge\! \sqrt{d}, \, |\im z| \!\ge\! 
		\sqrt{\max\{0,\abs{\re z}^{\!\frac{2}{k} %{s}
		}\!\!-\!\abs{\re z}^2\}}\Big\}
		\vspace{-2mm}
	\end{equation}
	\vspace{-2mm}and
	\begin{align*}
		\sigma(T) \cap \R 
		\begin{cases} 
			= \emptyset & \mbox { if } k\in[0,1), \\
			\subseteq(-\infty,-\sqrt{d}] & \mbox { if }k=1, \\
			\subseteq \!\Big(\!\!-\!\infty,-\sqrt{t_0}^{k}\!+\!\sqrt{t_0^{k}\!-\!t_0 } \,\Big]  & \mbox { if } k\in (1,2),
		\end{cases}
	\end{align*}
where in the latter case $t_0=\max\big\{ \big( k(2-k) \big)^{-\frac 1{k-1}},d\big\}$.
\end{exple}

{\bf Acknowledgements.} 
{\small 
The authors gratefully acknowledge the support of the Swiss National Science Foundation (SNF)
by the grants no.\ $200021\_169104$ and $200021\_204788$.
}

\medskip

{\bf Data availability statement.}
{\small
Data sharing not applicable to this article as no datasets were generated or analysed during the current study.
}

\bibliographystyle{acm} 
\bibliography{references} 

\end{document}